\documentclass[11pt, reqno]{article}
\title{Characterizations of amenability through stochastic domination and finitary codings}
\date{\today}
\author{Gourab Ray \thanks{University of Victoria. Research supported in part by NSERC 50311-57400. Email:gourabray@uvic.ca } \and Yinon Spinka\thanks{Tel Aviv University} \thanks{University of British Columbia. Research supported in part by NSERC. Email: yinon@math.ubc.ca}}

\usepackage{amsmath,amsfonts,amssymb,amsthm,graphics,enumerate,mathtools}
\usepackage{hyperref}
\hypersetup{
    colorlinks=true,
    linkcolor=blue,
    citecolor=red,
    urlcolor=blue,
    pdfborder={0 0 0}
}    

\usepackage{graphicx,xcolor,bbm,fullpage}
\usepackage{cleveref}
  \crefname{theorem}{Theorem}{Theorems}
  \crefname{thm}{Theorem}{Theorems}
  \crefname{lemma}{Lemma}{Lemmas}
  \crefname{lem}{Lemma}{Lemmas}
  \crefname{remark}{Remark}{Remarks}
  \crefname{prop}{Proposition}{Propositions}
  \crefname{proposition}{Proposition}{Propositions}
  \crefname{notation}{Notation}{Notations}
  \crefname{claim}{Claim}{Claims}
  \crefname{observation}{Observation}{Observations}
  \crefname{defn}{Definition}{Definitions}
  \crefname{corollary}{Corollary}{Corollaries}
  \crefname{section}{Section}{Sections}
  \crefname{figure}{Figure}{Figures}
  \crefname{exercise}{Exercise}{Exercises}
    \crefname{assumption}{Assumption}{Assumptions}

\newtheorem{thm}{Theorem}[section]

\newtheorem{claim}[thm]{Claim}
\newtheorem{lemma}[thm]{Lemma}

\newtheorem{prop}[thm]{Proposition}
\newtheorem{observation}[thm]{Observation}

\newtheorem{question}[thm]{Question}

\numberwithin{equation}{section}

\theoremstyle{definition}
\newtheorem{remark}[thm]{Remark}

\def\cX{\mathcal{X}}

\def\cU{\mathcal{U}}
\def\cT{\mathcal{T}}
\def\cS{\mathcal{S}}

\def\cG{\mathcal{G}}

\def\cE{\mathcal{E}}

\def\cC{\mathcal{C}}

\def\cA{\mathcal{A}}

\def\P{\mathbb{P}}
\def\E{\mathbb{E}}

\def\R{\mathbb{R}}
\def\Z{\mathbb{Z}}
\def\N{\mathbb{N}}

\def\R{\mathbb{R}}

\def\T{\mathbb{T}}
\newcommand{\1}{\mathbf{1}}
\def\eps{\varepsilon}

\DeclareMathOperator{\dist}{dist}

\DeclareMathOperator{\essinf}{essinf}

\newcommand{\iid}{i.i.d.}
\newcommand{\inv}{\mathsf{inv}}
\newcommand{\updategood}{\bullet}
\newcommand{\updatebad}{\triangledown}

\begin{document}
\maketitle

\begin{abstract}
We establish new characterizations of amenability of graphs through two probabilistic notions: stochastic domination and finitary codings (also called finitary factors).

On the stochastic domination side, we show that the plus state of the Ising model at very low temperature stochastically dominates a high density Bernoulli percolation if and only if the underlying graph is nonamenable. This answers a question of Liggett and Steif~\cite{liggett2006stochastic}.
We prove a similar result for the ``infinite cluster process'' of Bernoulli percolation, where a site is open if it belongs to an infinite open cluster of the underlying Bernoulli percolation. This is of particular interest as this process is not monotone and does not possess any nice form of the domain Markov property.
We also prove that the plus states of the Ising model at very low temperatures are stochastically ordered if and only if the graph is nonamenable. This answers a second question of Liggett and Steif.
We further show that these stochastic domination results can be witnessed by invariant monotone couplings.

On the finitary coding side, we show that the plus state of the Ising model at very low temperature is a finitary factor of an \iid\ process if and only if the underlying graph is nonamenable (assuming it supports a phase transition). We show a similar result for the infinite cluster process of a high-density Bernoulli percolation, with the assumption that $p_u<1$ when the graph is one-ended.

A main technique is to dilute the processes using independent Bernoulli percolation, which allows us to establish a version of the so-called Holley condition for the diluted processes. We also apply a more complicated dilution mechanism using lattice gas theory in order to stochastically compare two Ising models. Along the way we develop general tools to establish invariant domination in infinite graphs when Holley's condition is satisfied. 

The finitary factor results are based on the stochastic domination results and a dynamical construction involving bounding chains, along with a new technique to analyze coupling from the past in infinite-range processes via a novel disease spreading model, which we believe is of independent interest. 
\end{abstract}

\section{Introduction}
\label{sec:intro}
Capturing information about the geometry of a space through the lens of stochastic processes is of widespread interest in the probability community, see e.g.\ \cite{AL07,dichotomy15,benjamini1999group,haggstrom2000ising,HeSc,jonasson1999random,jonasson1999amenability,kesten1959full,kesten1959symmetric,lyons2000phase,lyons2017probability} for various results of this flavor. A common theme in such results is that there is a statistical mechanics model (e.g., Bernoulli percolation, the Ising model, uniform spanning trees) which behaves in distinct and complementary fashions on graphs which resemble Euclidean geometry (e.g., the integer lattice $\Z^d$) compared to graphs which resemble hyperbolic geometry (e.g., hyperbolic tessellations, regular trees). In this article, we add to the list of results of this flavor by establishing several characterizations of amenability through questions involving stochastic domination and finitary factors of i.i.d.\ processes.

Consider a process $X=(X_v)_{v \in V}$ taking values in $\{0,1\}$ on the vertex set $V$ of a countable graph~$G$.
The process $X$ \textbf{stochastically dominates} another such process $Y$ if $\E f(X) \ge \E f(Y)$ for every bounded increasing function $f$ on $\{0,1\}^V$, or equivalently, if the two processes can be coupled so that $X \ge Y$ almost surely. Such a coupling is called a \textbf{monotone coupling} of $X$ and $Y$. Let $\nu_p$ denote the \iid\ measure on $\{0,1\}^V$ of density $p$, so that $Y \sim \nu_p$ is Bernoulli site percolation on $G$ with parameter $p$.
Let us define
\[ p(X) := \sup \big\{ p \in [0,1] : X\text{ stochastically dominates }Y \sim \nu_p \big\} ,\]
which will be a key quantity throughout this article.
We say that $X$ is \textbf{invariant} if its distribution is invariant to all automorphisms of $G$ (if any exist).
When $X$ and $Y$ are invariant processes, we say that $X$ \textbf{invariantly dominates} $Y$ if the two processes can be coupled so that the joint process $(X,Y)$ is invariant and $X \ge Y$ almost surely. Such a coupling is called an \textbf{invariant monotone coupling} of $X$ and $Y$. Let us also define
\[ p_\inv(X) := \sup \big\{ p \in [0,1] : X\text{ invariantly dominates }Y \sim \nu_p \big\} ,\]
which clearly satisfies that $p_\inv(X) \le p(X)$.

The \textbf{vertex Cheeger constant} and \textbf{edge Cheeger constant} of $G$ are defined as
\[ h(G) := \inf_S \frac{|\partial S|}{|S|} \qquad\text{and}\qquad h_e(G) := \inf_S \frac{|\partial_e S|}{|S|} ,\]
where $\partial S$ is the external vertex boundary of $S$, i.e., the set of vertices in $G$ which are not in $S$ but have a neighbor in $S$, $\partial_eS$ is the edge boundary of $S$, i.e., the set of edges with one endpoint in $S$ and the other outside $S$, $|S|$ is the cardinality of $S$, and both infimums are over non-empty finite subsets of $V$. For bounded-degree graphs, $h(G)>0$ if and only if $h_e(G)>0$. We say that a bounded-degree graph $G$ is \textbf{amenable} if $h(G)=0$, and \textbf{nonamenable} otherwise.

Our first result gives a characterization of amenability in terms of the stochastic domination properties of the infinite clusters of Bernoulli percolation.

\begin{thm}\label{thm:stochastic-dom-perc}
Let $G$ be a bounded-degree graph.
Let $\omega$ be Bernoulli (site or bond) percolation with parameter $p$ on $G$.
Let $\omega^\infty$ consist of those sites which are in infinite clusters in $\omega$.
\begin{itemize}
 \item If $G$ is amenable, then $p(\omega^\infty)=0$ for all $0 \le p < 1$.
 \item If $G$ is nonamenable, then $p_\inv(\omega^\infty) \to 1$ as $p\to1$.
\end{itemize}
\end{thm}

Our second results gives a similar characterization in terms of the plus state of the Ising model. We refer the reader to \Cref{sec:ising-dom} for relevant definitions.

\begin{thm}\label{thm:stochastic-dom-Ising}
Let $G$ be a bounded-degree graph {with no finite connected components}. Let $\mu^+_\beta$ be the plus state of the Ising model on $G$ at inverse temperature $\beta \ge 0$.
\begin{itemize}
 \item If $G$ is amenable, then $p(\mu^+_\beta) \to 0$ as $\beta\to\infty$.
 \item If $G$ is nonamenable, then $p_\inv(\mu^+_\beta) \to 1$ as $\beta \to \infty$.
\end{itemize}
\end{thm}

The assumption that $G$ has no finite connected components is not essential: If $G$ is finite and connected, then it is not hard to see that $p(\mu^+_\beta)=p_\inv(\mu^+_\beta) \to 1-2^{-1/|V(G)|}$. It then follows from the theorem that if $G$ has finite connected components,
then $p(\mu^+_\beta),p_\inv(\mu^+_\beta) \to 1-2^{-1/m}$, where $m$ is the supremum of the sizes of the amenable connected components ($m$ is infinite if the finite components are unbounded or if there is an infinite amenable component).

For nonamenable transitive graphs, the non-invariant domination results  in both theorems %(i.e., that $p(\omega^\infty)$ and $p(\mu^+_\beta)$ tend to 1)
are already new and of interest.
In particular,
the conclusion that $p(\mu^+_\beta)\to1$ as $\beta\to\infty$ for nonamenable transitive graphs answers a question of Liggett and Steif~\cite[Question~6]{liggett2006stochastic}, where this fact (and more) was established in the special case of regular trees.

We also investigate the question of whether the plus state Ising measures at different temperatures are stochastically ordered. We show that on amenable graphs, they are not, whereas on nonamenable graphs they are, at least for sufficiently low temperatures. This answers a second question of Liggett and Steif~\cite[Question~7]{liggett2006stochastic}. In the nonamenable case, we further establish invariant domination.

\begin{thm}\label{thm:stochastic-order-Ising}
Let $G$ be a bounded-degree graph.
\begin{itemize}
 \item If $G$ is amenable, then $\mu_{\beta_1}^+$ and $\mu_{\beta_2}^+$ are not stochastically comparable for any $\beta_1 \neq \beta_2$.
 \item If $G$ is nonamenable, then $\mu_{\beta_1}^+$ invariantly dominates $\mu_{\beta_2}^+$ for all $\beta_1 >\beta_2$ sufficiently large.
\end{itemize}
\end{thm}

In the special case of $\Z^d$, the fact that the plus states at different temperatures are not stochastically comparable was already known~\cite{liggett2006stochastic}. Though that proof can be extended to quasi-transitive amenable graphs, we give a different proof. Let us mention that we show the stronger statement that the plus \emph{and minus} states at different temperatures are not stochastically comparable on amenable graphs. In the nonamenable case, the lower bound required on the inverse temperature depends on $G$ only through its Cheegar constant and maximum degree. For regular trees, Liggett and Steif~\cite{liggett2006stochastic} showed that the plus states are stochastically ordered throughout the entire low-temperature regime (i.e., for all $\beta_1 > \beta_2 > \beta_c$). Perhaps surprisingly, this fails to hold in general, even when restricting to quasi-transitive connected graphs (see \cref{rem:ising-nonamen-dom}). For regular trees, we do not know whether invariant domination holds throughout the entire low-temperature regime. In the nonamenable case, we further show that stochastic domination holds already in finite volume. We also extend the domination results to the case where an external magnetic field is present, even allowing the dominated measure to have a small \emph{positive} magnetic field, while the dominating measure has a small \emph{negative} magnetic field. This shows that there is extra ``wiggle room'' in the domination $\mu_{\beta_1}^+ \ge_{st} \mu_{\beta_2}^+$ (for example, the pair $(\mu_{\beta_2}^+,\mu_{\beta_1}^+)$ is upwards and downwards movable in the sense of~\cite{broman2006refinements}; see \cref{rem:up-down-movable}).

We now restrict attention to quasi-transitive graphs. This is a more usual setting for talking about invariant processes and invariant domination.
Suppose that $X$ and $Y$ are invariant processes.
When $G$ is amenable, a fairly standard averaging argument shows that if $X$ stochastically dominates $Y$, then $X$ also invariantly dominates $Y$.
%In particular, $p_\inv(X)=p(X)$ when $G$ is amenable.
When $G$ is nonamenable, this is not necessarily the case; see Mester~\cite{mester2013invariant} for a counterexample.
We emphasize that the invariant domination results discussed thus far, nevertheless, hold in the generality of bounded-degree nonamenable graphs.

Let us turn to our results on finitary factors.
A \textbf{factor of an \iid\ process} is any process of the form $X=\varphi(Y)$, where $Y=(Y_v)_{v \in V}$ is an \iid\ process and $\varphi$ is a measurable function which commutes with automorphisms of $G$.
Such a factor is \textbf{finitary} if in order to compute the value at any given vertex $v$, one only needs to observe a finite (but random) portion of the \iid\ process. More precisely, letting $B_n(v)$ denote the ball of radius $n$ around $v$, if $(Y_u)_{u \in B_R(v)}$ determines $X_v$, for some almost surely finite stopping time $R$ with respect to the filtration generated by $((Y_u)_{u \in B_n(v)})_{n \ge 0}$. In this case we say that $X$ is a \textbf{finitary factor of an \iid\ process}.

On a quasi-transitive amenable graph, it is known \cite{van1999existence} that the plus state of the Ising model at inverse temperature $\beta$ is a finitary factor of an \iid\ process if and only if it coincides with the minus state (i.e., $\mu^+_\beta=\mu^-_\beta$, which is known to occur if and only if $\beta \le \beta_c$). On a quasi-transitive nonamenable graph, it is also known that the plus state is a finitary factor of \iid\ when it coincides with the minus state (in particular, whenever $\beta<\beta_c$), but the converse direction is open (in particular, for any $\beta>\beta_c$). Our next result shows that in contrast to the situation for amenable graphs, the plus state on a nonamenable graph is in fact a finitary factor of \iid\ for large $\beta$. This yields a characterization of amenability in terms of the finitary codability of the Ising model
(putting aside amenable graphs for which no phase transition occurs).

\begin{thm}\label{thm:ising-ffiid}
Let $G$ be a nonamenable quasi-transitive graph.
Then for all $\beta$ sufficiently large, $\mu^+_\beta$ is a finitary factor of an \iid\ process.
\end{thm}

We also show that on ``most'' nonamenable quasi-transitive graphs, the infinite cluster(s) of Bernoulli percolation is a finitary factor of an i.i.d.\ process for $p$ close to 1. This is again in contrast to the situation on quasi-transitive amenable graphs, where it is not a finitary factor of an \iid\ process, except in the degenerate situation when percolation does not occur.
We focus on site percolation here for concreteness, but also since the result for bond percolation follows by applying the result for site percolation to the line graph.
Let $\omega$ be Bernoulli site percolation of parameter~$p$.
Recall the standard notation $p_c$, which is the infimum over those $p$ for which $\omega$ has an infinite cluster almost surely. Also recall that $p_u$ denotes the infimum over those $p$ for which $\omega$ has a unique infinite cluster.
A connected nonamenable quasi-transitive graph has either infinitely many ends or a single end (see, e.g., \cite[Section 6]{Moh91}).
In the former case, $\omega$ almost surely has infinitely many infinite clusters whenever $p>p_c$, so that $p_u=1$. 
In the latter case, it is a long-standing conjecture that Bernoulli percolation with $p$ close to 1 has a unique infinite cluster, so that $p_u<1$~\cite{lyons2000phase}.

\begin{thm}\label{thm:perc-ffiid}
Let $G$ be a nonamenable quasi-transitive connected graph with either infinitely many ends or with $p_u<1$. Let $\omega$ be Bernoulli site percolation of parameter $p$ and let $\omega^\infty$ consist of those sites which are in infinite clusters in $\omega$. 
Then $\omega^\infty$ is a finitary factor of an \iid\ process for all $p$ close to 1.
\end{thm}

\subsection{Outline of proofs and perspectives}

Given two $\{0,1\}$-valued stochastic processes $X$ and $Y$, it is often extremely useful to know whether $X$ stochastically dominates $Y$, especially when one of the processes is simple and well understood. Of particular interest is the case when $Y$ is an \iid\ process (i.e., Bernoulli percolation), with one possible application being that it allows to deduce that the dominating process $X$ percolates when the density of $Y$ is above the critical probability $p_c$ for Bernoulli percolation.
A classical result with many applications of this type is that of Liggett--Schonmann--Stacey~\cite{liggett1997domination}.

A standard tool to prove that $X$ stochastically dominates an \iid\ process of density $p$ is to show that its single-site conditional probabilities are always at least $p$. In \cref{sec:st-dom} we introduce the notation $p_*(X)$ which is the optimal (largest) such $p$. It is then evident that $p(X) \ge p_*(X)$. This idea has a extension to the case when $Y$ is not an \iid\ process. This is commonly known as \emph{Holley's criterion} and is essentially the only general available tool to prove stochastic domination. For fully supported processes on a finite graph, this criterion roughly says that if the conditional distribution of $X$ at a vertex $v$ given a configuration $\xi$ outside $v$ dominates that of $Y$ given $\xi'$ outside $v$ whenever $\xi \ge \xi'$, then $X$ stochastically dominates $Y$.
One way to prove such a statement is to run joint Glauber dynamics for $X$ and $Y$, maintaining the domination throughout the dynamics, in order to obtain a monotone coupling of $X$ and $Y$ in the limit (see, e.g., \cite[Theorems 2.1 and 2.3]{grimmett2006random} or \cite[Theorem~4.8]{georgii2001random}).

For brevity, let us say that $X$ \emph{Holley dominates} $Y$ if Holley's condition is satisfied. In \cref{sec:inv-dom} we give a similar condition for processes on an infinite graph, and we introduce the notation $X \succeq_* Y$ for this. In particular, $X \succeq_* Y$ implies that $X \ge_{st} Y$. In the case that $Y$ is an \iid\ process of density $p$, the condition $X \succeq_* Y$ amounts to saying that the single-site conditional probabilities of $X$ are always at least $p$. Though the condition $X \succeq_* Y$ is not precisely the same as $X$ Holley dominating $Y$, they are very similar and for the sake of the discussion here we use the terminology of Holley domination to describe the former too.
One big drawback of Holley's criterion is that sometimes it is too rigid in the sense that the single-site distributional comparison must hold for all $\xi \ge \xi'$. For example, the plus state of the Ising model at very high inverse temperature $\beta$ does not Holley dominate a high-density \iid\ process. Indeed, by considering a vertex of degree $\Delta$ which is surrounded by minuses, we see that
\[ p_*(\mu^+_\beta) = \frac1{e^{2\Delta\beta}+1} .\]
In particular, this gives a lower bound on $p(\mu^+_\beta)$ which tends to 0 as $\beta\to\infty$.

\Cref{thm:stochastic-dom-Ising} shows that, nevertheless, on nonamenable graphs, the plus state of the Ising model at very high inverse temperature does stochastically dominate a high-density \iid\ process.
The key idea behind the proof is to slightly `dilute' the pluses of the Ising model by independently flipping each plus to a minus with a fixed probability. This gets rid of the aforementioned rigidity in Holley's criterion (at the expense, however, of making the model non-Markovian). In fact, given the diluted process, the original process is conditionally distributed as the plus state of an Ising model at the same temperature, but with a negative magnetic field (whose precise value depends on the dilution probability). The diluted process itself has a more complicated law, but it turns out that on nonamenable graphs, it actually Holley dominates a high-density \iid\ process. The fundamental reason for this is that on a nonamenable graph, a strong boundary effect can overtake a small volume effect, and so there is a non-trivial phase transition in $\beta$ even in the presence of a small magnetic field~\cite{jonasson1999amenability}.
This translates back to the original Ising measure to yield the desired stochastic domination.
Summarizing this symbolically, if $\sigma$ is a sample from $\mu^+_\beta$ and $\tau$ is an independent dilution of $\sigma$ (where the dilution probability is appropriately tuned), then
\[ p(\sigma) \ge p(\tau) \ge p_*(\tau) \to 1 \qquad\text{as }\beta \to \infty .\]

This technique can potentially be seen as an extension of Holley's criterion for stochastic domination. 
A similar dilution idea can be found in~\cite{liggett1997domination}.
We show that this approach also works for models which are not Markov random fields nor have FKG, e.g., the infinite clusters of Bernoulli percolation, though the proof that the diluted process Holley dominates a high-density \iid\ process is more involved.

We also showcase the potential for further applications of this idea in the proof of the nonamenable case of \Cref{thm:stochastic-order-Ising}, where a simple independent dilution does not work, and we implement a more sophisticated dilution mechanism which is based on the theory of lattice gases (see \Cref{sec:ising-stochastic-ordering}).
In this case, neither of the processes we are stochastically comparing is an \iid\ process, but a similar logic applies. Symbolically, if $\sigma$ is a sample from $\mu^+_{\beta_1}$ and $\sigma'$ is a sample from $\mu^+_{\beta_2}$, we show that there is a way to dilute $\sigma$ to obtain a process $\tau \le \sigma$, so that
\[ \sigma \ge_{st} \tau \succeq_* \sigma' .\]

Let us now move on to discuss the issue of invariant domination, which can be quite elusive. Schramm and Lyons asked (in an unpublished work in 1997) whether for invariant $\{0,1\}$-valued processes on a Cayley graph, stochastic domination implies invariant domination (a more general version of this was asked in~\cite[Question~2.4]{AL07}).  Mester~\cite{mester2013invariant} gave a counterexample to this on a graph which is a Cartesian product of a regular tree with a finite graph.
The elusiveness of invariant domination is further evidence by the example of the free and wired uniform spanning forests, where it is known that the free forest stochastically dominates the wired forest~\cite{benjamini2001special}, but invariant domination has only been established for certain classes of graphs~\cite{bowen2004couplings,lyons2016invariant}.

On a finite graph, it is easy to convert stochastic domination to invariant domination by an abstract averaging procedure. While this is also possible on quasi-transitive amenable graphs, for general graphs, it is useful to have a more constructive approach to invariant domination.
On a finite graph, once Holley domination is established between two invariant processes, there is a natural invariant monotone coupling: simply run the Glauber dynamics mentioned above (recall that in Glauber dynamics, one uniformly picks a vertex and updates).
On an infinite graph, however, it is not obvious that the updating procedure is well defined in general. For Markov random fields, this can easily be made to work as one can simultaneously update all vertices in some invariant independent set. In our applications, however, the processes in question are not Markovian. For example, as mentioned, the Ising model, once diluted, is no longer a Markov random field. In \Cref{sec:inv-dom}, we develop some techniques to establish invariant domination from Holley domination for general processes. In particular, we establish this for the following classes of processes. Precise definitions can be found in \Cref{sec:inv-dom}.

\begin{itemize}
\item Processes which are \textbf{decoupled by ones} in the sense that if the boundary $\partial A$ of a finite set $A$ is all ones, then given this event, the conditional law inside $A$ is independent of that outside.
For this to be useful, we also require that every finite set is almost surely `surrounded' by ones.
The Ising model with large $\beta$ and small independent dilution falls into this category. The infinite clusters $\omega^\infty$ of Bernoulli percolation (even without a dilution), however, does not fall into this category. 

\item Processes which \textbf{can be invariantly decoupled}. This is a generalization of decoupling by ones. Roughly, a process $X$ can be invariantly decoupled if there is an invariant random set $A$ independent of $X$ such the following holds. Conditioned on $A$ and on the values of $X$ outside $A$, one can partition $A$ into finite sets $\{A_i\}$ such that the restrictions of $X$ to $A_i$'s are independent.

As an example consider the infinite clusters $\omega^\infty$ of a Bernoulli percolation with parameter $p$ close to 1, but diluted by another independent Bernoulli percolation $\eta$ with a small parameter. One can check that this process is not decoupled by ones. However, if a set $A$ is surrounded by ones and we are on the event that these vertices are connected to $\infty$ in the \emph{complement} of $A$, then the law inside $A$ is independent of the outside. 
Thus, letting $A$ to be the set of all zeros of $\eta$, one can onsider the `holes' left by the ones of the diluted process which are connected to infinity. One can check that if $p$ is close to 1 and the dilution is small enough, these holes are all finite almost surely, and $A$ restricted to these holes gives the required partition, and the process is invariantly decoupled.

 \item Processes which are \textbf{monotone limits}. By monotone we mean that the single-site conditional probabilities are monotonic in the conditioning (this is closely related to the process Holley dominating itself). By a monotone limit we mean that the process can be obtained as the limit of either all 0 or all 1 boundary conditions. Examples include the plus (or minus) state of the Ising model, as well as the independently diluted plus state.
\end{itemize}

In \cref{sec:inv-dom} we show that Holley domination implies invariant domination in each of the above three settings. Specifically, in \Cref{lem:invariant-monotone-coupling2} we prove that if $X$ can be invariantly decoupled and $Y$ is a Markov random field that is also a monotone limit with finite energy then $X \succeq_* Y$ implies that $X$ invariantly dominates $Y$.
In \cref{lem:invariant-monotone-coupling-decoupledbyones} we prove that if $X$ and $Y$ are decoupled by ones and both Holley dominate high-density \iid\ processes then $X \succeq_* Y$ implies that $X$ invariantly dominates $Y$. Finally, in \cref{thm:invariant-monotone-coupling-montone} we prove that if $X$ and $Y$ are monotone limits then $X \succeq_* Y$ implies that $X$ invariantly dominates $Y$.
As an application of either of the three particular theorems, we may deduce that the plus state of the Ising model at low temperature invariantly dominates a high-density \iid\ process on a nonamenable graph. As an application of \Cref{lem:invariant-monotone-coupling2}, we deduce that the infinite cluster(s) process of a high-density Bernoulli percolation invariantly dominates a high-density \iid\ process on a nonamenable graph.
As an application of either \cref{lem:invariant-monotone-coupling-decoupledbyones} or \Cref{thm:invariant-monotone-coupling-montone}, we can deduce the plus states of the Ising model at different temperatures on a nonamenable graph are invariantly stochastically ordered.

Now we come to finitary factors of i.i.d.\ processes.

The study of finitary factors has a long history originating in ergodic theory. The notion of a finitary factor of \iid\ has also gained attention in the probability community. One reason for this is that it gives a way (at least in principle) to construct/simulate the given process via \iid\ random variables. From the perspective of the current paper, it is of particular interest to investigate and understand the connections between this notion and classical statistical mechanics models. Van den Berg and Steif~\cite{van1999existence} established the first result in this direction, showing that on $\Z^d$, the Ising model has a phase transition (i.e., the plus and minus states differ) if and only if the plus state is a finitary factor of an \iid\ process. While this relation extends to all quasi-transitive amenable graphs, we show in \cref{thm:ising-ffiid} that it breaks down in the nonamenable setting.
Finitary factors have also been shown for classes of Markov random fields~\cite{haggstrom2000propp,spinka2018finitarymrf} and monotone processes~\cite{harel2018finitary}, and in some cases also for infinite-range processes lacking monotonicity~\cite{galves2008perfect}.
See also~\cite{de2012developments,ferrari2002perfect,galves2010perfect,haggstrom1999exact} and references therein for related results in the closely related area of exact sampling (also known as perfect simulation).

In \cref{sec:finitary} we obtain some general sufficient results for a process to be a finitary factor of an i.i.d.\ process. One such result, \Cref{thm:ffiid-general}, states that if a process is decoupled by ones and Holley dominates a high-density \iid\ process, and then it is a finitary factor of an i.i.d.\ process.
In fact, we prove something more: if $X$ and $Y$ are two processes with the above properties and $X$ Holley dominates $Y$, then $X$ and $Y$ can be jointly realized as a finitary factor of iid so that $X \ge Y$ almost surely. In particular, this gives an invariant monotone coupling so that we obtain as a corollary that $X$ invariantly dominates $Y$. The proof of \Cref{thm:ffiid-general} is based on Glauber dynamics and coupling from the past and employs bounding chains in order to `detect' when coupling has occurred in coupling from the past and a novel disease spreading problem (see \Cref{sec:disease}) which is used to show that such detection eventually occurs.

As applications of \Cref{thm:ffiid-general} and our stochastic domination results, we first obtain \Cref{thm:ising-ffiid} simply because the independently diluted Ising model is decoupled by ones and Holley dominates a high-density \iid\ process (recall that this was one of the steps in the proof of \cref{thm:stochastic-dom-Ising} described above). For \Cref{thm:perc-ffiid}, we note that the infinite cluster process $\omega^\infty$ of Bernoulli percolation (whether further diluted by another independent \iid\ process or not) is not in general decoupled by ones, but does satisfy a variation of it which is reminiscent of the decoupling property of the random-cluster model. We refer the reader to \Cref{sec:perc-ffiid} for details on this, and simply point out here that certain twists in the proof of \Cref{thm:ffiid-general} are required to make this work. We also mention that the proof of \Cref{thm:perc-ffiid} requires different arguments in the infinitely ended and one-ended cases.

\bigskip\noindent\textbf{Organization.}
In \cref{sec:st-dom} we discuss the approach we use to obtain the stochastic domination results, we prove some general results about the existence of invariant monotone couplings, and we prove \cref{thm:stochastic-dom-perc,thm:stochastic-dom-Ising}. The proof of \cref{thm:stochastic-order-Ising} is based on the same general approach, but is more involved, and we dedicate \cref{sec:ising-stochastic-ordering} to this. In \cref{sec:finitary} we prove a general result about finitary factors (\cref{thm:ffiid-general}) and use it and a variant of it (\cref{thm:ffiid-general-conn}) together with the earlier stochastic domination results to prove \cref{thm:ising-ffiid,thm:perc-ffiid} in \cref{sec:ising-ffiid,sec:perc-ffiid}. We end with open problems in \cref{sec:open}.

\bigskip\noindent\textbf{Notation.}
Throughout the paper, $G$ denotes a locally finite graph on a countable vertex set~$V$.
When $G$ has bounded degree, we denote its maximum degree by $\Delta=\Delta(G)$.
The vertex and edge Cheeger constants of $G$ are denoted by $h=h(G)$ and $h_e=h_e(G)$, respectively.
For $u,v \in V$, we write $\dist(u,v)$ for the graph distance between $u$ and $v$ in $G$. We write $B_r(v) := \{ u \in V : \dist(u,v) \le r \}$ for the ball of radius $r$ around $v$, and $B^*_r(v) := B_r(v) \setminus \{v\}$ for the punctured ball. We write $N(v) := B^*_1(v)$ for the neighborhood of $v$.

\bigskip\noindent\textbf{Acknowledgements.}
We are grateful to Matan Harel for discussions in the initial stage of the project. We also thank Omer Angel and Alexandre Stauffer for helpful discussions.

\section{Stochastic domination}\label{sec:st-dom}

We start with a general discussion around stochastic domination in $\{0,1\}^V$.
A basic method for showing that a random element $X$ in $\{0,1\}^V$ stochastically dominates the product measure $\nu_p$ is to show that its single-site conditional probabilities are at least $p$. To be precise, define
\[ p_*(X) := \inf_{v \in V} \essinf \E[X_v \mid (X_u)_{u \neq v}] .\]
It is standard that $X$ stochastically dominates $\nu_{p_*(X)}$, and hence that
\[ p(X) \ge p_*(X) .\]

It is often the case that $p(X)$ is strictly larger than $p_*(X)$, and in some situations it may even occur that $p_*(X)=0$ while $p(X)$ is positive or even close to 1. For example, in the case when $|V|=2$ and $\mu$ assigns probability $p^2$ to $(0,0)$ and $1-p^2$ to $(1,1)$, it is easy to see that $p(\mu)=1-p$ and $p_*(\mu)=0$.
More natural examples of this type are given by \cref{thm:stochastic-dom-perc,thm:stochastic-dom-Ising}:
the infinite clusters $\omega^\infty$ of Bernoulli percolation is easily seen to satisfy that $p_*(\omega^\infty)=0$ (if all neighbors $u$ of some vertex $v$ have $\omega^\infty_u=0$, then it must be the case that $\omega^\infty_v=0$), while the former theorem says that it can have $p(\omega_\infty)$ arbitrarily close to 1 on a nonamenable graph.
In the Ising model one easily sees that $p_*(\mu^+_\beta) = (e^{2\Delta\beta}+1)^{-1}$, which tends to 0 as $\beta$ tends to infinity, whereas the latter theorem says that $p(\mu^+_\beta)$ can be arbitrarily close to 1 on a nonamenable graph.
In particular, the naive approach of lower bounding $p(X)$ by $p_*(X)$ cannot be used to show that $p(\mu^+_\beta) \to 1$ as $\beta \to \infty$ when $G$ is nonamenable. 

Instead of this naive approach, the approach we use in order to lower bound $p(X)$ is to dilute the ones in $X$ and then lower bound $p_*(X')$ for the diluted process $X'$. More specifically, \cref{thm:stochastic-dom-perc,thm:stochastic-dom-Ising} will be obtained by using independent dilution, while \cref{thm:stochastic-order-Ising} will use a more sophisticated dilution mechanism. Such a dilution technique is also at the heart of a well-known domination result of Liggett, Schonmann and Stacey~\cite{liggett1997domination} for finitely dependent processes and a class of processes with weak independence conditions.

Given two elements $x,y \in \{0,1\}^V$, we write $xy$ for their pointwise product, i.e., $xy$ is the element of $\{0,1\}^V$ defined by $(xy)_v := x_vy_v$ for all $v\in V$. Clearly, if $X$ and $Y$ are any two random elements in $\{0,1\}^V$, then
\[ p(X) \ge p(XY) \ge p_*(XY) .\]
In the particular case that $Y \sim \nu_p$ is independent of $X$, we may think of $XY$ as an independent dilution of the ones in $X$, and furthermore we have that
\[ p_*(XY) \ge p \cdot p_*(X \mid XY) ,\]
where for a random element $Z \in \{0,1\}^V$ (coupled with $X$), we define
\[ p_*(X \mid Z) := \inf_{v \in V} \essinf \E[X_v \mid Z] .\]
We emphasize that the definition of $p_*(X \mid Z)$ involves conditioning on all of $Z$, rather than only on $(Z_u)_{u\neq v}$, which the reader may have expected (we could have instead defined and worked with the quantity $p'_*(X \mid Z) := \inf_{v \in V} \essinf \E[X_v \mid (Z_u)_{u \neq v}]$, which satisfies $p'_*(X \mid Z) \ge p_*(X \mid Z)$ and $p'_*(X \mid X) = p_*(X)$, but $p_*(X \mid Z)$ appears more naturally in the proofs).

We now give another interpretation of $p_*(X \mid XY)$ in the case when $Y \sim \nu_p$ is independent of~$X$.
We say that $X$ \textbf{supports $\bar 1$} if $\P(X_A=1)>0$ for all finite $A \subset V$. This is clearly necessary in order for $p(X)$ to be positive.
Suppose that $X$ supports $\bar1$.
For finite sets $A,B \subset V$, let us denote by $X^{(A,B,p)}$ a random variable whose law is that of $X$ when conditioned to equal 1 on $A$ and tilted by a factor of $1-p$ for each vertex in $B$ which is 1. That is, the law of $X^{(A,B,p)}$ is determined by the condition that for every bounded measurable function $f$ on $\{0,1\}^V$,
\[ \E\left[f(X^{(A,B,p)})\right] = \frac{\E\Big[f(X) \prod_{v \in A} X_v \prod_{v \in B} (1-p)^{X_v} \Big]}{\E\Big[\prod_{v \in A} X_v \prod_{v \in B} (1-p)^{X_v} \Big]} = \frac{\E\left[f(X) \1_{\{X_A \equiv 1\}} (1-p)^{\sum_{v \in B} X_v} \right]}{\E\left[\1_{\{X_A \equiv 1\}} (1-p)^{\sum_{v \in B} X_v} \right]} .\]
When $A$ and/or $B$ are infinite, and the weak limit of $X^{(A \cap F,B \cap F,p)}$ exists as $F$ increases to $V$, we write $X^{(A,B,p)}$ for a random variable with this limiting law.
In the case when $B=A^c$, we shorten $X^{(A,A^c,p)}$ to $X^{(A,p)}$.

\begin{lemma}\label{lem:domination*}
Let $X$ be a random element in $\{0,1\}^V$ that supports $\bar 1$. Let $p \in (0,1)$ and let $Y \sim \nu_p$ be independent of $X$. Then, almost surely,
\[ X\text{ given } XY\text{ has the distribution of }X^{(\{ u \in V : (XY)_u=1 \},p)} .\]
In particular,
\[ p_*(X \mid XY) = \inf_{v \in V} \inf_{A,B \Subset V} \E X^{(A,B,p)}_v  .\]
\end{lemma}

The last expression can be written more explicitly as
\[ \inf_{\substack{v_0,\dots,v_k \in V\\u_1,\dots,u_m \in V}} \frac{\E\left[X_{v_0}\cdots X_{v_k}(1-p)^{X_{u_1}+\cdots+X_{u_m}}\right]}{\E\left[X_{v_1}\cdots X_{v_k}(1-p)^{X_{u_1}+\cdots+X_{u_m}}\right]} = \inf_{\substack{v_0,\dots,v_k \in V\\u_1,\dots,u_m \in V}} \frac{\E\left[\1_{\{X_{v_0}= \cdots =X_{v_k}=1\}}(1-p)^{X_{u_1}+\cdots+X_{u_m}}\right]}{\E\left[\1_{\{X_{v_1}= \cdots =X_{v_k}=1\}}(1-p)^{X_{u_1}+\cdots+X_{u_m}}\right]} .\]

\begin{proof}
By Levy's zero-one law, the conditional law of $X$ given $XY$ is almost surely the limit of its conditional law given $(XY)_F$ as $F$ increases to $V$ (along a fixed sequence). Thus, it suffices to show that for any finite $F \subset V$, almost surely,
\begin{equation}\label{eq:cond-distrib-give-XY}
X\text{ given } (XY)_F\text{ has the distribution of }X^{(\{ u \in F : (XY)_u=1 \},\{ u \in F : (XY)_u=0 \},p)} .
\end{equation}
Fix finite sets $F \subset S \subset V$ and consider the conditional law of $X_S$ given $(XY)_F$. By Bayes' formula, for $x \in \{0,1\}^S$ and $z \in \{0,1\}^F$ such that $\P((XY)_F=z)>0$,
\[ \P(X_S=x \mid (XY)_F=z) = \frac{\P(X_S=x) \cdot p^{\#\{u \in F : z_u=1\}} (1-p)^{\#\{u \in F : x_u=1, z_u=0\}} \1_{\{x_F \ge z\}}}{\P((XY)_F=z)} .\]
Thus, for fixed $z$, we see that $\P(X_S=x \mid (XY)_F=z)$ is proportional to
\[ \P(X_S=x) \cdot (1-p)^{\#\{u \in F : x_u=1, z_u=0\}} \1_{\{x_F \ge z\}} .\]
This shows that the conditional law of $X_S$ given $(XY)_F$ is $X^{(\{ u \in F : (XY)_u=1 \},\{ u \in F : (XY)_u=0 \},p)}_S$.

For the `in particular' part, note first that $\E[X_v \mid XY] \ge \inf_{A,B} \E X_v^{(A,B,p)}$ almost surely by~\eqref{eq:cond-distrib-give-XY}, so that $p_*(X \mid XY)$ is at least the infimum. Conversely, $\essinf \E[X_v \mid XY] \le \inf_{A,B} \E X_v^{(A,B,p)}$. Indeed, $E := \{(XY)_A \equiv 1\} \cap \{(XY)_{B \setminus A} \equiv 0\}$ has positive probability since $X$ supports $\bar1$ and $Y$ is independent of $X$, and
\[\E X_v^{(A,B,p)} = \E[X_v \mid E] = \E[ \E[X_v \mid XY] \mid E] \]
by~\eqref{eq:cond-distrib-give-XY}, so that $\E[X_v \mid XY] \le \E X_v^{(A,B,p)}$ with positive probability.
 \end{proof}

\subsection{Invariant domination}\label{sec:inv-dom}
We now address the existence of invariant monotone couplings.
We shall show that under a certain decoupling condition, an invariant process $X$ on a locally finite connected graph satisfies
\[ p_\inv(X) \ge p_*(X) .\]
This allows to employ the same dilution approach as before to lower bound $p_\inv(X)$. Namely, for an invariant process $Y$ which is invariantly coupled with $X$, and for which $XY$ satisfies the decoupling condition, we have that
\[ p_\inv(X) \ge p_\inv(XY) \ge p_*(XY) .\]

We say that a random set $A \subset V$ is invariant if its indicator $\1_A \in \{0,1\}^V$ is an invariant process.
We say that \textbf{$X$ can be invariantly decoupled} if there exists a random set $A \subset V$ such that the following holds:
\begin{itemize}
\item $A$ is independent of $X$,
\item $A$ is invariant and $\P(v\in A)>0$ for all $v \in V$,
\item There exists an $(A,X_{A^c})$-measurable partition of $A$ into finite sets $\{A_i\}_i$ such that, given $(A,X_{A^c})$, almost surely, $\{X_{A_i}\}$ are conditionally independent.
\end{itemize}
For our applications, the indicator of the random invariant set $A$ will be an i.i.d.\ process.

A simple class of processes which can be invariantly decoupled are invariant Markov random fields (all that is needed of $A$ is that it has no infinite clusters almost surely).
For our applications, we will need to allow for non-Markov random fields. 
A slighter larger class of processes which can be invariantly decoupled is given by the following notion.
We say that a $\{0,1\}$-valued process $X$ is \textbf{decoupled by ones} if for any finite set $A \subset V$ such that $\P(X_{\partial A} \equiv 1)>0$, we have that $X_A$ and $X_{A^c}$ are conditionally independent given that $X_{\partial A} \equiv 1$. 

\begin{lemma}\label{lem:decoupled-by-ones-implies-decoupling}
Let $X$ be an invariant $\{0,1\}$-valued process which is decoupled by ones. Suppose that there exists $q \in [0,1)$ such that $\{ v\in V: (XY)_v=0\}$ almost surely has no infinite clusters when $Y \sim \nu_q$ is independent of $X$. Then $X$ can be invariantly decoupled.
\end{lemma}
\begin{proof}
Let $\{C_i\}$ be the clusters of $\{ v \in V : (XY)_v=0\}$.
Let $A := \{ v \in V : Y_v = 0 \}$ and $A_i := A \cap C_i$.
Then $\{A_i\}$ is a $(A,X_{A^c})$-measurable partition of $A$ into finite sets, and since $X_{\partial C_i} \equiv 1$ and $X$ is decoupled by ones, $\{X_{A_i}\}$ are conditionally independent given $(A,X_{A^c})$. This shows that $X$ can be invariantly decoupled.
\end{proof}

For our application to the plus state of the Ising model (\cref{thm:stochastic-dom-Ising}), the notion of decoupled by ones would suffice. However, for our application to the infinite clusters of Bernoulli percolation (\cref{thm:stochastic-dom-perc}), we need a more relaxed notion and it is for this reason that we introduced the general decoupling notion above.

\begin{thm}\label{lem:invariant-monotone-coupling}
Let $G$ be a locally finite connected graph. Let $X$ be an invariant $\{0,1\}$-valued process which can be invariantly decoupled.
Then $X$ invariantly dominates $Y \sim \nu_{p_*(X)}$.
In particular, $p_\inv(X) \ge p_*(X)$.
\end{thm}

The result given in \cref{lem:invariant-monotone-coupling} applies more generally than to a process $X$ and the associated process $Y \sim \nu_{p_*(X)}$.
Given two random elements $X$ and $Y$ in $\{0,1\}^V$, write $X \succeq_* Y$ if for all $v\in V$, finite $F \subset V\setminus \{v\}$, $x,y \in \{0,1\}^F $ such that $x \ge y$ and $\P(X_F=x),\P(Y_F=y)>0$,
\begin{equation}\label{eq:XY-*dom-def}
\E[X_v \mid X_F=x] \ge \E[Y_v \mid Y_F=y] .
\end{equation}

It is not hard to see that $X \succeq_* Y$ implies that $X \ge_{st} Y$. Indeed, one can sequentially couple the processes as follows. Enumerate the vertices in $V$ in any arbitrary order $v_1,v_2,\dots$. Suppose we have monotonically coupled $(X_{v_i} )_{1 \le i \le k-1}$ and $(Y_{v_i} )_{1 \le i \le k-1}$. Now apply the above inequality to monotonically couple $X_{v_k}$ and $Y_{v_k}$ conditioned on the values of $X$ and $Y$ on $F=\{v_1,\dots,v_{k-1}\}$.

We note the similarity of the above condition to Holley's criterion (see, e.g., \cite[Theorem 2.3]{grimmett2006random} or \cite[Theorem~4.8]{georgii2001random}). We point out however that our base set $V$ can be countably infinite and our processes are not assumed to be fully supported or that their support is connected in any sense. On the other hand, our condition requires comparison for all $F \subset V \setminus \{v\}$, not just for $F=V\setminus \{v\}$, as in Holley's criterion. 

Say that $Y$ is \textbf{monotone} if $Y \succeq_* Y$.
Say that $Y$ is a \textbf{$\bar1$-limit} if it supports $\bar1$ and its distribution is the limit of $\P(Y \in \cdot \mid Y_{A_n}=1)$ whenever $A_n$ are finite and receding to infinity in the sense that $A_n$ is eventually disjoint from any fixed finite set. The notion of a $\bar0$-limit is defined similarly.
Say that $Y$ is a \textbf{monotone limit} if it is monotone and either a $\bar1$-limit or a $\bar0$-limit.
Say that $Y$ has \textbf{finite energy} if $\E[Y_v \mid (Y_u)_{u \neq v}] \in (0,1)$ almost surely for every $v \in V$.

Note that $p_*(X) \ge p$ if and only if $X \succeq_* Y$ when $Y \sim \nu_p$. Note also, trivially, that $Y \sim \nu_p$ can be invariantly decoupled and is a monotone limit Markov random field having finite energy. \cref{lem:invariant-monotone-coupling} is therefore a special case of the following.

\begin{thm}\label{lem:invariant-monotone-coupling2}
Let $G$ be a locally finite connected graph and let $X$ and $Y$ be invariant $\{0,1\}$-valued processes. Suppose that $X$ can be invariantly decoupled, $Y$ is a monotone limit Markov random field having finite energy and $X \succeq_* Y$.
Then $X$ invariantly dominates $Y$.
\end{thm}

\begin{proof}
The cases where $Y$ is a $\bar0$-limit or a $\bar1$-limit are handled slightly differently.
In both cases the desired coupling is obtained as a stationary distribution of a joint Glauber dynamics for $X$ and $Y$. We begin by proving the latter case and then explain the required changes for the former case. Thus, we assume for now that $Y$ is a $\bar1$-limit.

Let $A$ be as guaranteed by the fact that $X$ can be invariantly decoupled. We now show that $A$ can be replaced by a set $B$ consisting only of isolated vertices.
Let $B:= \{ v\in A: J_v < J_u\text{ for all }u \sim v\}$ where $\{J_v\}_v$ are independent uniform random variables in $[0,1]$, independent also of $(A,X)$. Then, almost surely, $B$ is a random invariant set consisting solely of isolated vertices, and $\P(v \in B)>0$ for all $v$.
Furthermore, letting $\{A_i\}_i$ be the guaranteed $(A,X_{A^c})$-measurable partition of $A$ we have that $\{B_i:= A_i \cap B\}_i$ is a $(A,B,X_{B^c})$-measurable partition of $B$ such that given $(A,B,X_{B^c})$, $\{X_{B_i}\}_i$ are conditionally independent.

Before describing the joint dynamics, let us describe the individual dynamics separately. We begin with the dynamics for $Y$, which is simpler to define.
For $y \in \{0,1\}^V$, define
%\[ q_v(y) := \liminf_{n \to \infty} \E\left[Y_v \mid Y_{B^*_n(v)} = y_{B^*_n(v)}\right],\]
\[ q_v(y) := \E\big[Y_v \mid Y_{B^*_1(v)} = y_{B^*_1(v)}\big] ,\]
where the conditioning has positive probability since $Y$ has finite energy.
We make the following observations.
Since $Y$ is a Markov random field, $\E[Y_v \mid (Y_u)_{u \neq v}]=q_v(Y)$ almost surely.
Since $Y$ has finite energy, $q_v(y) \in (0,1)$ for all $y$.
Since $Y$ is monotone, $q_v(y)$ is an increasing function of $y$.

The state space for the dynamics of $Y$ is $\{0,1\}^V$.
Let us first define a single-site update operation:
Given a current state $y$, an update at $v \in V$ yields a new state $y'$ as follows: $y'_u=y_u$ for all $u\in V$ other than $v$, and $y'_v$ is set to equal 0 or 1 with probabilities $1-q_v(y)$ and $q_v(y)$, respectively.
Since $\E[Y_v \mid (Y_u)_{u \neq v}]=q_v(Y)$ almost surely, this single-site update operation preserves the distribution of $Y$.
A single step of the dynamics of $Y$ is then defined as follows:
Given a current state $y$, a new state $y'$ is obtained by taking an independent copy of $(A,B)$, letting $U=\{U_v\}_{v \in V}$ be uniform $[0,1]$ random variables, independent of each other and of everything else, and conditionally on $(A,B,U)$, applying single-site updates to the vertices in $B$, in the order induced by $U$. In fact, since $B$ consists of isolated vertices and $q_v(y)$ depends on $y$ only through $y_{N(v)}$, the order in which the updates are done is irrelevant (though this will be relevant for the joint dynamics), and we might as well update all vertices in $B$ simultaneously.
This defines the single-step transitions for a Markov chain on $\{0,1\}^V$, and completes the definition of the dynamics for $Y$.

Before continuing to the dynamics for $X$, let us prove the following convergence result for the dynamics for $Y$.
Let $(Y^n)_{n=0}^\infty$ be a Markov chain as above, started from an initial state $Y^0$ which stochastically dominates $Y$ (we will later take $Y^0$ to be $X$, which stochastically dominates $Y$ by the assumption that $X \succeq_* Y$). We claim that $Y^n$ converges in distribution to $Y$ as $n\to\infty$.
Since $q_v$ is increasing, ones sees by induction that $Y^n \ge_{st} Y$ for all $n$.
Similarly, $Y^n \le_{st} \tilde Y^n$ for all $n$, where $(\tilde Y^n)_{n=0}^\infty$ is the same chain but started from $\tilde Y^0 = \bar1$.
In fact, the same reasoning shows that $Y^n \le_{st} \tilde Y^n \le_{st} \tilde Y^{F,n}$, where $(\tilde Y^{F,n})_{n=0}^\infty$ is the finite-state Markov chain one gets by starting from $\bar1$ and suppressing all updates outside $F$, thereby freezing the configuration outside $F$ to remain all ones at all times. Note that this finite-state Markov chain is ergodic and that $\P(Y \in \cdot \mid Y_{V\setminus F}=1)$ is its (unique) stationary distribution. Indeed, $\P(Y \in \cdot \mid Y_{V\setminus F}=1)$ is stationary with respect to this dynamics since for every fixed $v$ this measure is stationary with respect to the single site update. Furthermore, $0<q_v<1$ along with $\P(v \in B)>0$ ensures that the chain is irreducible and aperiodic, i.e., ergodic.
Thus, $\tilde Y^{F,n}$ converges to $\P(Y \in \cdot \mid Y_{V\setminus F}=1)$ in distribution as $n \to \infty$, so that any subsequential limit of $Y^n$ is stochastically dominated by $\P(Y \in \cdot \mid Y_{V\setminus F}=1)$. Since $Y$ is a $\bar1$-limit, the latter converges in distribution to $Y$ as $F$ increases to $V$. Thus, any subsequential limit of $Y^n$ is stochastically dominated by $Y$. Since $Y^n \ge_{st} Y$ for all $n$, the reverse domination also holds, and we conclude that $Y^n$ converges in distribution to $Y$.

Let us now turn to the dynamics for $X$, which is slightly more technical.
Let $v \in V$ and recall that $B^*_n(v)$ is the punctured ball of radius $n$ around $v$.
Let $\cX_v \subset \{0,1\}^V$ denote the set of $x \in \{0,1\}^V$ for which
\[ p_v(x) := \lim_{n \to \infty} \E\left[X_v \mid X_{B^*_n(v)} = x_{B^*_n(v)}\right] \]
exists. By L\'evy's zero-one law, we have that $X \in \cX_v$ and $\E[X_v \mid (X_u)_{u \neq v}] = p_v(X)$ almost surely.
The state space for the dynamics of $X$ is $\cX := \bigcap_{v \in V} \cX_v$ (note that $X \in \cX$ almost surely).
As before, we first define a single-site update operation:
given a current state $x$, an update at $v \in V$ yields a new state $x'$ as follows: $x'_u=x_u$ for all $u\in V$ other than $v$, and $x'_v$ is set to equal 0 or 1 with probabilities $1-p_v(x)$ and $p_v(x)$, respectively.
Since $\E[X_v \mid (X_u)_{u \neq v}]=p_v(X)$ almost surely, this single-site update operation preserves the distribution of $X$.
A single step of the dynamics of $X$ is then defined as follows:
given a current state $x$, a new state $x'$ is obtained by taking a copy of $(A,B,U)$ (as before, $(A,B)$ and $U$ are independent of each other and everything else), and conditionally on $(A,B,U)$, applying single-site updates to the vertices in $B$, in the order induced by $U$. The fact that this is well defined is not immediate and requires justification.

We now show that $x'$ is well defined for $X$-almost every $x$.
To this end, let us replace $x$ with the random state $X$ (with $(A,B,U)$ independent of this).
Let $\{B_i\}_i$ be a $(A,B,X_{A^c})$-measurable partition of $B$ into finite sets so that, given $(A,B,X_{A^c})$, $\{X_{B_i}\}_i$ are conditionally independent.
In particular, given $(A,B)$, for any $v\in B$, $p_v(X)$ is measurable with respect to $X_{B^c \cup B_v}$, where $B_v$ is the set $B_i$ containing $v$.
Thus, given $(A,B,U)$, we can apply the single-site updates separately in each $B_i$, according to the order induced by $U$, to obtain the state $X'$.
Specifically, given $(A,B,U,X_{A^c})$, we can partition $B$ into $\{L_k\}_{k=1}^\infty$, where $L_k := \{ v \in B : |\{ u \in B_v : U_u \le U_v \}|=k\}$ consists of the $k$-th largest vertex of each partition class. Then $(p_v(X))_{v \in L_1}$ is measurable with respect to $(A,B,U,X_{L_1^c})$. This means that we can update all vertices in $L_1$ simultaneously (conditionally independently) to obtain a state $\bar X^1$.
Clearly, $\bar X^1$ and $X$ agree outside of $L_1$. In particular, they agree on $B^c$, and hence, $\{B_i\}_i$ and $\{L_k\}_{k=1}^\infty$ are also measurable with respect to $(A,B,U,\bar X^1_{B^c})$.
Continuing by induction, for $k \ge 1$, we similarly have that $(p_v(\bar X^k))_{v \in L_{k+1}}$ is measurable with respect to $(A,B,U,\bar X^k_{L_{k+1}^c})$, so that we can update all vertices in $L_{k+1}$ simultaneously to obtain a state $\bar X^{k+1}$.
Finally, define $X' := \lim_{k \to \infty} \bar X^k$, which clearly exists.
Note that $X'$ has the same law as $X$.
This defines the single-step transitions for a Markov chain on $\cX$ (defined almost everywhere with respect to the law of $X$), and completes the definition of the dynamics for $X$.

We are now ready to define the joint dynamics for $X$ and $Y$.
The state space for the dynamics is $\Omega := \{ (x,y) \in \cX \times \{0,1\}^V : x \ge y \}$, which may also be seen as a subset of $\{(0,0),(1,0),(1,1)\}^V$.
The transitions $(x,y) \mapsto (x',y')$ will be such that the probability to go from $(x,y)$ to an element of $\{x'\} \times \{0,1\}^V$ will be given by the dynamics for $X$ and in particular will not depend on $y$, and similarly, the probability to go from $(x,y)$ to an element of $\cX \times \{y'\}$ will be given by the dynamics for $Y$ and in particular will not depend on $x$.
As before, we first define a single-site joint update operation:
Given a current state $z=(x,y) \in \Omega$, an update at $v \in V$ yields a new state $z'=(x',y')$ as follows: $z'_u=z_u$ for all $u\in V$ other than $v$, and $z'_v$ is set to equal $(0,0)$, $(1,0)$, $(1,1)$ with probabilities $1-p_v(x)$, $p_v(x)-q_v(y)$, $q_v(y)$, respectively. Note that this definition makes sense, since $X \succeq_* Y$ implies that $p_v(x) \ge q_v(y)$ for all $x \in \cX_v$ and $y \le x$. This single-site joint update is the unique monotone coupling of the individual single-site updates for $X$ and $Y$.
A single step of the joint dynamics is then defined as follows:
Given a current state $z$, a new state $z'$ is obtained by taking a copy of $(A,B,U)$ (independent as before), and conditionally on $(A,B,U)$, applying single-site joint updates to the vertices in $B$, in the order induced by $U$. The fact that this is well defined is shown in a similar way as for the dynamics for $X$, recalling that the order is irrelevant for the dynamics for $Y$. %\notee{revisit this after proof for $X$ is well written}
This defines the single-step transitions for a Markov chain on $\Omega$, and completes the definition of the joint dynamics for $X$ and $Y$.
Note that if $({\sf X},{\sf Y})$ is invariant (as a $\{0,1\}^2$-valued process) and ${\sf X} \sim X$, then the next state $({\sf X}',{\sf Y}')$ is also invariant.

Let $((X^n,Y^n))_{n=0}^\infty$ be a Markov chain as above (for the joint dynamics), started from the initial state $(X^0,Y^0):=(X,X)$.
Note that $(X^n)_{n=0}^\infty$ is a stationary Markov chain and that $(Y^n)_{n=0}^\infty$ is a Markov chain that converges in distribution to $Y$. Let $\pi$ be a subsequential weak limit of the law of $(X^n,Y^n)$.
Then $\pi$ is an invariant monotone coupling between $X$ and $Y$. 
This completes the proof in the case when $Y$ is a $\bar1$-limit.

In the case when $Y$ is a $\bar0$-limit, we make the following changes:
When defining the joint chain, we start from the initial state $(X^0,Y^0) := (X,\bar0)$ instead of $(X,X)$. The proof that $Y^n$ converges in distribution to $Y$ is similar, and an invariant monotone coupling is then obtained as before as a subsequential weak limit of the law of $(X^n,Y^n)$.
\end{proof}

\begin{remark}
The assumptions of \cref{lem:invariant-monotone-coupling2} are clearly not optimal.
We only use \cref{lem:invariant-monotone-coupling} in this paper, but the proof of \cref{lem:invariant-monotone-coupling2} is not much more difficult and we hope it will find use in later applications. Furthermore, we don't know if $(X^n,Y^n)$ in the proof above converges in law, but this was not needed for our purposes.
\end{remark}

We will also need the following result concerning invariant domination.

\begin{thm}\label{thm:invariant-monotone-coupling-montone}
Let $G$ be a locally finite connected graph and let $X$ and $Y$ be invariant $\{0,1\}$-valued processes. Suppose that $X$ and $Y$ are monotone $\bar1$-limits and $X \succeq_* Y$.
Then $X$ invariantly dominates $Y$.
\end{thm}
\begin{proof}[Proof sketch]
Fix a finite $F \subset V$. Using Glauber dynamics we obtain a monotone coupling $\pi_F$ of $X$ and $Y$ when conditioned to be all ones outside $F$. We claim that $\pi_F$ decreases as $F$ increases, and in particular, it converges to a monotone coupling $\pi$ of $X$ and $Y$ as $F$ increases to $V$. It follows from the fact that the limit is independent of how $F$ increases to $V$ that $\pi$ is an invariant coupling.
\end{proof}

For possible future use, we record one additional result, which follows from \cref{thm:ffiid-general0} proved in \cref{sec:finitary}. 

\begin{thm}\label{lem:invariant-monotone-coupling-decoupledbyones}
Let $G$ be a bounded-degree connected graph and let $X$ and $Y$ be invariant $\{0,1\}$-valued processes. Suppose that $X$ and $Y$ are decoupled by ones, $p_*(X),p_*(Y)>1-\frac1{3\Delta-1}$ and $X \succeq_* Y$.
Then $X$ invariantly dominates $Y$.
\end{thm}

\subsection{Plus state of Ising model -- Proof of \cref{thm:stochastic-dom-Ising}}\label{sec:ising-dom}

Recall that $G$ is a bounded-degree graph on vertex set $V$.
The Ising model with inverse temperature $\beta \ge 0$ and magnetic fields ${\bf b} \in \R^V$ in finite volume $\Lambda \subset V$ with plus boundary conditions is the probability measure $\mu_{\Lambda,\beta,{\bf b}}^+$ on $\{+1,-1\}^V$ given by 
\begin{equation}
    \mu_{\Lambda,\beta,{\bf b}}^+(\sigma) \,\propto\, \exp\left(\sum_{\substack{\{u,v\} \in E(G)\\\{u,v\} \cap \Lambda \neq \emptyset}}\beta \sigma_u\sigma_v  + \sum_{v \in \Lambda} b_v\sigma_v\right) \cdot \1_{\{\sigma_{V \setminus \Lambda} \equiv + \}},\qquad \sigma \in \{+1,-1\}^V.
\end{equation}
It is well known that the limit of $\mu_{\Lambda,\beta,{\bf b}}^+$ exists as $\Lambda$ increases to all of $V$. This infinite-volume limit, wihch we denote by $\mu^+_{\beta, \bf b}$, is sometimes called the \emph{plus state} of the Ising model at inverse temperature $\beta$ with magnetic fields $\bf b$.
If ${\bf b} \equiv b$ is constant, we denote this by $\mu^+_{\beta, b}$, and in the special case that ${\bf b} \equiv 0$, we denote it by $\mu^+_\beta$.

Define
 \[ \alpha_{\beta,b} := \inf_{v \in V} \mu^+_{\beta,b}(\sigma_v=+) .\]

\begin{thm}\label{thm:stochastic-dom-Ising2}
Let $G$ be a bounded-degree infinite connected graph.
\begin{itemize}
 \item If $G$ is amenable, then as $\beta \to \infty$,
\[ p(\mu^+_{\beta,b})
\begin{cases}
    \text{tends to 0} &\text{if $b$ tends to 0},\\
    \text{tends to 1} &\text{if $b$ tends to $\infty$},\\
    \text{is bounded away from 0 and 1} &\text{if $b$ is bounded away from 0 and $\infty$}.
\end{cases} \]
 \item If $G$ is nonamenable, then $p(\mu^+_{\beta,b}) \to 1$ whenever $\beta h_e + b \to \infty$.
\item For any $\beta \ge 0$ and $b\in\R$,
 \[ \sup_{b' \le b} (1-e^{-2(b-b')}) \alpha_{\beta,b'} \le p(\mu^+_{\beta,b}) \le 1 - \max\left\{e^{-2b-2\beta\cdot h_e} \alpha_{\beta,b} ,~ \sup_{n \ge 1} \left((\tanh \beta)^{n-1} - \alpha_{\beta,b}\right)^{\frac1n} \right\} .\]
 \end{itemize}
Furthermore, the same also holds for $p_\inv(\mu^+_{\beta,b})$.
\end{thm}

We denote by $\sigma \in \{-1,+1\}^V$ a sample from $\mu^+_{\beta,b}$.
In order to be compatible with the earlier definitions of the section, we also consider the $\{0,1\}$-valued Ising model, namely, $X \in \{0,1\}^V$ defined by $X_v := \frac12(\sigma_v+1) \in \{0,1\}$ for all $v \in V$.
Recall the definition of $X^{(A,B,p)}$ from earlier in the section.

\begin{observation}\label{obs:X-tilted-for-ising}
$X^{(A,B,p)}$ has the law of the plus state of the $\{0,1\}$-valued Ising model on $G$ at inverse temperature $\beta$, with magnetic field $b$ on $V \setminus (A\cup B)$, $\infty$ on $A$, and $b+\frac12\log(1-p)$ on $B\setminus A$.
\end{observation}

\begin{proof}[Proof of \cref{thm:stochastic-dom-Ising2}]
We first show how the third item implies the first two items in the theorem. Suppose first that $G$ is amenable.
If $b\to\infty$, then taking $b'=b/2$ in the lower bound gives that $p(\mu^+_{\beta,b}) \ge (1-e^{-b})\alpha_{\beta,b/2} \ge (1-e^{-b})\tanh(b/2) \to 1$, where the inequality $\alpha_{\beta,b/2} \ge \tanh(b/2)$ can easily be seen using the random-cluster representation. If $b$ is bounded away from 0 and $\infty$, then using that $\alpha_{\beta,b} \ge \frac12$ for $b \ge 0$ we get that $\frac12 - \frac12 e^{-2b} \le p(\mu^+_{\beta,b}) \le 1-\frac12 e^{-2b}$, so that $p(\mu^+_{\beta,b})$ is bounded away from 0 and 1. If $b \to 0$, then either $\alpha_{\beta,b}\to 1$ so that $p(\mu^+_{\beta,b}) \le 1-e^{-2b}\alpha_{\beta,b} \to 0$, or $\alpha_{\beta,b}$ is bounded away from 1 so that $p(\mu^+_{\beta,b}) \le 1 - \sup_{n \ge 1} \left((\tanh \beta)^{n-1} - \alpha_{\beta,b}\right)^{1/n} \to 0$ (if $\alpha_{\beta,b}$ fluctuates then the two cases can be applied to appropriate subsequences). The same holds for $p_\inv(\mu^+_{\beta,b})$ using the corresponding inequalities in the third item for $p_\inv(\mu^+_{\beta,b})$.

Suppose now that $G$ is nonamenable.
We show that
\begin{equation}\label{eq:ising-non-amen-p-inv-lb}
p_\inv(\mu^+_{\beta,b}) \ge 1 - 2\sqrt{e(\Delta-1)} \cdot e^{-\beta h_e-b} ,
\end{equation}
which immediately yields that $p_\inv(\mu^+_{\beta,b}) \to 1$ whenever $\beta h_e + b \to \infty$.
For this, it suffices to show that 
\begin{equation}\label{eq:ising-alpha-lb}
\sup_{b' \le b} (1-e^{-2(b-b')}) \alpha_{\beta,b'} \ge 1 - 2\sqrt{e(\Delta-1)} \cdot e^{-\beta h_e-b} .
\end{equation}
Fix $b' \le b$ (which we later optimize over).
Our goal is to lower bound $\alpha_{\beta,b'}$.
Let $\cC_v$ denote the 0-cluster of a sample from $\mu^+_{\beta,b'}$ containing a vertex $v$. Note that $\mu^+_{\beta,b'}(\sigma_v=+) = \P(\cC_v=\emptyset)$, so that our goal is to upper bound the probability that $\cC_v$ is non-empty.
A standard Peierls argument, together with the fact that there are at most $(e(\Delta-1))^n$ connected sets $U$ of size $n$ containing $v$, yields that
\[ \P(|\cC_v|=n) \le \sum_{\substack{U:~ v\in U,~ |U|=n\\U\text{ connected}}} e^{-2\beta|\partial_e U|-2b'|U|} \le \left(e(\Delta-1) e^{-2\beta h_e-2b'}\right)^n .\]
To rule out the possibility that $\cC_v$ is infinite, we simply note that the same bound holds also in any finite volume $\Lambda \subset V$, when $\cC_v$ is defined with respect to a sample from $\mu^+_{\Lambda,\beta,b'}$.
Thus, if $e^{2\beta h_e+2b'} >e(\Delta-1)$, then
\[ \alpha_{\beta,b'} \ge 1- \frac 1{\frac{e^{2\beta h_e+2b'}}{e(\Delta-1)}-1} .\]
We wish to plug in $b':=\frac12 b- \frac12 \beta h_e + \frac14 \log (e(\Delta-1))$ (which is nearly the optimal choice). Note that $b' \ge b$ if and only if $\sqrt{e(\Delta-1)} e^{-\beta h_e-b} \ge 1$. Thus, if $b' \ge b$ then~\eqref{eq:ising-alpha-lb} holds trivially. Otherwise, $b'<b$ and $e^{2\beta h_e+2b'} >e(\Delta-1)$, so we may plug in this $b'$, which yields~\eqref{eq:ising-alpha-lb}.

We now turn to the main inequality of the theorem stated in the third item of the theorem.
We start with the upper bounds on $p(\sigma)$.
Note that if $\sigma$ stochastically dominates $\nu_p$, then since $\{\sigma_U = -\}$ is a decreasing event, we have that $\P(\sigma_U \equiv -) \le (1-p)^{|U|}$. Thus,
\[ p(\sigma) \le 1 - \sup_{\emptyset \neq U \Subset V} \P(\sigma_U \equiv -)^{1/|U|} .\]
Using finite energy and FKG, we see that
\begin{align*}
\P(\sigma_U\equiv-)
 \ge \P(\sigma_U\equiv +) \cdot e^{-2\beta|\partial_e U|-2b|U|}
 \ge (\alpha_{\beta,b} e^{-2b})^{|U|} \cdot e^{-2\beta|\partial_e U|} .
\end{align*}
Hence,
\[ p(\sigma) \le 1 - \alpha_{\beta,b} \cdot e^{-2\beta \cdot h_e-2b} .\]
This gives the first upper bound.

We now show the second upper bound.
For any connected set $S \subset V$ on size $n$ and any $v \in S$, we have
\[ \P(\sigma_S\equiv -) \ge \P(\sigma_S\text{ is constant}) - \P(\sigma_v=+) .\]
Let $\gamma_n$ be the infimum of $\P(\sigma_S\text{ is constant})$ over all connected sets $S \subset V$ of size $n$.
Since $G$ is infinite and connected, we deduce that
\[ p(\sigma) \le 1 - \sup_{v \in V,\,n \ge 1} \left(\gamma_n - \P(\sigma_v=+)\right)^{\frac1n} = 1 - \sup_{n \ge 1} \left(\gamma_n - \alpha_{\beta,b}\right)^{\frac1n} .\]
To get the claimed upper bound, it remains only to show that $\gamma_n \ge (\tanh \beta)^{n-1}$. This is not hard to show via the random-cluster representation, since $\P(\sigma_S\text{ is constant}) \ge \P(S\text{ is connected in }\omega)$, where $\omega$ is a sample from the wired FK-Ising random-cluster measure with parameter $p:=1-e^{-2\beta}$ for the edges of $G$ and parameter $p_b:=1-e^{-2b}$ for the edges leading to a ghost vertex.\footnote{This is a simple consequence of the Edwards--Sokal coupling, see \cite[Section 1.4]{grimmett2006random}. The magnetic field can be interpreted as an Ising model on the graph with a ghost vertex which is connected to all the vertices in the graph by an edge, with each such edge having inverse temperature $b$.} Since $S$ is connected, there is a spanning set of edges $E$ of size $|S|-1$, and it suffices to show that $\P(\omega_E \equiv 1) \ge (\tanh \beta)^{|E|}$. Indeed, every edge has conditional probability at least $\frac{p}{p+2(1-p)}=\tanh\beta$ to be open given the state of all other edges. This shows that $\gamma_n \ge (\tanh \beta)^{n-1}$. This completes the proof of the upper bounds.

Let us now turn to the lower bound.
Fix $b'<b$ and denote $p := 1-e^{-2(b-b')} \in (0,1)$.
Note that \cref{obs:X-tilted-for-ising} and FKG imply that $X^{(A,B,p)} \ge_{st} X^{(\emptyset,p)}$, where the latter is the law of the plus state of the $\{0,1\}$-valued Ising model with magnetic field $b'$.
\cref{lem:domination*} implies that
\[ p_*(X \mid XY) = \inf_{v \in V} \E X^{(\emptyset,p)}_v = \alpha_{\beta,b'} .\]
where $Y \sim \nu_p$ is independent of $X$. Thus,
\[ p(\sigma)=p(X) \ge p(XY) \ge p_*(XY) \ge p \cdot p_*(X \mid XY) = p \cdot \alpha_{\beta,b'} ,\]
which gives the stated lower bound on $p(\sigma)$.

To show that this lower bound holds also for $p_\inv(\sigma)$, we shall use that $p_\inv(X) \ge p_\inv(XY) \ge p_*(XY)$ holds by \cref{thm:invariant-monotone-coupling-montone} once we show that $XY$ is a monotone $\bar1$-limit.
Let us show that $XY$ is monotone.
To see this, fix $v\in V$, $F \subset V\setminus \{v\}$ finite and $z,z' \in \{0,1\}^F$ such that $z \ge z'$. We need to show that $$\E[(XY)_v \mid (XY)_F=z] \ge \E[(XY)_v \mid (XY)_F=z'].$$
Since $\E[(XY)_v \mid (XY)_F=z] = p \cdot \E[X_v \mid (XY)_F=z]$, we need to show that $$\E[X_v \mid (XY)_F=z] \ge \E[X_v \mid (XY)_F=z'].$$
Using \Cref{obs:X-tilted-for-ising}, the conditional law of $X$ given $(XY)_F=z$ is $\mu^+_{\beta,{\bf b}_z}$, where ${\bf b}_z$ equals $\infty$ on $\{ u \in F : z_u=1\}$, equals $b'$ on $\{ u \in F : z_u=0 \}$, and equals $b$ outside of $F$. The desired inequality is $\mu^+_{\beta,{\bf b}_z}(\sigma_v) \ge \mu^+_{\beta,{\bf b}_{z'}}(\sigma_v)$. This follows since $\mu^+_{\beta,{\bf b}_z} \ge_{st} \mu^+_{\beta,{\bf b}_{z'}}$ by FKG.
The fact that $XY$ is a $\bar1$-limit is now a straightforward matter.
\end{proof}

\begin{proof}[Proof of \cref{thm:stochastic-dom-Ising}]
For connected $G$, the theorem follows immediately from \cref{thm:stochastic-dom-Ising2}.
For disconnected $G$, we note that $p(\sigma)$ is equal to $\inf_i p(\sigma_{C_i})$, where $\{C_i\}$ are the connected components of $G$. Similarly, $p_\inv(\sigma) = \inf_i p_\inv(\sigma_{C_i})$ since if $p_\inv(\sigma_{C_i})>p$ for all $i$, then one obtains an invariant monotone coupling of $\sigma$ and $\nu_p$ by sampling each component independently from an invariant monotone coupling (note that $\{\sigma_{C_i}\}_i$ are independent), using the same coupling for isomorphic components.
\end{proof}

\subsection{Infinite clusters of Bernoulli percolation -- Proof of \cref{thm:stochastic-dom-perc}}\label{sec:perc-dom-proofs}

Recall that $G$ is a bounded-degree graph on vertex set $V$, $\omega$ is Bernoulli (site or bond) percolation with parameter $p$ on $G$, and $\omega^\infty$ consists of those sites which are in infinite clusters in $\omega$. As in the proof of \cref{thm:stochastic-dom-Ising}, we may assume that $G$ is connected. Observe also that it suffices to consider the case of site percolation. Indeed, if $G'$ is the line graph of $G$, then $G'$ has bounded degree, $G'$ is amenable if and only if $G$ is, and bond percolation on $G$ naturally corresponds to site percolation on $G'$ (with the same $p$). Thus, the result for site percolation on $G'$ yields the result for bond percolation on $G$. We henceforth assume that $\omega$ is site percolation. Denote $X:=\omega^\infty$. 

Let us begin with a simple observation.
For any finite $U\subset V$, we have that
\[ \P(X_U \equiv 0) \ge \P(\omega_{\partial U} \equiv 0) = (1-p)^{|\partial U|} .\]
If $X$ stochastically dominates $\nu_{p'}$, then $\P(X_U \equiv 0) \le (1-p')^{|U|}$. Hence,
\[ p(X) \le 1 - (1-p)^h .\]
In particular, if $G$ is amenable, then $p(X)=0$ for all $0 \le p < 1$.

Suppose now that $G$ is nonamenable.
Fix $q \in (0,1)$ and let $Y \sim \nu_q$ be independent of $\omega$. We shall show that $XY$ can be invariantly decoupled and
\begin{equation}\label{eq:perc-dom*}
p_*(XY) \to q \qquad\text{as }p \to 1 ,
\end{equation}
which yields the theorem since $q$ was arbitrary and $p_\inv(X) \ge p_\inv(XY) \ge p_*(XY)$ by \cref{lem:invariant-monotone-coupling}. Recalling that $q \ge p_*(XY) \ge q \cdot p_*(X \mid XY)$, we see that~\eqref{eq:perc-dom*} is equivalent to
\begin{equation}\label{eq:perc-dom**}
p_*(X \mid XY) \to 1 \qquad\text{as }p \to 1 .
\end{equation}

Denote $Z:=XY$.
Our goal is to lower bound the conditional probability that a fixed vertex $v$ is in an infinite cluster of $\omega$ given $Z$. We do so by a Peierls argument.
A cutset is a minimal finite set (with respect to inclusion) of vertices $\Pi$ which separates $v$ from infinity. The interior of a cutset $\Pi$, denoted $\text{Int}(\Pi)$, is the connected component of $v$ in $V \setminus \Pi$. The exterior of $\Pi$ is $\text{Ext}(\Pi) := (\Pi \cup \text{Int}(\Pi))^c$. {Note that $\Pi=\partial \text{Int}(\Pi)$ except in the trivial case when $\Pi=\{v\}$.} Observe that if $X_v=0$ then there must exist a cutset $\Pi$ which is closed in $\omega$ (i.e., $\omega_\Pi \equiv 0$).

Let us bound the probability that a given cutset $\Pi$ is closed (conditionally on $Z$). To this end, we further condition on $\omega_{\Pi^c}$. Note that a finite component of $\omega \cap \Pi^c$ could in fact belong to an infinite component in $\omega$, {but only if some vertices in $\Pi$ are open. In particular,
$$ \P(\omega_\Pi \equiv 0 \mid Z, \omega_{\Pi^c}) = 0  $$
on the event that some vertex $u \in \text{Int}(\Pi)$ has $Z_u=1$ or that some vertex $u \in \text{Ext}(\Pi)$ with $Z_u=1$ is in a finite component in $\omega_{\Pi^c}$.}

{
Assume henceforth that we are on the complementary event, that is, that every vertex $u$ with $Z_u=1$ is in an infinite component in $\omega_{\Pi^c}$.}
For every $\eps \in \{0,1\}^\Pi$, let $\omega_{\Pi^c} \cup \eps$ be the percolation configuration whose restriction to $\Pi^c$ and $\Pi$ are $\omega_{\Pi^c}$ and $\eps$ respectively (admitting an abuse of notation).
Let $V(\eps)=V(\eps,Z,\omega_{\Pi^c})$ be the set of vertices $u$ which do not belong to an infinite component in $\omega_{\Pi^c}$ {(equivalently, $Z_u=0$ under our current assumption)}, but are in an infinite component in $\omega_{\Pi^c} \cup \eps$. Let $o(\eps)$ be the number of $1$s in $\eps$ and $c(\eps)$ the number of $0$s in $\eps$. Then 
\begin{equation}
\P(\omega_\Pi=\eps \mid Z,\omega_{\Pi^c})  = \frac{p^{o(\eps)} (1-p)^{c(\eps)} (1-q)^{|V(\eps)|}}{\sum_{\eta \in \{0,1\}^{\Pi}} p^{o(\eta)} (1-p)^{c(\eta)} (1-q)^{|V(\eta)|}}.\label{eq:conditional_law}
\end{equation}
This formula is easy to compute by taking an exhaustion of $G$, and computing the corresponding formula in the finite graphs where the infinite clusters are counted as the ones which hit the boundary, and then taking a limit.

Let $\Pi^\infty$ be the set of vertices $v \in \Pi$ which are adjacent to an infinite cluster in $\omega \cap \text{Ext}(\Pi)$. Note that $\Pi^\infty$ is measurable with respect to $\omega_{\text{Ext}(\Pi)}$ and thus also with respect to $\omega_{\Pi^c}$.
Note that for every vertex $u \in V(\eps)$, the open vertices of $\eps$ in $\Pi^\infty$ separate $u$ from infinity in $\omega_{\Pi^c} \cup \eps$. Also, $u$ could belong to the exterior of $\Pi$ as well, but $|V(\eps)|$ is finite almost surely.

We proceed to upper bound the conditional probability that $\omega_\Pi \equiv 0$. We give two different bounds whose usefulness depend on the relative size of $\Pi^\infty$. Suppose first that $|\Pi^\infty| \le \frac12 |\Pi|$. Then
\[ \P(\omega_\Pi\equiv 0 \mid Z,\omega_{\Pi^c}) \le \P(\omega_{\Pi \setminus \Pi^\infty}\equiv0 \mid Z,\omega_{\Pi^c},\omega_{\Pi^\infty} \equiv 0) = (1-p)^{|\Pi \setminus \Pi^\infty|} \le (1-p)^{\frac12 |\Pi|} ,\]
since $\omega_{\Pi^\infty} \equiv 0$ implies that $V(\omega_\Pi)=\emptyset$ so that~\eqref{eq:conditional_law} implies that $\omega_{\Pi \setminus \Pi^\infty}$ follows independent Bernoulli percolation with parameter $p$ conditionally on $Z,\omega_{\Pi^c}$ and $\{\omega_{\Pi^\infty} \equiv 0\}$.
Suppose now that $|\Pi^\infty|>\frac12 |\Pi|$. Then by~\eqref{eq:conditional_law},
\begin{align*}
\P(\omega_\Pi \equiv 0 \mid Z,\omega_{\Pi^c})
 &\le \P(\omega_{\Pi^\infty} \equiv 0 \mid Z,\omega_{\Pi^c},\omega_{\Pi \setminus \Pi^\infty} \equiv 0) \\
 & = \frac{(1-p)^{|\Pi^\infty|}}{\sum_{\eta \in \{0,1\}^{\Pi^\infty}} p^{o(\eta)} (1-p)^{c(\eta)} (1-q)^{|V(\eta')|}},
\end{align*}
where $\eta' \in \{0,1\}^\Pi$ is defined by $\eta'_{\Pi^\infty}\equiv \eta$ and $\eta'_{\Pi \setminus \Pi^\infty} \equiv 0$, and we used that $V(\eta') = \emptyset$ when $\eta \equiv 0$. Note that $V(\eta')$ is always disjoint from $\text{Ext}(\Pi)$, and hence $|V(\eta')| \le |\Pi|+|\text{Int}(\Pi)|$. Thus,
\begin{align*}
\P(\omega_\Pi \equiv 0 \mid Z,\omega_{\Pi^c})
 &\le  \frac{(1-p)^{|\Pi^\infty|} }{(1-q)^{|\Pi|+|\text{Int}(\Pi)|} + (1-(1-q)^{|\Pi|+|\text{Int}(\Pi)|}) \cdot (1-p)^{|\Pi^\infty|}} \\
 &\le (1-q)^{-|\Pi|-|\text{Int}(\Pi)|} \cdot (1-p)^{|\Pi^\infty|} \\
 &\le (1-q)^{-|\Pi|-|\text{Int}(\Pi)|} \cdot (1-p)^{\frac12 |\Pi|}.
\end{align*}

Therefore, almost surely,
\begin{align*}
\P(\omega_\Pi \equiv 0 \mid Z) &\le (1-q)^{-|\Pi|-|\text{Int}(\Pi)|} \cdot (1-p)^{\frac12 |\Pi|} \\&\le \exp\left[ \left( -(\Delta+1)\log(1-q) + \tfrac h2 \log(1-p) \right) |\text{Int}(\Pi)| \right] .
\end{align*}
Since there are at most $\Delta^{2n}$ cutsets $\Pi$ with interior of size $n$, by summing over all cutsets, we get that
\[ \P(X_v=0 \mid Z) \le \sum_{n=1}^\infty \exp\left[ \left(2\log \Delta -(\Delta+1)\log(1-q) + \tfrac h2 \log(1-p) \right) n \right] .\]
Thus,
\begin{equation}\label{eq:perc-dom-Z}
\P(X_v=1 \mid Z) \to 1 \qquad\text{as }p\to 1,\qquad\text{uniformly in }v .
\end{equation}
This establishes~\eqref{eq:perc-dom**}.

To complete the proof of \cref{thm:stochastic-dom-perc}, it remains to show that $XY$ can be invariantly decoupled.
We claim that this is the case when $p$ and $q$ are close to 1. Let $A := \{ v\in V : Y_v=0\}$. Let $(XY)^\infty$ denote the set of vertices which are infinite clusters of $\{ v \in V : (XY)_v=1 \} = \{v \in A^c : X_v=1\}$. Since $p_*(XY)$ is close to 1, a Peierls argument (see \cref{lem:inf-cluster-finite-holes} below) shows that the clusters $\{B_i\}_i$ of the complement of $(XY)^\infty$ are all finite almost surely. Thus, $\{ A \cap B_i \}_i$ is a partition of $A$ into finite sets, and it remains only to show that $\{X_{A \cap B_i}\}_i$ are conditionally independent given $(A,X_{A^c})$. For this, it suffices to show that for any fixed finite set $F \subset V$, we have that $X_F$ and $X_{F^c}$ are conditionally independent given that $X_{\partial F} \equiv 1$ and each vertex in $F$ is in an infinite cluster of $\{ v \notin F : X_v=1 \}$. Note that the latter event is measurable with respect to $\omega_{F^c}$, and that, on this event, $X_{F^c}$ is measurable with respect to $\omega_{F^c}$. Thus, $\omega_F$ is conditionally independent of $X_{F^c}$. Finally, on this event, $X_F$ is measurable with respect to $\omega_F$, which establishes the desired independence.
\qed

\begin{lemma}\label{lem:inf-cluster-finite-holes}
Let $G$ be a bounded-degree nonamenable graph. Let $\omega$ be Bernoulli site percolation with parameter $p$. Let $\omega^\infty$ consist of those sites which are in infinite clusters in $\omega$. Then, for all $p$ close to 1, the complement of $\omega^\infty$ has no infinite connected component almost surely.
\end{lemma}
\begin{proof}
Fix $v \in V$. For $n \in \N \cup \{\infty\}$, let $I_n$ be the set of sites which are in infinite clusters of $\omega \cup (V \setminus B_n(v))$. Let $A_n$ be the connected component of $v$ in $V \setminus I_n$. We need to show that $\P(|A_\infty|=\infty)=0$. Since $A_n$ increases to $A_\infty$, we need to show that $\{A_n\}_{n=1}^\infty$ is a tight collection of random variables
(the advantage of $A_n$ over $A_\infty$ is that the former is clearly almost surely finite). We do this by a Peierls argument.

Let $S \subset V$ be a finite connected set containing $v$.
Then $\P(A_n = S) \le (1-p)^{|\partial_{\text{int}} S|} \le (1-p)^{|S|h/\Delta}$, where $\partial_{\text{int}}S = \partial (S^c)$ is the internal vertex boundary of $S$. Since there are at most $\Delta^{2k}$ such sets $S$ with $|S|=k$, we get that $\P(|A_n| \ge k) \le \sum_{i=k}^\infty \Delta^{2i}(1-p)^{ih/\Delta} \le Ce^{-ck}$.
\end{proof}

We end this section with a strengthening of the second part of \cref{thm:stochastic-dom-perc}, which is most relevant in the case of nonamenable graphs with infinitely many ends. Recall that $B_r(v)$ is the ball of radius $r$ around $v$.

\begin{thm}\label{thm:stochastic-dom-perc-inf-ends}
Let $G$ be a bounded-degree nonamenable graph. Fix $r \ge 1$.
Let $\omega$ be Bernoulli site percolation with parameter $p$ on $G$.
Let $\tilde X$ consist of those vertices $v$ for which $B_r(v)$ is open in $\omega$ and every vertex $u \in \partial B_r(v)$ is either in a finite connected component of $V \setminus B_r(v)$ or is connected to infinity by an open path in $\omega$ disjoint from $B_r(v)$. Then $p_\inv(\tilde X) \to 1$ as $p\to1$.
\end{thm}
\begin{proof}
Let $G^v_1,\dots,G^v_{\ell_v}$ be the infinite connected components of $V \setminus B_r(v)$. Note that each $G^v_i$ is nonamenable.
In fact, we claim that the Cheegar constants of $G^v_i$ are bounded below uniformly over all $(v,i)$. To see this, note that if $S$ is a finite subset of $G^v_i$, then its boundary in $G^v_i$ has size at least $\max\{|\partial S| - |B_r(v)|,1\} \ge \frac1{|B_r(v)|+1}|\partial S| \ge c|\partial S|$, where $c>0$ is a constant (depending on the maximum degree in $G$). Thus, $h(G^v_i) \ge ch$.

Let $X=\omega^\infty$ consist of those sites which are in infinite open clusters in $\omega$. Let $Y \sim \nu_q$ be independent of $\omega$.
Recall from~\eqref{eq:perc-dom-Z} that $\P(v \leftrightarrow \infty \mid XY) \to 1$ as $p\to1$, uniformly in $v$. A minor modification of the argument shows that this also holds conditionally on $Z:=\tilde XY$, instead of $XY$. This argument further shows that $\P(u \leftrightarrow \infty\text{ in }G^v_i \mid Z) \to 1$ as $p \to 1$, uniformly in $(v,i)$ and $u \in G^v_i$. Thus, letting $E_v$ be the event that every vertex in $B_r(v)$ is in an infinite open cluster, and letting $E_v'$ be the event that every vertex $u \in \partial B_r(v)$ which belongs to an infinite connected component of $V \setminus B_r(v)$ is connected to infinity by an open path in $V \setminus B_r(v)$, we obtain that
\[ \P(E_v \cap E'_v \mid Z) \to 1 \qquad\text{as }p \to 1,\]
uniformly in $v$.
Let $O_v$ be the event that $B_r(v)$ is open. Noting that $\{\tilde X_v=1\}=E'_v \cap O_v$ and that $E_v \subset O_v$, we conclude that $\P(\tilde X_v=1 \mid Z) \to 1$ as $p\to 1$, uniformly in $v$. That is,
\begin{equation}\label{eq:perc-dom-Z-tilde}
 p_*(\tilde X \mid Z) \to 1 \qquad\text{as }p \to 1.
\end{equation}
As before, this shows that $p_*(Z) \to q$ as $p\to1$. Since $q$ was arbitrary and $p(\tilde X) \ge p(Z) \ge p_*(Z)$, we get that $p(\tilde X) \to 1$. To get that $p_\inv(\tilde X)\to 1$, we use that $Z$ can be invariantly decoupled (this is shown in a similar way as for $XY$ in the proof of \cref{thm:stochastic-dom-perc}) so that $p_\inv(\tilde X) \ge p_\inv(Z) \ge p_*(Z)$ by \cref{lem:invariant-monotone-coupling}.
\end{proof}

\section{Stochastic ordering of Ising measures}\label{sec:ising-stochastic-ordering}

In this section we prove \Cref{thm:stochastic-order-Ising}.
The proof is split into the amenable case and the nonamenable case, and in each case we prove slightly more than what was stated in \Cref{thm:stochastic-order-Ising} (see \cref{thm:stochastic-order-Ising1,thm:stochastic-order-Ising2}).

Recall that $\mu_\beta^\pm$ denotes the plus/minus state of the Ising model on $G$ at inverse temperature $\beta$.
Let $\mu^+_{\Lambda,\beta}$ denote the Ising model in finite volume $\Lambda \subset V$ with plus boundary conditions at inverse temperature $\beta$. It is well known that $\mu^+_{\Lambda,\beta}$ stochastically decreases to $\mu^+_\beta$ as $\Lambda$ increases to $V$.

\subsection{The amenable case}
It is a standard fact that $\mu_\beta^+ \ge_{st} \mu_\beta^-$ for any $\beta$. Therefore, the following result implies the first part of \cref{thm:stochastic-order-Ising}.

\begin{thm}\label{thm:stochastic-order-Ising1}
Let $G$ be a bounded-degree amenable graph. Then
$\mu_{\beta_1}^+ \ge_{st} \mu_{\beta_2}^-$ only if $\beta_1=\beta_2$.
\end{thm}

\begin{proof}
Suppose that $\beta_1 \neq \beta_2$.
Let $\sigma$ be sampled from $\mu^+_{\beta_1}$ and let $\tau$ be sampled from $\mu^-_{\beta_2}$.
Suppose towards a contradiction that $\sigma$ stochastically dominates $\tau$.
Let $U \subset V$ be finite.
Using finite energy and FKG, we see that
\[ \P(\sigma_U\equiv+ \mid \sigma_{\partial U}\equiv+) \ge \P(\sigma_U\equiv+) \ge \P(\tau_U\equiv+) \ge \P(\tau_U\equiv+ \mid \tau_{\partial U}\equiv+) \cdot e^{-2\Delta\beta_2|\partial_e U|} ,\]
\[ \P(\tau_U\equiv- \mid \tau_{\partial U}\equiv-) \ge \P(\tau_U\equiv-) \ge \P(\sigma_U\equiv-) \ge \P(\sigma_U\equiv- \mid \sigma_{\partial U}\equiv-) \cdot e^{-2\Delta\beta_1 |\partial_e U|} .\]
By plus-minus symmetry of the left-most and right-most probabilities, we get that
\[ e^{-2\Delta\beta_2 |\partial_e U|} \le \frac{\P(\sigma_U\equiv+ \mid \sigma_{\partial U}\equiv+)}{\P(\tau_U\equiv+ \mid \tau_{\partial U}\equiv+)} \le e^{2\Delta\beta_1 |\partial_e U|} .\]
Since $G$ is amenable, given $\eps>0$, we may choose $U$ so that $2\Delta\max\{\beta_1,\beta_2\}|\partial_e U| \le \eps|U|$, and hence,
\[ e^{-\eps|U|} \le \frac{\P(\sigma_U\equiv+ \mid \sigma_{\partial U}\equiv+)}{\P(\tau_U\equiv+ \mid \tau_{\partial U}\equiv+)} \le e^{\eps|U|} .\]
Let $U=\{u_1,\dots,u_n\}$ be an enumeration such that $u_i \sim \partial U \cup \{u_1,\dots,u_{i-1}\}$ for all $1 \le i \le n$. Note that $\P(\sigma_U\equiv+ \mid \sigma_{\partial U}\equiv+) = p_1\cdots p_n$,
where $p_i := \P(\sigma_{u_i}=+ \mid \sigma_{\partial U \cup \{u_1,\dots,u_{i-1}\}}\equiv+)$. Defining $q_i$ similarly using $\tau$, we have that
\[ e^{-\eps n} \le \frac{p_1 \cdots p_n}{q_1 \cdots q_n} \le e^{\eps n} .\]
The desired contradiction will arise once we show that if $\eps$ is chosen sufficiently small, then either $\frac{p_i}{q_i} > e^{\eps}$ for all $i$ (this occurs when $\beta_1>\beta_2$), or $\frac{q_i}{p_i} > e^{\eps}$ for all $i$ (this occurs when $\beta_2>\beta_1$). Indeed, we claim more generally that if $\beta_1>\beta_2$ then
\[ \inf_{\substack{v \in \Lambda \Subset V\\v \sim \partial\Lambda}} \frac{\P(\sigma_v=+ \mid \sigma_{\partial \Lambda}\equiv+)}{\P(\tau_v=+ \mid \tau_{\partial \Lambda}\equiv+)} > 1 .\]
Fix $\Lambda$ finite and $v\in\Lambda$ such that $v \sim \partial\Lambda$. For $i \in \{1,2\}$, let $\omega_i$ denote a sample from the FK-Ising random-cluster model on $\Lambda$, wired on $\partial \Lambda$, with parameter $p_i := 1-e^{-2\beta_i}$.
It is well known (see \cite[Theorem~1.16]{grimmett2006random}) that $\P(\sigma_v=+ \mid \sigma_{\partial \Lambda}=+)$ equals $\frac12 + \frac12\P(v \leftrightarrow \partial\Lambda\text{ in }\omega_1)$, and similarly for $\tau$. Thus, it suffices to show that $\P(v \leftrightarrow \partial\Lambda\text{ in }\omega_1) \ge \P(v \leftrightarrow \partial\Lambda\text{ in }\omega_2) + c$ for a constant $c>0$ which does not depend on $\Lambda$ or $v$.
Let $E$ be the edges incident to $v$, and let $N_i$ be the number of neighbors of $v$ which are connected to $\partial \Lambda$ in $\omega_i \setminus E$.
Conditioning on $N_i$, we may easily compute that
\[ \P(v \leftrightarrow \partial\Lambda\text{ in }\omega_i \mid N_i) = 1 - (1-p_i)^{N_i} .\]
Since $p_1>p_2$, we have that $\omega_1$ stochastically dominates $\omega_2$, and thus, $N_1$ stochastically dominates $N_2$. Noting that $N_1 \ge 1$ almost surely since $v \sim \partial\Lambda$, and that $(1-p_1)^{n_2} - (1-p_2)^{n_1} \ge c > 0$ whenever $1 \le n_1 \le \Delta$ and $n_2 \le n_1$, the desired inequality follows after taking expectations.
\end{proof}

\subsection{A dilution mechanism}\label{sec:dilution_lattice_gas}
The proof of the nonamenable case of \Cref{thm:stochastic-order-Ising} requires the use of a dilution mechanism more sophisticated than the independent dilution employed so far.
Indeed, if $\beta_1$ is only slightly large than $\beta_2$, then the magnetization under $\mu^+_{\beta_1}$ is in turn only slightly larger than under $\mu^+_{\beta_2}$. This means that the dilution that is applied to $\mu^+_{\beta_1}$ must necessarily be small (i.e., it must be unlikely to flip the spin at any particular site from plus to minus). Thus, if one was to apply an independent i.i.d.\ dilution, this would have the effect of creating a very strong negative magnetic field where the diluted field is minus (recall \cref{lem:domination*} and \cref{obs:X-tilted-for-ising}), which makes this approach fail.

To see this more explicitly, let us fix the graph to be the $\Delta$-regular tree for simplicity, and suppose that $\sigma \sim \mu^+_{\beta}$, $\sigma' \sim \mu^+_{\beta-\eps}$ and $\tau:=\min\{\sigma,\omega\}$ is a dilution of $\sigma$ by an independent $\{-1,+1\}$-valued \iid\ process $\omega$ with density $p$ of $+1$. Let us show that $\tau \succeq_* \sigma'$ does not hold for sufficiently small $\eps$ (regardless of the choice of $p$) when $\beta$ is large.
We first argue that we may assume that $p \ge 1-\eps$. For this, we shall use that $\P(\sigma'_v=+) \ge \P(\sigma_v=+) - \eps$, which can be seen from exact known formulas (see, e.g., \cite[Section~12.2]{georgii2011gibbs}). Thus, in order for $\tau$ to stochastically dominate $\sigma'$, we must have that $p \cdot \P(\sigma_v=+) = \P(\tau_v=+) \ge \P(\sigma'_v=+) \ge \P(\sigma_v=+) - \eps$, which yields the lower bound on $p$.
Let us now argue that
\[ \P(\tau_v=+ \mid (\tau_u)_{u \neq v} \equiv -) < \P(\sigma'_v=+ \mid (\sigma'_u)_{u \neq v} \equiv -) ,\]
so that $\tau \succeq_* \sigma'$ does not hold.
To see this, note that the right-hand side is
\[ \P(\sigma'_v=+ \mid (\sigma'_u)_{u \neq v} \equiv -) = \frac1{e^{2(\beta-\eps)\Delta}+1} = a+2a(1-a)\eps\Delta + O(\eps^2) ,\]
where $a:=(e^{2\beta\Delta}+1)^{-1}$.
For the left-hand side, recall from \cref{lem:domination*} and \cref{obs:X-tilted-for-ising} that the conditional law of $\sigma$ given $(\tau_u)_{u \neq v} \equiv -$ is $\mu^+_{\beta,{\bf b}}$, where ${\bf b}$ equals 0 at $v$ and equals $b:=\frac12\log(1-p)$ elsewhere. Thus, letting $\bar\sigma \sim \mu^+_{\beta,{\bf b}}$,
\[ \P(\tau_v=+ \mid (\tau_u)_{u \neq v} \equiv -) = p \cdot \P(\sigma_v=+ \mid (\tau_u)_{u \neq v} \equiv -) = p \cdot \P(\bar\sigma_v=+) \le \P(\bar\sigma_v=+) .\]
A Peierls argument gives that
\[ \P(\bar\sigma_v=+) \le a + \sum_{\substack{U:~ v\in U,~|U| \ge 2\\U\text{ connected}}} e^{-2\beta|\partial_e U|+2b(|U|-1)} \le a + \eps \Delta e^{-2\beta(2\Delta-2)} + O(\eps^2) ,\]
where we used that $e^{2b} = 1-p \le \eps$.
Since $e^{-4\beta(\Delta-1)} < 2a(1-a)$, we are done.

Instead of a simple independent \iid\ dilution, we apply a more delicate dependent dilution. The probability measure we use requires input from the theory of repulsive lattice gas which has intimate connections to the Lov\'asz local lemma, the cluster expansion, independent sets, the independent set polynomial (the hard-core model partition function) and its zeros. The core ideas appear already in a work of Shearer~\cite{shearer1985problem}, parts of which were made more explicit by Dobrushin~\cite{dobrushin1996estimates,dobrushin1996perturbation}, and the more recent \cite{scott2006dependency} (see also~\cite{scott2005repulsive}) elucidates and expands on these ideas.
One application of this set of ideas appears in~\cite{liggett1997domination} where it was used to construct finitely dependent processes which do not dominate any i.i.d.\ process (see~\cite{temmel2014shearer} for further results in this direction).
Another application is the transference of results from one of related domain to another; see~\cite{bissacot2011improvement,temmel2014shearer}. Since the Lov\'asz local lemma is a standard tool to obtain existence of combinatorial structures, this in turn leads to further applications (e.g., in~\cite{bissacot2011improvement}, to latin transversal matrices and satisfiability of $k$-SAT forms).
As far as we know, our application is rather distinct and new.
The following result is a consequence of the results in~\cite{scott2006dependency}, and will provide us with the dilution mechanism we need by applying it to a certain graph on the connected sets in $G$.

Let $\cG$ be a simple graph on a finite vertex set $\cX$.
A set $S \subseteq \cX$ is independent if it contains no two adjacent vertices. We write $S^+$ for the union of $S$ and its external vertex boundary $\partial S$. 

\begin{thm}\label{thm:dilution_measure}
Let $\{\alpha_x\}_{x \in \cX},\{r_x\}_{x \in \cX} \in [0,1)^\cX$ be such that, for any $x \in \cX$,
\begin{equation}\label{eq:sufficient}
\alpha_x \le r_x \prod_{y \sim x} (1-r_y).
\end{equation}
Then there exists a random variable ${\sf X} \subseteq \cX$ such that, for any $S \subseteq \cX$,
\begin{equation}\label{eq:dilution_measure_formula}
\P({\sf X} \supseteq S) = \prod_{x \in S} \alpha_x \cdot \1_{\{S\text{ is independent in }\cG\}} .
\end{equation}
%The distribution of $\sf X$ is unique.
Furthermore, for any independent sets $S_1,S_2 \subseteq \cX$ such that $S_1^+ \subseteq S_2^+$,
\begin{equation}\label{eq:dilution_measure_ratio}
\frac{\P({\sf X}=S_1 \mid {\sf X} \supseteq S_1)}{\P({\sf X}=S_2 \mid {\sf X} \supseteq S_2)} \ge \prod_{x \in S_2^+ \setminus S_1^+} (1-r_x) .
\end{equation}
\end{thm}

Let us give some required background on the repulsive lattice gas.
A lattice gas is a statistical mechanics model in which particles are placed on sites (allowing multiple particles at a site), with each particle carrying a site-dependent fugacity and with an interaction involving pairs of particles. When the pairwise interaction is always penalizing (or neutral), the lattice gas is said to be repulsive. When this pairwise interaction is only forbidding or neutral, the lattice gas is said to have hard-core pair interactions.
When the pairwise interaction forbids multiple particles at a site, the lattice gas is said to have hard-core self-repulsion.
We focus on repulsive lattice gas with hard-core self repulsion and hard-core pair interaction, as this is sufficient for our application.

A repulsive lattice gas with hard-core self repulsion and hard-core pair interaction is defined by a finite (simple) graph $\cG$ on vertex set $\cX$ and a collection of weights ${\bf w} := \{w_x\}_{x \in \cX} \in \R^\cX$ (it is important for our purposes that negative weights are allowed here). The edges of the graph correspond to forbidden pair interactions and a lattice gas configuration corresponds to an independent set $X' \subset \cX$ in $\cG$ whose associated weight is $\prod_{x \in X'} w_x$.
The lattice gas partition function, or independent set polynomial, is defined as
\begin{equation}
Z_\cG ({\bf w}) := \sum_{\substack{X' \subseteq \cX\\X' \text{ independent in }\cG}} \prod_{x \in X'}w_x . \label{eq:lattice_gas}
\end{equation}
For any $S \subseteq \cX$, define
\begin{equation}
Z_\cG ({\bf w};S) := \sum_{\substack{S \subseteq X' \subseteq \cX\\X' \text{ independent in }\cG}} \prod_{x \in X'}w_x . \label{eq:lattice_gas_S}
\end{equation}

We now state two results we need.
The first result is an equivalence between the positivity of the independent set polynomial in a polydisc and the existence of certain probability measure on $2^\cX$ with nice properties. This result is taken from~\cite{scott2006dependency}, but the method of proof is already implicit in~\cite{shearer1985problem}.
The second result is a sufficient condition for the positivity of the independent set polynomial in a polydisc. Such a result is implicit in~\cite{shearer1985problem} and explicit in~\cite{dobrushin1996estimates,dobrushin1996perturbation}, and the version stated here is again taken from~\cite{scott2006dependency}.

\begin{thm}[{\cite[Theorem 2.2 (b),(g)]{scott2006dependency}}]\label{thm:lattice_gas1}
Let ${\bf R} = \{R_x\}_{x \in \cX} \in \R_+^\cX$. Then $Z_\cG({\bf w})>0$ for all ${\bf w}$ with $-{\bf R} \le {\bf w} \le 0$ if and only if there exists a probability measure $P$ on $2^\cX$ such that $P(\emptyset)>0$ and such that for any $S \subseteq \cX$,
\begin{equation}
\sum_{T \supseteq S}P(T) = \prod_{x \in S} R_x \cdot \1_{\{S \text{ is independent in }\cG\}}.\label{eq:dilution_measure}
\end{equation}
This probability measure is unique and satisfies that for any $S \subseteq \cX$,
\begin{equation}\label{eq:P-formula}
P(S) = (-1)^{|S|} Z_\cG (-{\bf R};S) .
\end{equation}
\end{thm}

\begin{thm}[{\cite[Corollary 4.5]{scott2006dependency}}]\label{thm:lattice_gas2}
Let ${\bf R} = \{R_x\}_{x \in \cX}, \{r_x\}_{x \in \cX} \in [0,1)^\cX$ be such that for any $x \in \cX$,
\begin{equation}
R_x \le r_x \prod_{y \sim x } (1-r_y) .\label{eq:sufficient2}
\end{equation}
Then $Z_\cG({\bf w})>0$ for all ${\bf w}$ with $-{\bf R} \le {\bf w} \le 0$, and moreover, for any $Y,Z \subseteq \cX$, 
\begin{equation}
\frac{Z_\cG({\bf w}1_{Y \cup Z})}{Z_\cG({\bf w}1_Z)} \ge \prod_{x \in Y}(1-r_x),\label{eq:ratio}
\end{equation}
where ${\bf w}1_A$ is defined by $({\bf w}1_A)_x := w_x \cdot \1_A(x)$.
\end{thm}

\begin{proof}[Proof of \cref{thm:dilution_measure}]
Set ${\bf R}=\{R_x\}_{x \in \cX}=\{\alpha_x\}_{x \in \cX}$.
\cref{thm:lattice_gas1,thm:lattice_gas2} yield a probability distribution $P$ as in \cref{thm:lattice_gas1}. Let ${\sf X}$ be a random variable whose law is $P$. Then \eqref{eq:dilution_measure_formula} holds. % and $P$ is unique. 

It remains to check that \eqref{eq:dilution_measure_ratio} holds. Note that $I \mapsto I \setminus S$ is a bijection between the independent sets containing $S$ and the independent sets disjoint from $S^+$. Since the former are the sets contributing to $Z_\cG(-{\bf R};S)$ and the latter are the sets contributing to $Z_\cG(-{\bf R}\1_{\cX \setminus S^+})$, we see that $Z_\cG(-{\bf R};S)= \prod_{x \in S} (-R_x) \cdot Z_\cG(-{\bf R}\1_{\cX \setminus S^+})$. Thus,
\[ \frac{\P({\sf X}=S_1 \mid {\sf X} \supseteq S_1)}{\P({\sf X}=S_2 \mid {\sf X} \supseteq S_2)} = \frac{(-1)^{S_1} Z_\cG(-{\bf R};S_1) / \prod_{x \in S_1} R_x}{(-1)^{S_2} Z_\cG(-{\bf R};S_2) / \prod_{x \in S_2} R_x} = \frac{Z_\cG(-{\bf R}\1_{\cX \setminus S_1^+})}{Z_\cG(-{\bf R}\1_{\cX \setminus S_2^+})} .\]
Set $Z:=\cX \setminus S_2^+$ and $Y:=(\cX \setminus S_1^+) \setminus Z = S_2^+ \setminus S_1^+$. Since $S_1^+ \subseteq S_2^+$ by assumption, we have that $Y \cup Z = \cX \setminus S_1^+$. Thus, \eqref{eq:ratio} yields that
\[ \frac{\P({\sf X}=S_1 \mid {\sf X} \supseteq S_1)}{\P({\sf X}=S_2 \mid {\sf X} \supseteq S_2)} = \frac{Z_\cG(-{\bf R}\1_{Y\cup Z})}{Z_\cG(-{\bf R}\1_Z)} \ge \prod_{x \in Y}(1-r_x) = \prod_{x \in S_2^+\setminus S_1^+}(1-r_x) . \qedhere \]
\end{proof}

\subsection{The nonamenable case}

The bulk of the proof goes toward establishing the following stochastic domination between finite-volume Ising measures with plus boundary conditions, which already implies stochastic domination between the plus states. Invariant domination is then established in \cref{thm:stochastic-order-Ising2} below.

\begin{thm}\label{thm:stochastic-order-Ising-finite-volume}
Let $G$ be a bounded-degree nonamenable graph. Then there exists $\beta_0$ such that $\mu_{\Lambda,\beta_1}^+$ stochastically dominates $\mu_{\Lambda,\beta_2}^+$ for all $\beta_1 > \beta_2 > \beta_0$ and all finite $\Lambda \subset V$.
\end{thm}

The proof of \cref{thm:stochastic-order-Ising-finite-volume} actually shows that such a finite-volume stochastic domination result holds also in the presence of an external magnetic field. In fact, even if one applies a small \emph{negative} magnetic field to $\mu_{\Lambda,\beta_1}^+$ and a small \emph{positive} magnetic field to $\mu_{\Lambda,\beta_2}^+$, the domination still holds.
This implies that the pair $(\mu^+_{\beta_2},\mu^+_{\beta_1})$ is upwards and downwards movable in the sense of~\cite{broman2006refinements} (see \cref{rem:up-down-movable}).
A quantified statement of this is given in the following theorem. We also allow for different magnetic fields at different sites.
As it turns out, this extension of \cref{thm:stochastic-order-Ising-finite-volume} is useful for proving the invariant domination result for the plus state Ising measures with no external magnetic field.

Recall that $\mu_{\Lambda,\beta,{\bf b}}^+$ is the Ising measure in finite volume $\Lambda$ with plus boundary conditions at inverse temperature $\beta$ and with magnetic fields~${\bf b}=(b_v)_{v \in \Lambda} \in \R^\Lambda$.

\begin{thm}\label{thm:stochastic-order-Ising-finite-volume2}
Let $G$ be a bounded-degree nonamenable graph. Then $\mu_{\Lambda,\beta_1,{\bf b}_1}^+ \ge_{st} \mu_{\Lambda,\beta_2,{\bf b}_2}^+$ holds for all finite $\Lambda \subset V$ and all
$\beta_1 > \beta_2 \ge \frac{100\Delta}{h_e}$ and ${\bf b}_1,{\bf b}_2 \in \R^\Lambda$ such that ${\bf b}_2-{\bf b}_1 \le 0.99 (\beta_1-\beta_2) h_e$ and ${\bf b}_2 \ge - \beta_2 h_e + \frac12 \log \Delta+1$.
\end{thm}

Let us note that the result holds also when ${\bf b}_1,{\bf b}_2 \in (\R \cup \{\infty\})^\Lambda$, in which case the upper bound on ${\bf b}_2-{\bf b}_1$ should be interpreted as implying that if ${\bf b}_2$ is $\infty$ at a site then so is ${\bf b}_1$. This can be seen by taking a suitable limit in the magnetic fields.

\begin{proof}
Fix $\Lambda \subset V$ finite and let $X \sim \mu^+_{\Lambda,\beta_1,{\bf b}_1}$ and $X' \sim \mu^+_{\Lambda,\beta_2,{\bf b}_2}$. We have
\[ \frac{\P[X_v=+ \mid (X_u)_{u \neq v}]}{\P[X_v=- \mid (X_u)_{u \neq v}]} = e^{2\beta_1 m_v(X) +2b_1(v)} \qquad\text{and}\qquad \frac{\P[X'_v=+ \mid (X'_u)_{u \neq v}]}{\P[X'_v=- \mid (X'_u)_{u \neq v}]} = e^{2\beta_2 m_v(X')+2b_2(v)} ,\]
where $m_v(\sigma) := \sum_{u \sim v} \sigma_u$.
We will dilute $X$ to obtain a process $Z \le X$ with the property that, for any $v$, almost surely,
\begin{equation}\label{eq:Z-cond-prob-ratio-lower-bound}
\frac{\P[Z_v=+ \mid (Z_u)_{u \neq v}]}{\P[Z_v=- \mid (Z_u)_{u \neq v}]} \ge e^{2\beta_2 m_v(Z) + 2b_2(v)} .
\end{equation}
By Holley's criteria, this will show that $X \ge_{st} Z \ge_{st} X'$.

In this proof, it will be convenient to identify configurations $\sigma \in \{+,-\}^V$ with their minus set $\{ v \in V : \sigma_v = - \}$, so that, for example, $\{\sigma_v=-\}=\{v \in \sigma\}$ and $\{\sigma_v=+\}=\{v \notin \sigma\}$, and if $\sigma'$ is another such configuration then $\sigma\sigma' = \sigma \cup \sigma'$.

Denote $\eps := \beta_1-\beta_2$.
For $S \subset \Lambda$, denote $b_i(S) := \sum_{v \in S} b_i(v)$.

We define the diluted process $Z$ to be $Z := XY = X \cup Y$, for a process $Y$ which is not independent of $X$, nor an i.i.d.\ process itself. For a finite connected set $S \subset V$, denote
\[ \alpha_S := e^{-2\beta_1|\partial_e S|-2b_1(S)} (e^{2\eps|\partial_e S|-2\eps h|S|}-1) .\]
(This choice of $\alpha_S$ may seem to lack motivation at the moment, but will become more apparent later in \eqref{eq:simplification}). Observe that $|\partial_e S| \ge h|S|$ by the definition of the edge Cheegar constant, and hence $\alpha_S \ge 0$. Also $\alpha_S<1$ holds since ${\bf b}_1 \ge -\beta_1 h$. We will show that there exists a process $\bar Y$ such that for any finite connected sets $S_1,\dots,S_n \subset \Lambda$ which are at pairwise distance at least 2 from each other,
\begin{equation}\label{eq:Y-def-prop}
\P(S_1,\dots,S_n\text{ are clusters of }\bar Y) = \alpha_{S_1}\cdots\alpha_{S_n} .
\end{equation}
The existence of this measure will be shown later using \Cref{thm:dilution_measure} (which relies on the theory of lattice gases), but we assume its existence for now and finish the proof.
For $\sigma=S_1\cup\cdots\cup S_n$, let $Y^\sigma$ denote a process whose law is the conditional law $\bar Y \setminus \sigma$ given that $S_1,\dots,S_n$ are clusters of $\bar Y$.
Observe that if $\tau=T_1\cup\cdots\cup T_m$ is at distance at least 2 from $\sigma$, where $T_1,\dots,T_m \subset \Lambda$ are the clusters of $\tau$, then
\[ \P(Y^\sigma = \tau) = \alpha_{T_1}\cdots\alpha_{T_m} \cdot \P(Y^{\sigma \cup \tau}=\emptyset) .\]
Finally, choose all $Y^\sigma$ to be independent of $X$, and set $Y:=Y^X$.
Observe that every minus cluster of $Z$ is either a minus cluster of $X$ or of $Y$, but not of both.

Let $\sigma \in \{+,-\}^\Lambda$ and let $\{K_i\}_{i \in I}$ be the connected components of $\sigma$. For $J \subseteq I$, let $K_J := \bigcup_{j \in J} K_j$ (so that $\sigma=K_I$).
Then
\begin{align}
\P(Z=\sigma)
 &= \sum_{J \subseteq I} \P(Z=\sigma,Y=K_J)\nonumber \\
 &= \sum_{J \subseteq I} \P(X=\sigma \setminus K_J) \cdot \P(Y=K_J \mid X=\sigma \setminus K_J) \nonumber \\
 &= \P(X=\sigma) \cdot \sum_{J \subseteq I} e^{2\beta_1|\partial_e K_J|+2b_1(K_J)} \cdot \P(Y^{\sigma \setminus K_J}=K_J) \nonumber \\
 &= \P(X=\sigma) \cdot \P(Y^\sigma = \emptyset) \cdot \sum_{J \subseteq I} \prod_{j \in J} e^{2\beta_1|\partial_e K_j|+2b_1(K_j)} \alpha_{K_j} \nonumber \\
& =  \P(X=\sigma) \cdot \P(Y^\sigma = \emptyset) \cdot \sum_{J \subseteq I} \prod_{j \in J} (e^{2\eps |\partial_e K_j|-2 \eps h |K_j|}-1)\nonumber\\
  &= \P(X=\sigma) \cdot \P(Y^\sigma = \emptyset) \cdot e^{2\eps|\partial_e\sigma|-2\eps h|\sigma|} \label{eq:simplification} .
\end{align}

Fix $v \in \Lambda$ and $\sigma \in \{+,-\}^{\Lambda \setminus \{v\}}$. Let $\sigma^\pm \in \{+,-\}^\Lambda$ equal $\pm$ at $v$ and equal $\sigma$ elsewhere. Then
\begin{align*}
\frac{\P[Z_v=+ \mid (Z_u)_{u \neq v}=\sigma]}{\P[Z_v=- \mid (Z_u)_{u \neq v}=\sigma]} &= \frac{\P(Z=\sigma^+)}{\P(Z=\sigma^-)} \\
 &= \frac{\P(X=\sigma^+)}{\P(X=\sigma^-)} \cdot \frac{\P(Y^{\sigma^+} = \emptyset)}{\P(Y^{\sigma^-} = \emptyset)} \cdot \frac{e^{2\eps|\partial_e\sigma^+|-2\eps h|\sigma^+|}}{e^{2\eps|\partial_e\sigma^-|-2\eps h|\sigma^-|}} \\
 &= e^{2(\beta_1-\eps)m_v(\sigma) + 2\eps h + 2b_1(v)} \cdot \frac{\P(Y^{\sigma^+} = \emptyset)}{\P(Y^{\sigma^-} = \emptyset)} .
\end{align*}
Therefore, \eqref{eq:Z-cond-prob-ratio-lower-bound} is equivalent to
\begin{equation}\label{eq:Y-prob-ratio-lower-bound}
\frac{\P(Y^{\sigma^+} = \emptyset)}{\P(Y^{\sigma^-} = \emptyset)} \ge e^{-2\eps h-2b_1(v)+2b_2(v)} .
\end{equation}

Thus, the proof will be complete once we show the existence of a process $\bar Y$ satisfying~\eqref{eq:Y-def-prop} and~\eqref{eq:Y-prob-ratio-lower-bound}. To this end, we aim to apply \cref{thm:dilution_measure} to an auxiliary graph $\cG$ which will be defined shortly. With condition~\eqref{eq:sufficient} in mind, let us define
\[ r_S := \alpha_S \cdot e^{2\eps h |S|+2b_1(S)-2b_2(S)} = e^{-2\beta_1|\partial_e S|+2\eps h |S|-2b_2(S)} (e^{2\eps|\partial_e S|-2\eps h|S|}-1) \]
and check that
\begin{equation}
\alpha_S \le r_S \prod_{S' \sim S} (1-r_{S'}), \label{eq:check}
\end{equation}
where $S \sim S'$ means that $S \neq S'$ and $\dist(S,S') \le 1$.
Using that $r_S \in [0,\frac12)$ (see below) and that $1-x \ge e^{-2x}$ for $x \in [0,\frac12)$, it suffices to show that
\[ \sum_{S' \sim S} r_{S'} \le \eps h |S| +b_1(S)-b_2(S) .\]
Since $\sum_{S' \sim S} r_{S'} \le \sum_{u \in S \cup N(S)} \sum_{S' \ni u} r_{S'}$, and since ${\bf b}_2-{\bf b}_1 \le 0.99 \eps h$, it suffices to show that for any $u$,
\begin{equation}
\sum_{S' \ni u} r_{S'} \le \frac{0.01\eps h}{\Delta+1}  \label{eq:sumr}
\end{equation}
Using that $e^x-1 \le xe^x$ for all $x \ge 0$, we get that
\[ r_{S'} \le 2\eps (|\partial_e S'|-h|S'|) e^{-2\beta_2|\partial_e S'|-2b_2(S')} \le \frac\eps{e\beta_2} \cdot e^{-2(\beta_2 h |S'| + b_2(S'))} ,\]
where in the second inequality we used that $xe^{-x} \le \frac1e$ for all $x \ge 0$.
Since $\beta_2 \ge \frac{100(\Delta+1)}{eh}$, it suffices to show that $\sum_{S' \ni u} e^{-2(\beta_2 h|S'|+b_2(S'))} \le 1$.
Since there are at most $(e(\Delta-1))^n$ connected sets $S' \ni u$ of size $n$, and since $\beta_2 h+{\bf b}_2 \ge \frac12 \log(2(e-1)\Delta)$, this is easily seen to hold. This establishes~\eqref{eq:sumr}.

Define a simple graph $\cG$ whose vertex set $\cX$ is the collection of all connected subsets of $\Lambda$, and where two distinct $x,y \in \cX$ are connected by an edge in $\cG$ if and only if $\dist(x,y) \le 1$ (so that $x \sim_\cG y$ if and only if $x \sim y$ in the sense defined above). 
Let ${\sf Y} \subseteq \cX$ be a random variable as in \cref{thm:dilution_measure}.
Translating ${\sf Y}$ to a subset of $\Lambda$ in the obvious manner, we obtain a random variable $\bar Y$ satisfying~\eqref{eq:Y-def-prop}.

Let us now argue that $\bar Y$ also satisfies \eqref{eq:Y-prob-ratio-lower-bound}. Let us write $\sigma^+ = S^+_1\cup\ldots S^+_k$ and $\sigma^- = S^-_1\cup\ldots S^-_\ell$. Note that the difference between $\sigma^+$ and $\sigma^-$ occurs only near $v$ where the $-$ at $v$ for $\sigma^-$ may merge several components together.
Let $x_1,\dots,x_k \in \cX$ and $y_1,\dots,y_\ell \in \cX$ denote the vertices of $\cG$ corresponding to the $\sigma^+$ and $\sigma^-$. Let $T_1:=\{y_1,\dots,y_\ell\}$ and $T_2:=\{x_1,\dots,x_k\}$. Observe that in the graph $\cG$ we have $T_1^+ \subseteq T_2^+$; in words, every connected set which is compatible with $\sigma^-$ (i.e., at distance at least 2 from its minus clusters) is also compatible with $\sigma^+$. Furthermore, $T_2^+ \setminus T_1^+ \subseteq \{\{v\}\} \cup N_\cG(\{v\})$; in words, if a connected set is compatible with $\sigma^+$ but not with $\sigma^-$, then it is necessarily at distance at most one from $v$ (this can be interpreted in $G$ or $\cG$). Therefore, \eqref{eq:dilution_measure_ratio} yields that
\[ \frac{\P(Y^{\sigma^+} = \emptyset)}{\P(Y^{\sigma^-} = \emptyset)} = \frac{\P({\sf Y} = T_1 \mid {\sf Y} \supseteq T_1)}{\P({\sf Y} = T_2 \mid {\sf Y} \supseteq T_2)} \ge \prod_{x \in T_2^+ \setminus T_1^+} (1-r_x) \ge \prod_{S \sim \{v\}} (1-r_S) \ge \frac{\alpha_{\{v\}}}{r_{\{v\}}} = e^{-2\eps h-2b_1(v)+2b_2(v)} , \]
so that \eqref{eq:Y-prob-ratio-lower-bound} holds.
\end{proof}

\begin{remark}
The reader might wonder why we need such a sophisticated dilution mechanism coming from lattice gas theory. For example, a natural choice could have been to simply consider a Bernoulli percolation on the graph $\cG$ with a cluster $S$ being open with probability $\alpha_S$, and define $Y$ using this. However, this would complicate the corresponding calculation in \eqref{eq:simplification} and we could not make this work.
\end{remark}

\begin{remark}\label{rem:ising-nonamen-dom}
The proof of \cref{thm:stochastic-order-Ising-finite-volume} is perturbative in the inverse temperature; e.g., it needs $\beta$ to be large enough so that $\sum \alpha_S$ converges (and to a sufficiently small value), where the sum is over all finite connected sets $S$ of $G$ containing a fixed vertex.
In particular, it does not yield the theorem with $\beta_0=\beta_c$, and one may wonder whether this is an artifact of the proof. In fact, it is not, as is demonstrated by the following example.

For $i \ge 1$, let $G^{(i)}$ be the graph obtained by taking the 3-regular tree and attaching to each vertex a dangling path of length $i$. Note that each $G^{(i)}$ is quasi-transitive and nonamenable, with maximum degree 4 and the same $\beta_c$ as the 3-regular tree, and with Cheegar constant $h(G^{(i)})$ tending to 0. The disjoint union of all the $G^{(i)}$ therefore yields a bounded-degree amenable graph with the same $\beta_c$.
Thus, by \cref{thm:stochastic-order-Ising}, $\mu^+_{\beta_1}$ and $\mu^+_{\beta_2}$ are not stochastically comparable for any $\beta_1 \neq \beta_2$. This means that for any given $\beta_1>\beta_2>\beta_c$, there is \emph{some} $i$ for which the corresponding Ising measures on $G^{(i)}$ are not stochastically comparable.
We have thus found a nonamenable quasi-transitive connected graph for which $\mu_{\beta_1}^+$ and $\mu_{\beta_2}^+$ are not stochastically comparable for some $\beta_1>\beta_2>\beta_c$.
\end{remark}

\begin{remark}\label{rem:up-down-movable}
    A pair $(\lambda,\mu)$ of measures on $\{0,1\}^V$ such that $\lambda \le_{st} \mu$ is said to be downwards movable if there exists $\eps>0$ such that $\lambda \le_{st} \mu^{(-,\eps)}$, where $\mu^{(-,\eps)}$ is the distribution of $XY$ when $X \sim \mu$ and $Y \sim \nu_{1-\eps}$ are independent. Upwards movablility is defined in a similar way.
    For the Ising model, \cref{thm:stochastic-order-Ising-finite-volume2} implies that the pair $(\mu^+_{\beta_2},\mu^+_{\beta_1})$ is both downwards and upwards movable for all $\beta_1>\beta_2 \ge \frac{100\Delta}{h_e}$.
\end{remark}

We now establish the invariant domination stated in the second item of \cref{thm:stochastic-order-Ising}, while also extending it to the Ising model with a magnetic field.

\begin{thm}\label{thm:stochastic-order-Ising2}
Let $G$ be a bounded-degree nonamenable graph. There exists $\beta_0$ such that for all $\beta_1 > \beta_2 > \beta_0$ there exists $b_0>0$ such that $\mu^+_{\beta_1,b_1}$ invariantly dominates $\mu^+_{\beta_2,b_2}$ when $-b_1,b_2 \le b_0$.
\end{thm}
\begin{proof}
Let $X \sim \mu^+_{\beta_1,b_1}$ and $X' \sim \mu^+_{\beta_2,b_2}$ (viewed as $\{0,1\}$-valued).
Fix $p \in (1-\frac1\Delta,1)$ and let $Y \sim \nu_p$ be independent of $X$ and $X'$.
Denote $Z:=XY$ and $Z':=X'Y$.

Let us show that $Z \succeq_* Z'$ (when $\beta_0$ and $b_0$ are chosen suitably).
To see this, fix $v\in V$, $F \subset V\setminus \{v\}$ finite and $z,z' \in \{0,1\}^F$ such that $z \ge z'$. We need to show that $$\E[Z_v \mid Z_F=z] \ge \E[Z'_v \mid Z'_F=z'].$$
Since $\E[Z_v \mid Z_F=z] = p \cdot \E[X_v \mid Z_F=z]$, and similarly for $Z'$, we need to show that $$\E[X_v \mid Z_F=z] \ge \E[X'_v \mid Z'_F=z'].$$
Using \Cref{obs:X-tilted-for-ising}, the conditional law of $X$ given $Z_F=z$ is $\mu^+_{\beta_1,{\bf b}_1}$, where ${\bf b}$ equals $\infty$ on $\{ u \in F : z_u=1\}$, equals $b'_1:=b_1+\frac12\log(1-p)$ on $\{ u \in F : z_u=0 \}$, and equals $b_1$ outside of $F$. Similarly, $X'$ given $Z'_F=z'$ has law $\mu^+_{\beta_2,{\bf b}_2}$, with ${\bf b}_2$ defined similarly via $z'$ and $b_2$. The desired inequality is $\mu^+_{\beta_1,{\bf b}_1}(\sigma_v) \ge \mu^+_{\beta_2,{\bf b}_2}(\sigma_v)$. This follows since $\mu^+_{\beta_1,{\bf b}_1} \ge_{st} \mu^+_{\beta_2,{\bf b}_2}$ by \cref{thm:stochastic-order-Ising-finite-volume2}.

We have seen in the proof of \cref{thm:stochastic-dom-Ising2} that $Z$ is a monotone $\bar1$-limit. The same holds for $Z'$.
Thus, $Z$ invariantly dominates $Z'$ by \cref{thm:invariant-monotone-coupling-montone}.
Suppose henceforth that $Z$ and $Z'$ are coupled via an invariant monotone coupling.
It is also shown in the proof of \cref{thm:stochastic-dom-Ising2} that $p_*(Z') \to p$ as $\beta_2 h_e+b_2 \to \infty$. In particular,
if $\beta_0$ is chosen large enough and $b_0$ is bounded, then $p_*(Z')>1-\frac1\Delta$ so that $Z'$ almost surely has no infinite 0-clusters.

To obtain an invariant monotone coupling between $X$ and $X'$ we proceed as follow.
For finite $\Lambda \subset \Lambda' \subset V$, let $\pi_{\Lambda,\Lambda'}$ be a monotone coupling of $\mu^+_{\Lambda,\beta_1,b'_1}$ and $\mu^+_{\Lambda',\beta_2,b'_2}$, which exists by \cref{thm:stochastic-order-Ising-finite-volume2}. Furthermore, we choose these so that $\pi_{\gamma(\Lambda),\gamma(\Lambda')} = \pi_{\Lambda,\Lambda'} \circ \gamma^{-1}$ for all automorphisms $\gamma$ of $G$.
Given $(Z,Z')$, independently for each 0-cluster $A'$ of $Z'$, letting $A := \{ v \in A' : Z_v=0 \}$, sample $(\bar X_A,\bar X'_{A'})$ from $\pi_{A,A'}$.
This yields an invariant monotone coupling between $\bar X$ and $\bar X'$.

It remains to show that $\bar X$ and $\bar X'$ have the laws of $X$ and $X'$, respectively.
Fix a finite $F \subset V$ and $x \in \{0,1\}^F$.
Let $\{A'_i\}$ be the 0-clusters of $Z'$ interesting $F$, and define $A_i := \{v \in A'_i : Z_v = 0 \}$. Then
\[ \P(\bar X_F=x \mid Z,Z') = \1_{\{x \ge Z\}} \prod_i \pi_{A_i,A'_i}(\sigma_{A_i \cap F} = x_{A_i \cap F}) = \1_{\{x \ge Z\}} \prod_i \mu^+_{A_i,\beta_1,b'_1}(\sigma_{A_i \cap F} = x_{A_i \cap F}) .\]
Let $\{B_j\}$ be the 0-clusters of $Z$ intersecting $F$. Note that $\mu^+_{\Lambda,\beta_1,b'_1}$ decomposes into a product of measures $\mu^+_{\Lambda_\ell,\beta_1,b'_1}$, where $\{\Lambda_\ell\}$ are the connected components of $\Lambda$. Since $B:= \bigcup_j B_j = \bigcup_i A_i$, we see that $\mu^+_{B,\beta_1,b'_1} = \prod_j \mu^+_{B_j,\beta_1,b'_1} = \prod_i \mu^+_{A_i,\beta_1,b'_1}$, and hence,
\[ \P(\bar X_F=x \mid Z,Z') = \1_{\{x \ge Z\}} \mu^+_{B,\beta_1,b'_1}(\sigma_F = x_F) .\]
The right-hand side is measurable with respect to $Z$, and hence, $\bar X$ and $Z'$ are conditionally independent given $Z$. Moreover, the conditional law of $\bar X$ given $Z$ above is precisely the same as the conditional law of $X$ given $XY$. Thus, $\bar X$ has the same law as $X$. One sees that $\bar X'$ has the same law as $X'$ in a similar (slightly easier) manner.
\end{proof}

\begin{remark}
    A closer attention to parameters in the proof of \cref{thm:stochastic-order-Ising2} shows that the conclusion that $\mu^+_{\beta_1,b_1}$ invariantly dominates $\mu^+_{\beta_2,b_2}$ holds whenever
    \[ \beta_1>\beta_2 \ge \frac{100\Delta}{h_e}, \qquad b_2-b_1 \le 0.99(\beta_1-\beta_2)h_e, \qquad b_2 \ge -\beta_2 h_e + 10\log\Delta .\]
    To see this, set $p:=1-\sqrt{e(\Delta-1)} e^{b_2-\beta_2h_e}$ so that $b'_2= \frac12 b_2- \frac12 \beta_2 h_e + \frac14 \log (e(\Delta-1))$.
    From the proof of \cref{thm:stochastic-dom-Ising2} we then see that $p_*(Z') \ge 1 - 2\sqrt{e(\Delta-1)} \cdot e^{-\beta_2 h_e-b_2} > 1-\frac1\Delta$.
\end{remark}

\section{Finitary codings}\label{sec:finitary}

Recall the definitions of $p_*$ of from \cref{sec:st-dom}.
Recall that a $\{0,1\}$-valued process $X$ is \textbf{decoupled by ones} if for any finite set $A \subset V$ such that $\P(X_{\partial A} \equiv 1)>0$, we have that $X_A$ and $X_{A^c}$ are conditionally independent given that $X_{\partial A} \equiv 1$.

\begin{thm}\label{thm:ffiid-general}
Let $G$ be a connected quasi-transitive graph and let $X$ be an invariant $\{0,1\}$-valued process. Suppose that $X$ is decoupled by ones and $p_*(X)>1-\frac1{3\Delta-1}$. Then $X$ is a finitary factor of an i.i.d.\ process, and this factor has a coding radius with exponential tails.
\end{thm}

The theorem extends to processes taking values in finite sets other than $\{0,1\}$ (see \cref{sec:ffiid-finite-alphabet}).
H{\"a}ggstr{\"o}m and Steif~\cite{haggstrom2000propp} proved a similar result for invariant Markov random fields on $\Z^d$. While the extension of this to quasi-transitive graphs is immediate, the relaxation of the Markov property is less so.
For monotone processes, Harel and the second author~\cite{harel2018finitary} showed that the assumptions can be weakened to a certain uniqueness property.
Such a result is also available for infinite-range models with suitably summable interactions~\cite{galves2008perfect}. As far as we know, the above theorem is the first result for non-monotone infinite-range models which does not require such a condition.
Our result falls somewhere between those of~\cite{haggstrom2000propp,harel2018finitary} and our proof combines ideas from both, together with additional ingredients.

Our proof yields more than what is stated in \cref{thm:ffiid-general}.
One enhancement is that if $X \succeq_* Y$, where $Y$ is another process satisfying the same assumptions as $X$, then $Y$ can be expressed as a finitary factor of the same i.i.d.\ process in such a way that $X \ge Y$ almost surely. In particular, this gives an invariant monotone coupling of $X$ and $Y$. This yields \cref{lem:invariant-monotone-coupling-decoupledbyones} in the case when $G$ is quasi-transitive. To obtain the general case of the lemma, we extend the above theorem to general bounded-degree graphs; while the notion of a finitary factor is usually not discussed in such a setting, it can be made sense of (basically with the same definitions), and we proceed to do so now.

Let $G$ be a locally finite graph.
A \textbf{factor of an \iid\ process} is any process of the form $X=\varphi(Y)$, where $Y=(Y_v)_{v \in V}$ is an \iid\ process and $\varphi$ is a measurable function which commutes with automorphisms of $G$.
Such a factor is \textbf{finitary} if in order to compute the value at any given vertex $v$, one only needs to observe a finite (but random) portion of the \iid\ process, or more precisely, if $(Y_u)_{u \in B_{R_v}(v)}$ determines $X_v$, for some almost surely finite stopping time $R_v$ with respect to the filtration generated by $((Y_u)_{u \in B_n(v)})_{n \ge 0}$. In this case we say that $X$ is a \textbf{finitary factor of an \iid\ process}. We call $R_v$ a \textbf{coding radius} for $v$.
We say that a factor is \textbf{uniformly finitary} if there are coding radii $\{R_v\}_{v \in V}$ which form a tight collection of random variables.

In quasi-transitive graphs every finitary factor is uniformly finitary. In non quasi-transitive graphs, the notion of finitary factor can be rendered trivial, and the notion of uniformly finitary seems to be more natural.  For example, every process on the graph $G=\N$ is invariant and is a finitary factor of an \iid\ process (with a deterministic coding radius $R_n=n$ for each $n \in \N$), but not every such process is a uniformly finitary factor of an \iid\ process.

When extending \cref{thm:ffiid-general} to bounded-degree graphs, we obtain a uniformly finitary factor of an \iid\ process, with coding radii having uniform exponential tails, which furthermore depend only on the maximum degree of the graph and on $p_*(X)$.

\begin{thm}\label{thm:ffiid-general0}
Let $G$ be a connected bounded-degree graph. Let $X$ be an invariant $\{0,1\}$-valued process. Suppose that $X$ is decoupled by ones and $p_*(X)>1-\frac1{3\Delta-1}$. Then $X$ is a uniformly finitary factor of an i.i.d.\ process, and there is a coding radius $R_v$ for $v \in V$ satisfying
\begin{equation}\label{eq:coding-radius-bound}
\P(R_v > r) \le C_{\Delta,p_*(X)} ((3\Delta-1)(1-p_*(X)))^r \qquad\text{for all }r>0 .
\end{equation}

Moreover, if $Y$ is another process satisfying the assumptions above and $X \succeq_* Y$, then $Y$ can be expressed as a uniformly finitary factor of the same i.i.d.\ process in such a way that $X \ge Y$ almost surely. In particular, there is an invariant monotone coupling of $X$ and $Y$. In the special case when $Y \sim \nu_{p_*(X)}$, we can choose $Y$ as a 0-block factor of the i.i.d.\ process.
\end{thm}

\subsection{The construction of the finitary factor using bounding chains} \label{S:Xbar}

Informally, we wish to consider the Markov chain on $\{0,1\}^V$ in which from a current state $x \in \{0,1\}^V$ one moves to a new state $\tilde x$ by resampling each vertex in a random order (according to independent times chosen uniformly in $[0,1]$ for each vertex). While this can be made precise and shown to be well defined (at least for $X$-almost every starting state $x$; see \cref{rem:nostar-MC}), we will circumvent this by considering a ``bounding chain''. Such ideas appeared previously in the context of finitary factors for Markov random fields on $\Z^d$ in~\cite{haggstrom2000propp,spinka2018finitaryising,spinka2018finitarymrf} and in the context of perfect sampling algorithms for Markov random fields on finite graphs in~\cite{huber1998exact,haggstrom1999exact,huber2004perfect,bhandari2020improved} (see also~\cite{he2021perfect}).
Our application differs from these in that our processes are not Markov random fields. The paper~\cite{harel2018finitary} also deals with non-Markov random fields, but relies instead on a monotonicity property of the random fields, which we do not have here. The Markov chains we consider will be in discrete time although the whole procedure could be done in continuous time as well.

The bounding chain we use is defined as follows. The state space is $\{0,1,*\}^V$.
For $y,y' \in \{0,1,*\}^V$, we write
\[ y \precsim y' \qquad\text{if } y_v = y'_v \text{ whenever }y'_v\neq*.\]
This is the pointwise partial order induced by the partial order on $\{0,1,*\}$ in which $0,1 \precsim *$, but $0$ and $1$ are incomparable. The star symbol is thought of as an unknown value, and $y \precsim y'$ is thought of as meaning that $y$ is ``more specified'' than $y'$. In particular, the all star configuration $\bar*$ is the unique maximal element, while every element in $\{0,1\}^V$ is a minimal element.
Given a current state $y \in \{0,1,*\}^V$, informally, we define a new state $\tilde y$ by setting its value at a vertex $v$ to be 0 or 1 only if this value can be guaranteed to arise in the previous Markov chain when the ``tail'' of $y$ is unknown. We will also couple the transitions $y \mapsto \tilde y$ for all possible starting states $y$. Formally, we proceed as follows.

To accommodate processes $X$ which do not have finite energy, we consider the support of $X$ given by
\[ \Omega := \big\{ x \in \{0,1\}^V : \P(X_F=x_F)>0\text{ for all finite }F \subset V \big\} .\]
Define
\[ \Omega^* := \big\{ y \in \{0,1,*\}^V : \text{there exists $x\in \Omega$ such that }x \precsim y \big\} .\]

Let $u=(u_v)_{v \in V} \in [0,1]^V$ be arbitrary and let $t=(t_v)_{v \in V} \in [0,1]^V$ consist of distinct numbers.
We shall define $\tilde y=\varphi(y)=\varphi(y;u,t)$ for all $y \in \Omega^*$.
We start by defining, for any finite set $A \subset V$, a configuration $\psi_A(y)=\psi_A(y;u,t)$ which represents the state obtained after applying the updates to the vertices in $A$.
When $A$ is a singleton $\{v\}$, we define $\psi_{\{v\}}(y)=\psi_v(y)$ by 
\[ \psi_v(y)_w := \begin{cases}
 y_w &\text{if }w \neq v \\
 1 &\text{if }w = v\text{ and }u_v \le q^-_v(y) \\
 0 &\text{if }w = v\text{ and }u_v > q^+_v(y) \\
 * &\text{otherwise}
\end{cases} ,\]
where 
\[ q_v^-(y) := \inf_{\substack{x \in \Omega\\x \precsim y}} \liminf_{r\to\infty}\, \E[X_v \mid X_{B^*_r(v)} = x_{B^*_r(v)}] ,\]
and $q_v^+(y)$ is defined similarly with $\sup$ and $\limsup$ instead of $\inf$ and $\liminf$. 
When $A$ contains more than one element, we write its elements $v_1,\dots,v_{|A|}$ in the order induced by the times $t_v$, and set $\psi_A(y) := \psi_{v_{|A|}} \circ \cdots \circ \psi_{v_1}(y)$.
Let $S_A(y)$ be the configuration which equals $y$ on $A$ and is all stars outside of $A$. Define $\psi^*_A := \psi_A \circ S_A$.

\begin{figure}
\centering
\includegraphics{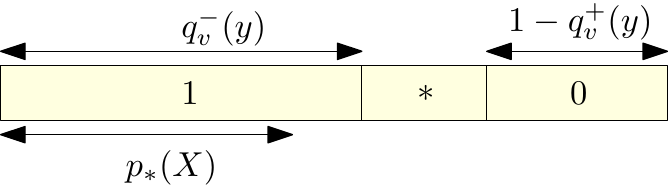}
\caption{An illustration of the dynamics $\psi_v$. From a current state $y$, an update at $v$ creates an output $*$ (which we think of as an unknown value given the current data represented by $y$) if the variable $u_v$ falls in the segment between $q^-_v(y)$ and $q^+_v(y)$.}\label{fig:psi}
\end{figure}

This construction enjoys some nice monotonicity properties with respect to $\precsim$. Note that
\[ q^-_v(y) \ge q_v^-(y') \quad\text{and}\quad q^+_v(y) \le q^+_v(y) \qquad\text{whenever }y \precsim y' .\]
Hence, $\psi_v(y) \precsim \psi_v(y')$  whenever $y \precsim y'$. Successive applications of this yield that $\psi_A(y) \precsim \psi_A(y')$ whenever $y \precsim y'$
and that
\begin{equation}\label{eq:psi*-monotone}
\psi^*_A(y) \precsim \psi^*_{A'}(y') \qquad\text{whenever }y \precsim y'\text{ and }A \supset A' .
\end{equation}
In particular,
\[ \varphi(y) := \lim_{A \uparrow V} \psi^*_A(y) \]
exists and is monotone in the sense that
\begin{equation}\label{eq:varphi-monotone-in-y}
\varphi(y) \precsim \varphi(y') \qquad\text{whenever }y \precsim y' .
\end{equation}
This gives a consistent and well-defined way to apply updates at all vertices of $V$, at the expense of perhaps creating stars in $\varphi(y)$ even when starting from a state $y$ with no stars. Let us point out that, by construction, if $u_v \le q_v^-(\bar*)$ then $\varphi(y)_v=1$ for all $y$. Similarly, if $u_v > q_v^+(\bar*)$ then $\varphi(y)_v=0$ for all $y$, though this will not be used.

Let $(U,T)$ be an \iid\ process where the value at each vertex $v \in V$ is a sequence $(U^i_v,T^i_v)_{i=1}^\infty$ of independent uniform random variables in $(0,1)$.
Denote $(U^i,T^i) := (U^i_v,T^i_v)_{v \in V}$ and
\[ \Phi^i := \varphi(\cdot ; U^i,T^i) .\]
These are i.i.d.\ random functions from $\Omega^*$ to itself.
Define a process $\bar X \in \Omega^*$ by 
\[ \bar X := \lim_{n \to \infty} \Phi^1 \circ \cdots \circ \Phi^n(\bar*) .\]
This is well defined almost surely, since $\{\Phi^1 \circ \cdots \circ \Phi^n(\bar*)\}_{n=1}^\infty$ is a $\precsim$-decreasing sequence by~\eqref{eq:varphi-monotone-in-y}. {To see this, note that \eqref{eq:varphi-monotone-in-y} implies that $\Phi^1 \circ \cdots \circ \Phi^n(y) \precsim \Phi^1 \circ \cdots \circ \Phi^n(y')$ whenever $y \precsim y'$, so that $\Phi^{n+1}(\bar *) \precsim \bar*$ implies that $\Phi^1 \circ \cdots \circ \Phi^{n+1}(\bar*) \precsim \Phi^1 \circ \cdots \circ \Phi^n(\bar*)$.} {In particular, if $\Phi^1 \circ \cdots \circ \Phi^n(\bar*)_v \in \{0,1\}$ for some $n$ and $v$, then $\Phi^1 \circ \cdots \circ \Phi^m(\bar*)_v = \bar X_v$ for all $m \ge n$.
We point out that $\{ \Phi^n \circ \cdots \circ \Phi^1(\bar*) \}_{n=1}^\infty$ is a Markov chain on $\{0,1,*\}^V$, but the order of compositions are reversed in the definition of $\bar X$; this is a form of coupling-from-the-past and is crucial for the stated monotonicity (and hence also for the almost sure limit) to hold.}

\subsection{Proof of \Cref{thm:ffiid-general0}}
Given the above construction, the main step toward establishing the theorem is to show that all stars vanish in the above limit, i.e., that $\bar X \in \{0,1\}^V$ almost surely. This is stated in the following proposition. In order to also obtain control on the coding radius, we define
\[ \Phi^i_{r,v} := \psi^*_{B_r(v)}(\cdot;U^i,T^i) .\]
An important difference between $\Phi^i$ and $\Phi^i_{r,v}$ is that in the latter, all updates are computed within the ball of radius $r$ centered around $v$. In particular, $\Phi^i$ is invariant, while $\Phi^i_{r,v}$ is not and actually always places stars outside of $B_r(v)$. Observe that $\Phi^1 \circ \cdots \circ \Phi^n(\bar*) \precsim \Phi^1_{r,v} \circ \cdots \circ \Phi^n_{r,v}(\bar*)$ by~\eqref{eq:varphi-monotone-in-y}.

\begin{prop}\label{prop:coding-radius}
There exist a constant $C=C(\Delta,p_*(X))$ such that for all $v\in V$ and $r>0$,
\[ \P\left(\Phi^1_{r,v} \circ \cdots \circ \Phi^{2r}_{r,v}(\bar*)_v = *\right) \le C((3\Delta-1)(1-p_*(X)))^r .\]
In particular, $\bar X \in \{0,1\}^V$ almost surely.
\end{prop}

We postpone the proof of \cref{prop:coding-radius} to below.

The next step towards establishing the theorem is to show that $\bar X$ is equal in distribution to $X$. %This is not hard to do using \cref{prop:coding-radius}.

\begin{prop}\label{prop:X-bar-distrib}
$\bar X$ has the same distribution as $X$.
\end{prop}
\begin{proof}
We write $X \precsim_{st} X'$ to indicate that $X'$ stochastically dominates $X$ with respect to $\precsim$ (recall that $\precsim$ is a partial order on $\{0,1,*\}^V$, even though it is not a total order on $\{0,1,*\}$).
Let us show that $X \precsim_{st} \Phi^1(X)$, where $X$ is taken to be independent of the \iid\ process $(U,T)$.
To see this, first note that $q_v^-(x) := \liminf_{r\to\infty}\, \E[X_v \mid X_{B^*_r(v)} = x_{B^*_r(v)}]$ and $q_v^+(x) := \limsup_{r\to\infty}\, \E[X_v \mid X_{B^*_r(v)} = x_{B^*_r(v)}]$ when $x \in \Omega$ and $v\in V$. Levy's zero-one law thus implies that $q^-_v(X)=q^+_v(X)=\E[X_v \mid (X_u)_{u \neq v}]$ almost surely. It is therefore easy to see that $\psi_v(X;U^1,T^1)$ has the same law as $X$. Applying this successively yields that the same is also true of $\psi_A(X;U^1,T^1)$ for any finite set $A \subset V$. Since $\psi_A(X;U^1,T^1) \precsim \psi^*_A(X;U^1,T^1)$ almost surely, we see that $X \precsim_{st} \psi^*_A(X;U^1,T^1)$.
Since $\Phi^1(X) = \lim_{A^\uparrow V} \psi^*_A(X;U^1,T^1)$ almost surely, we deduce that $X \precsim_{st} \Phi^1(X)$.

In the same manner, $X \precsim_{st} \Phi^1 \circ \cdots \circ \Phi^n(X)$. Since $\Phi^1 \circ \cdots \circ \Phi^n(X) \precsim \Phi^1 \circ \cdots \circ \Phi^n(\bar*)$, and since the latter converges to $\bar X$ almost surely, we see that $X \precsim_{st} \bar X$. Since $\bar X \in \{0,1\}^V$ almost surely, we conclude that $\bar X$ has the same law as $X$.
\end{proof}

\begin{remark}\label{rem:nostar-MC}
It follows that $\Phi^1(X)$ has the same law as $X$, and thus that $\{ \Phi^n \circ \cdots \circ \Phi^1(X) \}_{n=1}^\infty$ is a stationary Markov chain on $\{0,1\}^V$.
Indeed, letting $\Phi^0$ be a copy of $\Phi^1$ independent of $\{\Phi^i\}_{i=1}^\infty$, we have that $\Phi^0(\bar X) = \Phi^0 \circ (\lim_{n\to\infty} \Phi^1 \circ \cdots \circ \Phi^n(\bar*)) \precsim \lim_{n\to\infty} \Phi^0 \circ \cdots \circ \Phi^n(\bar*)$, so that $\Phi^0(\bar X) \precsim_{st} \bar X$. Since $\bar X$ has the same law as $X$, this shows that $\Phi^0(X) \precsim X$, from which we conclude that $\Phi^0(X)$ has the same law as $X$.
\end{remark}

We are now ready to prove the theorem.

\begin{proof}[Proof of \cref{thm:ffiid-general0}]
Since $\bar X$ is defined as a measurable equivariant function of the i.i.d.\ process $(U,T)$, this description immediately shows that $\bar X$ is a factor of an i.i.d.\ process. To see that this factor is finitary, note that
$$
\bar X_v = \lim_{r \to \infty} \Phi^1_{r,v} \circ \cdots \circ \Phi^r_{r,v}(\bar*)_v.
$$
Furthermore, almost surely, the sequence in the limit is $\precsim$-decreasing to an element in $\{0,1\}$ (by \cref{prop:coding-radius}), so that it consists of an initial run of stars (of some finite length) followed by all zeros or all ones.
This, together with the simple observation that each $\Phi^i_{v,r}$ is a block factor, implies that $\bar X$ is a finitary factor of the i.i.d.\ process $(U,T)$.
Since $X$ and $\bar X$ have the same distribution by \cref{prop:X-bar-distrib}, this shows that $X$ is a finitary factor of an i.i.d.\ process.

To see that the coding radius has exponential tails, we first note that $\P(\Phi^1_{r,v} \circ \cdots \circ \Phi^n_{r,v}(\bar*)_v = *)$ (for any $r$ and $n$) is an upper bound on the probability that the coding radius is greater than $r$. This is because $\bar X \precsim \Phi^1_{r,v} \circ \cdots \circ \Phi^n_{r,v}(\bar*)$ and because the latter only depends on $\{(U^i_u,T^i_u)\}_{u \in B_r(v), 1 \le i \le n}$.
\cref{prop:coding-radius} thus yields the claimed bound on the coding radius (which also shows that the factor is uniformly finitary). This establishes the first part of the theorem.

Before proving the moreover part, let us consider the special case when $Y \sim \nu_{p_*(X)}$. Define $\bar Y := (\1_{\{U^1_v \le p_*(X)\}})_{v \in V}$. Clearly, $\bar Y$ has law $\nu_{p_*(X)}$ and is a 0-block factor of $(U,T)$. Since $q^-_v(y) \ge p_*(X)$ for all $y$ and $v$, it is clear from the construction that $\bar X \ge \bar Y$ almost surely.

We now prove the moreover part. We thus suppose that $X'$ is an invariant $\{0,1\}$-valued process that is decoupled by ones and has $p_*(X') > 1-\frac1{3\Delta-1}$ (we use $X'$ instead of $Y$ for notation purposes).
We construct $\bar X'$ in the same manner as above, using the same i.i.d.\ process $(U,T)$ used to construct $X$, with the only difference being that $q^\pm_v$ are defined via $X'$. Let us write $q^\pm_{v;X}$ and $q^\pm_{v;X'}$ to distinguish between the two definitions (we similarly add a subscript $X$ or $X'$ to other definitions). By what we have shown above, $\bar X'$ has the same law as $X'$ and it is a uniformly finitary factor of $(U,T)$ with a similar exponential bound on the coding radius.
It remains to show that $\bar X \ge \bar X'$ almost surely.
Let us extend the usual pointwise partial order $\ge$ on $\{0,1\}^V$ to the pointwise partial order on $\{0,1,*\}^V$ induced by the total order in which $1 \ge * \ge 0$.
We claim that $X \succeq_* X'$ implies that $q^\pm_{v;X}(y) \ge q^\pm_{v;X'}(y')$ whenever $y \in \Omega^*_X$, $y' \in \Omega^*_{X'}$ and $y \ge y'$.
It readily follows from this that $\psi_{v;X}(y) \ge \psi_{v;X'}(y')$ whenever $y \ge y'$.
Hence, $\psi^*_{A;X}(y) \ge \psi^*_{A;X'}(y')$ whenever $y \ge y'$, and hence also $\varphi_X(y) \ge \varphi_{X'}(y)$ whenever $y \ge y'$.
It follows that, almost surely, $\Phi^1_X \circ \cdots \circ \Phi^n_X(\bar*) \ge \Phi^1_{X'} \circ \cdots \circ \Phi^n_{X'}(\bar*)$ for all $n$, and thus, $\bar X \ge \bar X'$ as required.

It remains to prove the claim that $q^\pm_{v;X}(y) \ge q^\pm_{v;X'}(y')$ whenever $y \in \Omega^*_X$, $y' \in \Omega^*_{X'}$ and $y \ge y'$. Fix such $y,y'$. Let us show that $q^-_{v;X}(y) \ge q^-_{v;X'}(y')$; the other case is similar. We need to show that
\[ \inf_{\substack{x \in \Omega_X\\x \precsim y}} \liminf_{r\to\infty}\, \E[X_v \mid X_{B^*_r(v)} = x_{B^*_r(v)}] \ge \inf_{\substack{x' \in \Omega_{X'}\\x' \precsim y'}} \liminf_{r\to\infty}\, \E[X'_v \mid X'_{B^*_r(v)} = x'_{B^*_r(v)}] .\]
Fix $x \in \Omega_X$ such that $x \precsim y$. We need to show that $\liminf \E[X_v \mid X_{B^*_r(v)} = x_{B^*_r(v)}]$ is at least the right-hand side.
If we can show that there exists $x' \in \Omega_{X'}$ such that $x \ge x'$ and $x' \precsim y'$, then we would be done as $X \succeq_* X'$ would imply via~\eqref{eq:XY-*dom-def} that
\[ \liminf \E[X_v \mid X_{B^*_r(v)} = x_{B^*_r(v)}] \ge \liminf \E[X'_v \mid X'_{B^*_r(v)} = x'_{B^*_r(v)}] .\]
To show the existence of $x'$, it suffices to show that $\Omega(A)\neq\emptyset$ for every finite $A \subset V$, where
\[ \Omega(A) := \big\{ x' \in \{0,1\}^A : x_A \ge x',\, x' \precsim y'_A,\,  \P(X'_A=x')>0 \big\} .\]
We prove this by induction on $|A|$. The base of the induction is trivial as $|\Omega(\emptyset)|=2^0=1$.
Let $A$ be non-empty and denote $B := \{ v \in A : y'_v = * \}$.
Suppose first that $B=\emptyset$. Since $y' \in \Omega^*_{X'}$, there exists $z \in \Omega_{X'}$ such that $z \precsim y'$. Since $B=\emptyset$, this means that $z_A \le x_A$ (if $z_v=1$ then $1=y'_v\le y_v$ so that $y_v=x_v=1$), so that $z_A \in \Omega(A)$.
Suppose next that $B \neq \emptyset$. Let $v \in B$ and denote $F:=A \setminus \{v\}$.
By the induction hypothesis, there exists $x' \in \Omega(F) \neq \emptyset$. If $x_v=1$, then $x'' \in \Omega(A)$, where $x''$ is any extension of $x'$ to $\{0,1\}^A$ so that $\P(X'_v=x''_v \mid X'_F=x')>0$. If $x_v=0$, then~\eqref{eq:XY-*dom-def} implies that $\P(X'_v=0 \mid X'_F=x') \ge \P(X_v=0 \mid X_F=x_F)>0$, so that $x'' \in \Omega(A)$, where $x''$ is the extension of $x'$ to $\{0,1\}^A$ having $x''_v=0$. This completes the proof.
\end{proof}

It remains to prove \cref{prop:coding-radius}.
We shall dominate the $\Phi^n_{r,v}$ dynamics by a simpler auxiliary dynamics which depends on the parameter $p:=p_*(X)$, but otherwise does not depend on the law of $X$. This new dynamics can be seen as a model for disease spreading and is motivated by the following properties of $q_v^\pm$. We write $v \xleftrightarrow{y} *$ to indicate that either there is a simple path $(v,v_1,\dots,v_k)$ such that $y_{v_1},\dots,y_{v_{k-1}}=0$ and $y_{v_k}=*$, or there is an infinite simple path $(v,v_1,v_2,\dots)$ such that $y_{v_1},y_{v_2}=\cdots=0$.

\begin{claim}\label{cl:qv}
Let $v \in V$ and $y \in \Omega^*$.
\begin{enumerate} 
\item $q_v^-(y) \ge p_*(X)$.
\item If $q_v^+(y) \neq q_v^-(y)$ then $v \xleftrightarrow{y} *$.
\end{enumerate}
\end{claim}
\begin{proof}
Using that $\E[X_v \mid X_{B^*_r(v)}=x_{B^*_r(v)}] = \E[\E[X_v \mid (X_u)_{u \neq v}] \mid X_{B^*_r(v)}=x_{B^*_r(v)}]$ for $x \in \Omega$, the first item follows easily from the definition of $p_*(X)$.

For the second item, let $A$ be the connected component of $v$ in $\{ u \in V : y_u \neq 1\} \cup \{v\}$. Then $v \not\xleftrightarrow{y} *$ if and only if $A$ is finite and $y_{A \setminus \{v\}} \equiv 0$.
Suppose that $v \not\xleftrightarrow{y} *$.
Let $x \in \Omega$ be such that $x \precsim y$. Then $x_{A \setminus \{v\}} \equiv 0$ and $x_{\partial A} \equiv 1$. Since $X$ is decoupled by ones, $\E[X_v \mid X_{B^*_r(v)}=x_{B^*_r(v)}] = \E[X_v \mid X_{A \setminus \{v\}}\equiv 0,~X_{\partial A}\equiv 1]$ whenever $B_r(v)$ contains $A$, and hence $q^\pm_v(y)$ are both equal to the latter.
\end{proof}

Consequently, we have the following two simple observations: first, $\psi_v(y)_v=1$ whenever $u_v \le p$, and second, $\psi_v(y) = *$ only if $v \xleftrightarrow{y} *$.
The disease spreading dynamics is obtained by taking these observations to be the update rule, that is, by replacing $\psi_v$ with $\tilde\psi_v$ defined below. We proceed to discuss this in the next section.

\subsection{A disease spreading model}\label{sec:disease}

We now describe a disease spreading model which is closely related to the above dynamics.
The state space for the dynamics is $\{0,1,*\}^V$ and there is a parameter $p \in (0,1)$. Roughly speaking, we think of the value $*$ as indicating an infection, of $1$ as uninfected, and of 0 as a spreader of the infection. When an update occurs at a site $v$, with probability $p$, independently of everything else, the value at $v$ is set to be $1$; otherwise, it is deterministically set to be $*$ or 0 according to whether $v$ is connected to an infected site by a path passing only through spreaders. The main question of concern here is whether the disease eventually dies out when starting from all sites infected. We show in \cref{prop:coding-radius2} below that it does indeed die out for $p$ close enough to 1.

To define the dynamics formally, we use a similar construction and notation as in \cref{S:Xbar}, replacing the single-site update rule $\psi_v$ with $\tilde\psi_v$ defined by
\[ \tilde\psi_v(y)_w := \begin{cases}
 y_w &\text{if }w \neq v \\
 1 &\text{if }w = v\text{ and }u_v \le p \\
 0 &\text{if }w = v\text{ and }u_v > p\text{ and } v \not\xleftrightarrow{y} * \\
 * &\text{if }w = v\text{ and }u_v > p\text{ and } v \xleftrightarrow{y} *
\end{cases} ,\]
where $v \xleftrightarrow{y} *$ was defined before \cref{cl:qv}.

Crucially, there is a comparison between the original construction of \cref{S:Xbar} and this new one. Let $\le$ be the pointwise partial order on $\{0,1,*\}^V$ induced by the total order on $\{0,1,*\}$ in which $1 \le 0 \le *$ (note that $\le$ extends $\preceq$).
Using the two observations after \cref{cl:qv}, and the additional observation that $v \xleftrightarrow{y} *$ is $\le$-monotone in $y$, it is not hard to see that
\begin{equation}\label{eq:dynamics-comparison}
\psi_v(y) \le \tilde\psi_v(y') \qquad\text{whenever }y \le y' .
\end{equation}
Hence, defining $\tilde\psi_A$ and $\tilde\psi_A^*$ as before,
\[ \psi^*_A(y) \le \tilde\psi^*_A(y') \qquad\text{whenever }y \le y' .\]
In particular, $\Phi^1_{r,v} \circ \cdots \circ \Phi^n_{r,v}(\bar*) \le \tilde\Phi^1_{r,v} \circ \cdots \circ \tilde\Phi^n_{r,v}(\bar*)$, where $\tilde\Phi^i_{r,v} := \tilde\psi^*_{B_r(v)}(\cdot ; U^i,T^i)$.
\cref{prop:coding-radius} is therefore a consequence of the following.

\begin{prop}\label{prop:coding-radius2}
There exists a constant $C=C(\Delta,p)$ such that for all $v\in V$ and $n,r>0$,
\[ \P\left(\tilde\Phi^1_{r,v} \circ \cdots \circ \tilde\Phi^n_{r,v}(\bar*)_v = * \mid T \right) \le C((3\Delta-1)(1-p))^{\min\{r,n/2\}} \qquad\text{almost surely} .\]
\end{prop}

Note that the upper bound holds even conditionally on the update times $T=(T^i_v)_{v \in V,i \ge 1}$. {Note also that $\tilde\Phi^i_{r,v}$ depends on $(U^i,T^i)$ only through $(\1(U^i_u \le p),T^i_u)_{u \in B_r(v)}$.}

The proof of \Cref{prop:coding-radius2} uses a technique which is ubiquitous in problems involving disease/rumor spreading (e.g., the voter model and the contact process; see~\cite{liggett1985interacting}).
The rough idea is illustrated in \Cref{fig:auxiliary} and we now briefly explain it. In \Cref{fig:auxiliary}, time runs in the vertical direction (downwards) and the graph is represented by the horizontal axis.
The final configuration $\tilde\Phi^1_{r,v} \circ \cdots \circ \tilde\Phi^n_{r,v}(\bar*)$ is at the bottom (this is time $t=0$), while the initial configuration $\bar*$ is at the top (this is time $t=-n$).
Since the dynamics is restricted to the ball of radius $r$ around $v$, there are stars farther out in the horizontal direction.
The $i$-th step of the dynamics (corresponding to $\tilde\Phi^i_{r,v}$) is associated with the time interval $[-i,-i+1]$.
In each such step, there is a single update at each site $u$, which occurs at time $t=T^i_u-i$, and is represented in the figure by either $\updatebad$ or $\updategood$, according to whether $U^i_u>p$ or $U^i_u \le p$, respectively. In other words, a $\updategood$ represents an update that results in the value 1, while a $\updatebad$ represents an update that results in either the value $*$ or 0. Of course, in the latter case, whether the resulting value is $*$ or 0 is determined by whether $v \xleftrightarrow{y} *$ occurs or not for the configuration $y$ just before the update takes place, and this in turn depends on all the updates at previous times.
Being agnostic of this information yields precisely \iid\ updates.

As described above, we think of the times at which a site $u$ is updated as a subset of $(-\infty,0)$ given by $\{ T^i_u-i \}_{i=1}^\infty$.
Typically, the point process governing the update times of a site is taken to be a Poisson point process, but for our purposes the special properties of Poisson point processes are not needed, and we could allow the updates to be governed by any point process\footnote{We only need that different sites have distinct update times and every site has infinitely many updates.}. For us, it is convenient that the gaps between consecutive updates of a vertex are bounded (this allows us to get an upper bound which holds conditionally on the update times), but this is not crucial.

\begin{figure}
\centering
\includegraphics[scale = 1]{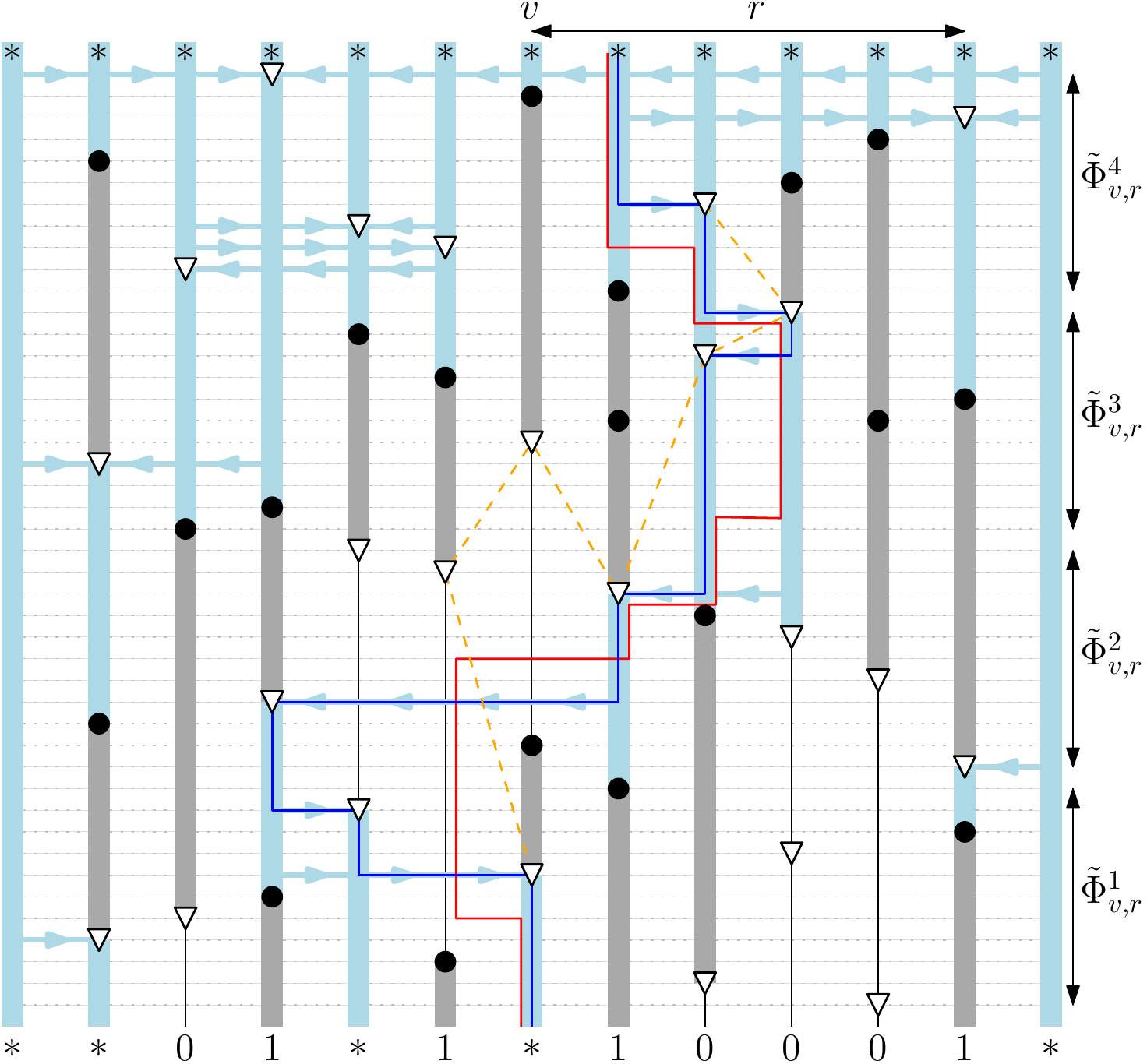}
\caption{An illustration of the dynamics $\tilde\Phi^i_{v,r}$ dynamics with $n=4$ and $r=5$.
The initial configuration at time $t=-n$ consists of all stars (at the top), and stars are also placed outside the ball of radius $r$ around $v$ (at the sides). The final configuration at time $t=0$ is $\tilde\Phi_{v,r}^1 \circ \cdots \circ \tilde\Phi_{v,r}^n(\bar *)$ (at the bottom).
A solid disc $\updategood$ denotes an update that produces a 1, and a triangle $\updatebad$ denotes an update that produces either a 0 or a $*$.
If water sources are placed at the top and sides, and water (shown in blue) is allowed to flow downwards without passing through an update or sideways into a $\updatebad$ without passing through a barrier (shown in gray) created by a $\updategood$, then a star value occurs in the final configuration precisely when water reaches the bottom. 
In the figure, water flows along the blue path from the top to bottom so that $v$ is a star in the final configuration. The blue path describes a chain that witnesses the occurrence of $\cE$, while the red path describes a shortest such chain ${\bf c}$ (note that the red path does not conform to the rules of water flow; this relaxation in the definition of a chain is useful in \cref{claim:U}). The dashed orange path depicts the update sequence $u({\bf c})$ corresponding to the chain $\bf c$. The proof of \Cref{prop:coding-radius2} is based on a Peierls argument over such update sequences.
}
\label{fig:auxiliary}
\end{figure}

\begin{proof}[Proof of \Cref{prop:coding-radius2}]
Fix $v \in V$ and $n,r>0$.
For $x \in V$, let $\cT_x := \{ T^i_x-i \}_{i=1}^\infty \subset (-\infty,0)$ be the set of times at which $x$ is updated.
Throughout the proof, we assume that we are on the almost sure event that no two updates occur simultaneously (i.e., $\cT_x \cap \cT_y = \emptyset$ for all $x \neq y$). 
Let
 \[\cU := \{ (x,t) \in V \times (-\infty,0) : t \in \cT_x \} \]
denote the update locations in space-time. We will only use $x \in B_r(v)$ and $t \ge -n$, but we have not incorporated this in the definition of $\cU$ as it is not important for this definition.
For $t \le 0$, denote the first time prior to $t$ (or at $t$) at which $x$ is updated by
\[ T_{x,t} := \max (\cT_v \cap (-\infty,t]) = \max\{ s \le t : (x,s) \in \cU \}. \]

Roughly speaking, a chain is a path in $G \times (-\infty,0]$, which can move only up and horizontally in \Cref{fig:auxiliary}, and can only pass through an update at a turning point from vertical to horizontal. Formally, we define a \textbf{chain} to be a pair consisting of a path $(v=x_0,x_1,\dots,x_k)$ in $G$ and a decreasing sequence of times $0=t_0 \ge t_1 \ge \dots \ge t_{k+1}$ such that $T_{x_i,t_i} \le t_{i+1}$ for all $0 \le i < k$. 
We think of such a chain as a `path'  in $G \times (-\infty,0)$\footnote{It is not technically a path as $G \times (-\infty,0)$ is not a graph.} consisting of alternating temporal and spatial steps (some of which may be degenerate), with the former being of the form $((x_i,t) : t_i \ge t \ge t_{i+1})$ for $0 \le i \le k$ and the latter being of the form $((x_{i-1},t_i),(x_i,t_i))$ for $1 \le i \le k$. Furthermore, the condition $T_{x_i,t_i} \le t_{i+1}$ for all $0 \le i < k$ ensures that the temporal segments cannot move `through' an update, and must switch to a spatial segment at an update point. However it is allowed to switch to a spatial segment if there is no update.
We call $k$ the \textbf{length} of the chain, and we call $(x_k,t_{k+1})$ the \textbf{endpoint} of the chain.

For $(x,t) \in \cU$, denote the indicator of the event that $x$ is updated to be 1 at time $t$ by
\[ S_{x,t} := \1(U^{\lceil -t \rceil}_x \le p) .\]
For a chain ${\bf c}$ as above, let $E_{\bf c}$ be the event that $S_{x_i,T_{x_i,t_i}} =0$ for all $0 \le i < k$ (all updates along the chain are $\updatebad$, except possibly the last).
Let $\cE$ be the event that $E_{\bf c}$ occurs for some chain $\bf c$ whose endpoint is outside $B_r(v) \times (-n,0]$.

\begin{claim}\label{claim:E}
$\{ \tilde\Phi^1_{r,v} \circ \cdots \circ \tilde\Phi^n_{r,v}(\bar*)_v = *\} \subset \cE $. 
\end{claim}
\begin{proof} 
Note that $\cU \cap (B_r(v) \times (-n,0))$ is the set of updates associated to $\tilde\Phi^1_{r,v} \circ \cdots \circ \tilde\Phi^n_{r,v}$.
For $0 \ge t \ge -n$, let $y^t$ be the state obtained by starting from $\bar*$ and applying the updates in $\cU \cap (B_r(v) \times (-n,t])$ in order, so that $y^{-n}=\bar*$ and $y^0=\tilde\Phi^1_{r,v} \circ \cdots \circ \tilde\Phi^n_{r,v}(\bar*)$.
The claim is that $ \{y^0_v=*\} \subset E$.

Suppose that $y^0_v=*$. We construct a chain ${\bf c}$ witnessing $\cE$. Set $v_0:=v$ and $s_0:=0$. For $i \ge 1$, we proceed inductively as follows. Suppose we have defined $v_{i-1}$ and $s_{i-1}$ so that $v_{i-1} \in B_r(v)$, $s_{i-1}>-n$ and $y^{s_{i-1}}_{v_{i-1}} = *$.
Set $s_i := T_{v_{i-1},s_{i-1}}$.
If $s_i \le -n$, then we stop the inductive procedure.
Otherwise, note that $y^{s_i}_{v_{i-1}} = *$, since $y^{s_{i-1}}_{v_{i-1}} = *$ and the only update of $v_{i-1}$ in $[s_{i-1}, s_{i}]$ is at time $s_i$.
Thus, $v_{i-1} \xleftrightarrow{y^{s_i}} *$.  
Let $p_i$ be a path in $G$ witnessing this connection, i.e., a path starting at $v_{i-1}$, ending at a site $v_i$ such that $y^{s_i}_{v_i}=*$, and otherwise passing only through sites $w$ with $y^{s_i}_w=0$.
If $v_i \notin B_r(v)$, then we set $s_{i+1}:=s_i$ and stop the inductive procedure.

At the end of this procedure, we have sequences $(v_0,\dots,v_m)$, $(s_0,\dots,s_{m+1})$ and $(p_1,\dots,p_m)$.
Note that $p_i$ starts at $v_{i-1}$ and ends at $v_i$.
Let $p'_i$ be $p_i$ without its first element.
Set $x_0:=v_0=v$ and let $(x_1,\dots,x_k)$ be the concatenation of $p'_1,p'_2,\dots,p'_m$.
Set $t_0:=s_0=0$ and $t_{k+1}:=s_{m+1}$ and let $(t_1,\dots,t_k)$ be the sequence of times $s_1,\dots,s_m$, where each $s_i$ is repeated consecutively $|p'_i|$ times.
It is not hard to see that $((x_0,\dots,x_k),(t_0,\dots,t_{k+1}))$ describes a chain ${\bf c}$ for which $E_{\bf c}$ occurs.
 \end{proof}

Our goal now is to bound the probability of $\cE$ via a Peierl's type argument.
Given a chain ${\bf c}$, define its \textbf{update sequence} to be $u({\bf c}) := ((x_i,T_{x_i,t_i}) : 0 \le i < k) \in \cU^k$. The elements of $u({\bf c})$ are not necessarily distinct. When they are distinct, $\{ S_{x_i,T_{x_i,t_i}} : 0 \le i < k\}$ is a collection of (distinct) independent random variables, so that
\[ \P(E_{\bf c} \mid T) = (1-p)^k .\]
Let $H$ be the graph on $\cU$ in which $(x,t)$ and $(x',t')$, with $t>t'$, are adjacent if $x \sim x'$ and $T_{x',t}=t'$ (i.e., if $x'$ is not updated in $(t',t]$).

\begin{claim}\label{claim:U}
If $\cE$ occurs and ${\bf c}$ is a shortest chain witnessing its occurrence, then $u({\bf c})$ is a simple path in $H$ of length at least $\min\{r,n/2\}$.
\end{claim}
\begin{proof}
Suppose $\cE$ occurs and let ${\bf c}$ be a shortest chain witnessing its occurrence.
Let us first show that $u({\bf c})$ is a path in $H$.
Fix $0<i<k$ and consider the consecutive updates points $(x,t):=(x_{i-1},T_{x_{i-1},t_{i-1}})$ and $(x',t'):=(x_i,T_{x_i,t_i})$. We have that $x \sim x'$ by the definition of a chain.
Also by the definition of a chain, we have that $t \le t_i \le t_{i-1}$.
Thus, we see that $T_{x,t_i}=t$ and $T_{x',t_i}=t'$.
It follows that $T_{x,t \vee t'}=t$ and $T_{x',t \vee t'}=t'$.
This shows that $(x,t)$ and $(x',t')$ are adjacent in $H$. Thus, $u({\bf c})$ is a path in $H$.

Let us now show that $u({\bf c})$ is a simple path, i.e., that its elements are distinct.
Write $u({\bf c})=(u_0,\dots,u_{k-1})$. Suppose towards a contradiction that the $u_a=u_b$ for some $0 \le a < b \le k-1$. Then we claim that
$((x_0,\dots,x_a,x_{b+1},\dots,x_k),(t_0,\dots,t_a,t_{b+1},\dots,t_{k+1}))$ describes a shorter chain ${\bf c}'$. Indeed, in $\bf c'$, the $t_i$'s are decreasing and the $x_i$'s are adjacent in $G$ since $x_a=x_b$. Furthermore, $T_{x_a, t_a} = T_{x_b,t_b} \le t_{b+1}$, since $u_a=u_b$ and $\bf c$ is a chain. Thus, $\bf c'$ is a chain. Since $E_{\bf c}$ occurs, it is also clear that $E_{{\bf c}'}$ occurs, and since ${\bf c}'$ has the same endpoint as ${\bf c}$, it also witnesses $\cE$. This contradicts the choice of ${\bf c}$ as the shortest such chain.

Let us show that $k \ge \min\{r,n/2\}$.
By definition, the endpoint $(x_k,t_{k+1})$ of ${\bf c}$ is outside $B_r(v) \times (-n,0]$. If $x_k \notin B_r(v)$, then we clearly have that $k \ge r$. Otherwise, $t_{k+1} \le -n$. Since $t_i-2 < T_{x_i,t_i} \le t_{i+1}$ for $0 \le i < k$ (the former inequality is due to the fact that every site is updated once in every unit interval, and the latter is by the definition of a chain), $t_{k+1}>-2k$, so that $n<2k$.  
\end{proof}

We can now bound the probability of $\cE$ via a union bound.
The maximum degree of $H$ is at most $3\Delta$.
The number of simple paths of length $k$ in $H$ starting from $(v,T_{v,0})$ is thus at most $3\Delta(3\Delta-1)^{k-1}$. Therefore,
\[ \P(\cE \mid T) \le \sum_{k=\min\{r,n/2\}}^{\infty} 3\Delta(3\Delta-1)^{k-1}(1-p)^k =  C((3\Delta-1)(1-p))^{\min\{r,n/2\}}, \]
where $C:=3\Delta/((3\Delta-1)(1-(3\Delta-1)(1-p)))$.
\end{proof}

\subsection{An extension to processes taking values in a finite set}\label{sec:ffiid-finite-alphabet}

\cref{thm:ffiid-general} and its proof can be extended to a version allowing processes taking values in a finite set other than $\{0,1\}$.
Let $\cS$ be a finite set.
Let $X$ be a $\cS$-valued process.
For $s \in \cS$, define
\[ p_{*;s}(X) := \inf_{v \in V} \essinf \P(X_v = s \mid (X_u)_{u \neq v}) .\]
Denote $p_{*;S}(X) := \sum_{s \in S} p_{*;s}(X)$ for $S \subset \cS$.
Given $S \subset \cS$, we say that $X$ is \textbf{decoupled by $S$} if for any finite set $A \subset V$ and $\tau \in S^{\partial A}$ such that $\P(X_{\partial A}=\tau)>0$, we have that $X_A$ and $X_{A^c}$ are conditionally independent given that $X_{\partial A}=\tau$. 

Recall the definition of $X \succeq_* Y$ for $\{0,1\}$-valued processes from~\eqref{eq:XY-*dom-def}. When $\cS$ is totally ordered, this extends to $\cS$-valued processes by requiring that $\P(X_v \in \cdot \mid X_F=x) \ge_{st} \P(Y_v \in \cdot \mid Y_F=y)$ for all $v\in V$, finite $F \subset V\setminus \{v\}$, $x,y \in \cS^F$ such that $x \ge y$ and $\P(X_F=x),\P(Y_F=y)>0$.
As before, it is not hard to see that $X \succeq_* Y$ implies that $X \ge_{st} Y$.

\begin{thm}\label{thm:ffiid-general2}
Let $G$ be a connected bounded-degree graph and let $X$ be an invariant $\cS$-valued process. Suppose that there exists $S \subset \cS$ such that $X$ is decoupled by $S$ and $p_{*;S}(X)>1-\frac1{3\Delta-1}$. Then $X$ is a uniformly finitary factor of an i.i.d.\ process, and the coding radii satisfy~\eqref{eq:coding-radius-bound}.

Moreover, if $\cS$ is totally ordered, $X$ and $Y$ are invariant $\cS$-valued processes such that $X \succeq_* Y$, and there exists $S \subset \cS$ such that both $X$ and $Y$ are decoupled by $S$ and $\sum_{s \in S} \min\{p_{*;s}(X),p_{*;s}(Y)\} > 1-\frac1{3\Delta-1}$, then $X$ and $Y$ can be expressed as uniformly finitary factors of a common i.i.d.\ process in such a way that $X \ge Y$ almost surely. In particular, there is an invariant monotone coupling of $X$ and $Y$. In the special case when $Y$ is an \iid\ process, we can choose $Y$ as a 0-block factor of the i.i.d.\ process.
\end{thm}
\begin{figure}[!ht]
\centering
\includegraphics[width=\textwidth]{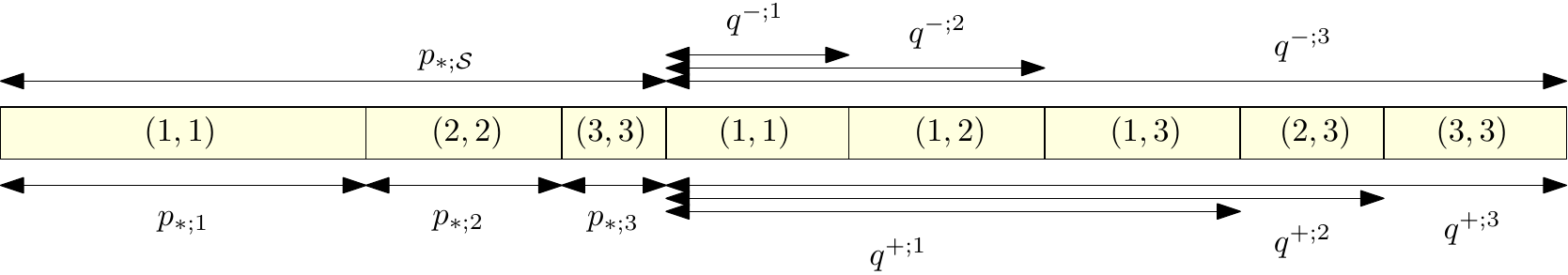}
\caption{An illustration of the definition of $\psi_v$ in the proof of \Cref{thm:ffiid-general2}.}\label{fig:auxilary_gen}
\end{figure}

\begin{proof}[Proof sketch]
Let us fix $\cS=\{1,\dots,q\}$.
If we would only want to show the first part of the theorem, it would suffice to work with the state space $\{1,\dots,q,*\}$ in place of $\{0,1,*\}$. However, in order to prove the second part, we need to work with the space $\cS^* := \{(a,b) \in \{1,\dots,q\}^2 : a \le b\}$. We think of $y_v=(a,b)$ as specifying that $a \le x_v \le b$ (note that $(a,b)$ is a pair of integers, not an interval). In particular, $(a,a)$ means the value is known to be $a$, while $(1,q)$ means the value is completely unknown (this was previously the role of the star). Other values give partial information. The partial order $\precsim$ on $\cS^*$ is given by $(a,b) \precsim (a',b')$ if and only if $a \ge a'$ and $b \le b'$, and this induces the pointwise partial order $\precsim$ on $(\cS^*)^V$.

Define $p_{*;s} := \min\{p_{*;s}(X),p_{*;s}(Y)\}$ for $s \in S$ (the distinguished subset of $\cS$ as assumed in the theorem) and set $p_{*;s} := 0$ for $s \in \cS \setminus S$. Denote $p_{*;S'} := \sum_{s \in S'} p_{*;s}$ for $S' \subset \cS$.
Set $q_v^{\pm;i}(y):=0$ and for $i \in \{0,1,\dots,q\}$, define
\begin{align*}
q_v^{-;i}(y) &:= \inf_{x \in \Omega, x \precsim y} \liminf_{r \to \infty} \P(X_v \le i \mid X_{B^*_r(v)}=x_{B^*_r(v)}) - p_{*;\{1,\dots,i\}} ,\\
q_v^{+;i}(y) &:= \sup_{x \in \Omega, x \precsim y} \limsup_{r \to \infty} \P(X_v \le i \mid X_{B^*_r(v)}=x_{B^*_r(v)})- p_{*;\{1,\dots,i\}} .
\end{align*}
We then define
\[ \psi_v(y)_v := \begin{cases}
(i,i) &\text{if } p_{*;\{1,\dots,i-1\}} < u_v \le p_{*;\{1,\dots,i\}} \text{ for some }i \in \cS \\
(a,b) &\text{if }\,\substack{q_v^{+;a-1}(y) < u_v-p_{*;\cS} \le q_v^{+;a}(y)\\q_v^{-;b-1}(y) < u_v-p_{*;\cS} \le q_v^{-;b}(y)}\, \text{ for some }(a,b) \in \cS^*
\end{cases} .\]
It is straightforward to check that $a\le b$ always in the above definition using the facts that $q^{\pm;i}_v(y)$ is non-decreasing in $i$ and   that  $q^{-;i}_v(y) \le q^{+;i}_v(y)$  (see \Cref{fig:auxilary_gen}).
One checks that~\eqref{eq:psi*-monotone} still holds.
For the auxiliary dynamics, the idea is to clump together all $(a,b) \in \cS^*$ with $a<b$ into a single state $*$. To that end, we can define $p:=p_{*;\cS}$ and continue to work with the state space $\{0,1,*\}^V$ and the same definition of $\tilde\psi_v$, with the interpretation that 1 indicates a known value in $S$, 0 indicates any known value, and $*$ indicates an unknown value. The comparison~\eqref{eq:dynamics-comparison} between the dynamics can then be verified to hold upon reinterpreting $\le$ as a relation between $(\cS^*)^V$ and $\{0,1,*\}^V$ induced by the relation between $\cS^*$ and $\{0,1,*\}$ in which $(a,b) \le c$ if and only if either $c=*$, or $a=b \in S$, or $c=0$ and $a=b$. The rest of the proof that $\bar X$ is a uniformly finitary factor satisfying~\eqref{eq:coding-radius-bound} remains unchanged. The fact that $\bar X$ has the same distribution as $X$ is shown as before, noting that $q^{-;i}_v(X)=q^{+;i}_v(X)=\P(X_v \le i \mid X_{B^*_r(v)}) - p_{*;\{1,\dots,i\}}$ almost surely.

The reason for having $Y$ involved in the definitions for $X$ (via the definition of $p_{*;s}$) is for the moreover part of the theorem. This too proceeds as before, with $\ge$ being the pointwise partial order on $(\cS^*)^V$ induced by the partial order on $\cS^*$ in which $(a,b) \ge (a',b')$ if and only if $a \ge a'$ and $b \ge b'$.
The only thing that needs to be checked is that $q^{\pm;i}_{v;X}(y) \ge q^{\pm;i}_{v;Y}(y')$ whenever $y \ge y'$, as it readily follows from this that $\psi_{v;X}(y) \ge \psi_{v;Y}(y')$ whenever $y \ge y'$ (here it is important that we used the same $p_{*;s}$ in the definitions for both $X$ and $Y$). This can be verified in a similar manner as before.
\end{proof}

\subsection{Plus state of Ising model -- Proof of \cref{thm:ising-ffiid}}\label{sec:ising-ffiid}

We begin with a simple lemma regarding decoupling by ones.
When $X$ is a $\{0,1\}^2$-valued process, we think of $(1,1)$ as the ``one'' state. In particular, if $X$ and $Y$ are $\{0,1\}$-valued processes and $(X,Y)$ is decoupled by ones, then $XY$ is also decoupled by ones.
 
\begin{lemma}\label{lem:decoupled-by-ones-product}
Let $X$ and $Y$ be independent processes each of which is decoupled by ones. Then $(X,Y)$ is also decoupled by ones.
\end{lemma}
\begin{proof}
Let $A \subset V$ be a finite set. Let $E_X$ be the event that $X_{\partial A} \equiv 1$. Similarly define $E_Y$ and $E_{X,Y}$. Note that that $E_{X,Y} \subset E_X \cap E_Y$.
Suppose that $E_{X,Y}$ has positive probability.
We need to show that $(X,Y)_A$ and $(X,Y)_{A^c}$ are conditionally independent given $E_{X,Y}$. Since $E_{X,Y}$ is measurable with respect to $(X,Y)_{A^c}$, this is the same as saying that the conditional law of $(X,Y)_A$ given $(X,Y)_{A^c}$ is almost surely equal to some law $\mu_A$ on the event $E_{X,Y}$.
Indeed, given $(X,Y)_{A^c}$, since $X$ and $Y$ are independent and each is decoupled by ones, the conditional law of $X_A$ is almost surely equal to some law $\nu_A$, the conditional law of $Y_A$ is almost surely equal to some law $\pi_A$, and $X_A$ and $Y_A$ are conditionally independent, so that $\mu_A = \nu_A \times \pi_A$.
\end{proof}

Let $\sigma$ be sampled from $\mu^+_\beta$ and define $X \in \{0,1\}^V$ by $X_v := \frac12(\sigma_v+1)$. Set $p_0 := 1-\frac1{3\Delta-1}$, let $p \in (p_0,1)$ and let $Y \sim \nu_p$ be independent of $X$.
We have shown in the proof of \cref{thm:stochastic-dom-Ising2} that $p_*(XY) \to p$ as $\beta\to\infty$. In particular, $p_*(XY)>p_0$ for all $\beta$ sufficiently large.

Suppose that $\beta$ is such that $p_*(XY)>p_0$.
Then \cref{thm:ffiid-general} implies that $XY$ is a finitary factor of an \iid\ process. Since $(X,Y)$ is decoupled by ones, the law of $X$ given $XY$ is independent on the clusters of $\{ v \in V : (XY)_v=0 \} = \{ v\in V : (X,Y)_v \neq (1,1) \}$, and the law on each cluster depends only on the shape of the cluster (in fact, we have seen that $X$ given $XY$ is an Ising model with a magnetic field). Thus, one may obtain a sample of $X$ by first sampling $XY$ as a finitary factor of an \iid\ process, setting $X$ to equal 1 wherever $XY$ is such, and then independently sampling $X$ on each 0-cluster of $XY$ (all of which are finite almost surely) according to its conditional law, which depends only on the shape of the 0-cluster.
This translates to a description of $X$ as a finitary factor of $XY$ and an additional independent \iid\ source, or simply, as a finitary factor of a single (larger) \iid\ process. While such things are fairly standard, we still include a proof for completeness. 

Let us make the details of this description more explicit (there are many ways to do this). Let $\pi$ be the coupling between $X$ and $Z:=XY$. First, we need to talk about finite connected labeled subsets of $V$, up to automorphisms. Let $\cA$ be the set of all pairs $(U,\tau)$, where $U$ is a finite connected subset of $V$ and $\tau \colon U \to \{1,\dots,|U|\}$ is a bijection. Two such pairs are equivalent if there is a label-preserving automorphism of $G$ between them. Write $[U,\tau]$ for the equivalence class of $(U,\tau)$ and let $\bar\cA$ be the set of all equivalence classes. For each $[U,\tau] \in \bar\cA$, fix a function $f_{[U,\tau]} \colon [0,1] \to \{0,1\}^{|U|}$ such that $f_{U,\tau}(\Xi)$ has the same law as $\pi_U \circ \tau^{-1}$, where $\Xi$ is a uniform random variable on $[0,1]$ and $\pi_U$ is the conditional law of $X_U$ under $\pi$ given that $Z_U \equiv 0$ and $Z_{\partial U} \equiv 1$. This is possible since $\pi$ is an invariant coupling.
Now let $\xi=(\xi_v)_{v \in V}$ and $\eta=(\eta_v)_{v \in V}$ be i.i.d.\ processes, with $Z,\xi,\eta$ independent, consisting of Uniform$(0,1)$ random variables. Let $U_v$ denote the 0-cluster of $Z$ containing $v$ (set $U_v$ to be empty if $Z_v=1$). Endow $U_v$ with the total order induced by $(\eta_u)_{u \in U_v}$, let $\tau_v \colon U_v \to \{1,\dots,|U_v|\}$ be the unique order-preserving bijection and let $u_v \in U_v$ denote the minimal element in $U_v$.
Define
\[ \phi(Z,\xi,\eta)_v := \begin{cases}
 1 &\text{if }Z_v=1\\
 f_{[U_v,\tau_v]}(\xi_{u_v})_{\tau_v(v)} &\text{if }Z_v=0
\end{cases} .\]
It is straightforward to check that $\phi$ is a finitary map that commutes with automorphisms of $G$.
Using that $X$ is decoupled by ones of $Z$, one also checks that $\phi(Z,\xi,\eta)$ has the same law as $X$.
Thus, $X$, or equivalently, $\sigma$, is a finitary factor of an \iid\ process.

\subsection{Infinite clusters of Bernoulli percolation -- Proof of \cref{thm:perc-ffiid}}\label{sec:perc-ffiid}

Let $\omega$ denote Bernoulli percolation on $G$ of parameter $p$ and let $\omega^\infty$ consist of those vertices which are in infinite clusters in $\omega$. We need to show that $\omega^\infty$ is a finitary factor of an i.i.d.\ process whenever $p$ is sufficiently close to 1.

Let us point out that unlike in the Ising model, $\omega^\infty$ is not a Markov random field, and in fact, is not even decoupled by ones. Consider, for example, a regular tree with all neighbors of a vertex $v$ having $\omega^\infty$-value 1, and one such neighbor $u$ having $\omega^\infty$-value 0 for all its  neighbors other than $v$. In this case, $\omega^\infty_v$ is deterministically 1 as the infinite cluster containing $u$ must also contain $v$. It is also clear that this conditional probability can be made strictly less than 1 by other appropriate conditionings (noting that there is positive probability that all neighbors of $v$ are connected to infinity without passing through $v$). It is easy to extend this idea to conclude that one cannot determine the conditional distribution of $\omega^\infty_v$ even if the conditioning is known on an arbitrarily large ball around $v$.

The proof of \cref{thm:perc-ffiid} broadly follows similar lines as for the Ising model, though the lack of the domain Markov property requires different approaches for the one-ended and infinitely ended cases.

\subsubsection{The one-ended case}

We say that \textbf{$X$ is decoupled by connected ones} if for any finite set $A \subset V$, we have that $X_A$ and $X_{A^c}$ are conditionally independent given that $X_{\partial A} \equiv 1$ and $\partial A$ is contained in a single cluster of $\{ v \in A^c : X_v=1\}$ (assuming this occurs with positive probability). When $X$ is a $\{0,1\}^2$-valued process, we think of $(1,1)$ as the ``one'' state. In particular, if $X$ and $Y$ are $\{0,1\}$-valued processes and $(X,Y)$ is decoupled by connected ones, then $XY$ is also decoupled by connected ones.

While $\omega^\infty$ is not decoupled by ones, it is easy to see that it is decoupled by connected ones. In order for this to be useful, we need to know that every finite subset of $V$ is almost surely contained in a set $A$ for which the above event holds for $\omega^\infty$. This is the reason for the assumption that $p_u<1$. Indeed, when $p>p_u$, there is almost surely a unique infinite cluster~\cite{schonmann1999stability}, i.e., $\omega^\infty$ is connected (and non-empty). Moreover, when $p$ is close to 1, the complement of the unique infinite cluster $\omega^\infty$ almost surely contains only finite connected components (this is easily seen to hold by a Peierls argument; see \cref{lem:inf-cluster-finite-holes}). From this it is clear that such a set $A$ exists almost surely. Note that these two properties are jointly monotone, so that any $\{0,1\}$-valued process which stochastically dominates a high-density Bernoulli percolation also has the same two properties almost surely. 
These observations allow us to proceed in a similar fashion as for the Ising model, via the following modification of \cref{thm:ffiid-general}.

\begin{thm}\label{thm:ffiid-general-conn}
Let $G$ be a connected quasi-transitive nonamenable graph with $p_u<1$. There exists $p_0<1$ such that the following holds. Let $X$ be an invariant $\{0,1\}$-valued process which is decoupled by connected ones and has $p_*(X) \ge p_0$. Then $X$ is a finitary factor of an i.i.d.\ process.
\end{thm}

\begin{proof}
The theorem is proven in a similar way as \cref{thm:ffiid-general} and we only detail the required changes. The construction of $\bar X$ is the same as in \Cref{S:Xbar}, and once we know $\bar X \in \{0,1\}^V$ almost surely, the remainder of the proof is remain unchanged. Let us now explain the modifications in \Cref{sec:disease} needed to prove $\bar X \in \{0,1\}^V$ almost surely.

We will use the notations and the setup of \Cref{sec:disease}.
We modify the definition of $v \xleftrightarrow{y} *$ by declaring that $v \not\xleftrightarrow{y} *$ when there is a finite set $A \subset V$ such that $v \in A$, $y_{\partial A} \equiv 1$, $\partial A$ is contained in a single cluster of $\{ w \in A^c : y_w=1\}$, and $y_v \neq *$ for all $v \in {A \setminus \{v\}}$. Observe that since $X$ is decoupled by connected ones, $q^-_v(y) \neq q^+_v(y)$ only if $v \xleftrightarrow{y} *$, and hence, $\psi_v(y)_v = *$ only if $v \xleftrightarrow{y} *$. Using this modified definition of $v \not\xleftrightarrow{y} *$, we define $\tilde\psi_v$ as before, and get that $\psi^*_A(y) \le \tilde\psi^*_A(y')$ whenever $y \le y'$. Denote $\tilde X := \lim_{n \to \infty} \tilde\Phi^1\circ\cdots\circ\tilde\Phi^n(\bar*)$.

Since $\bar X \le \tilde X$ almost surely, it only remains to explain why $\tilde X \in \{0,1\}^V$ almost surely.
 Since we are not attempting to bound the coding radius here, we can simplify the approach taken in the proof of \cref{prop:coding-radius2} (which is helpful since the new definition of $v \xleftrightarrow{y} *$ adds a different complexity to the proof). For $t \le 0$, define $y^t$ to be the configuration after applying the updates that occurred up to time t (i.e. updates in $\cU \cap (V\times  (-\infty,t]))$. (Although this set of update times is a countable dense set, $y^t$ is still well defined using the same monotonicity properties as in \Cref{sec:disease}). In particular, $y^0 = \tilde X$.
 Let $y^{t,\infty}_u$ be the indicator that $u$ is in an infinite cluster of 1s in $y^t$.
This is relevant since if $y^{t,\infty}_v=0$ and $v$ is blocked from infinity by the open sites of $y^{t,\infty}$, and the open sites of $y^{t,\infty}$ are connected,  then there exists a finite set $A \ni v$ such that $\partial A$ is contained in a single cluster of $\{ u \in A^c : y^t_u=1\}$. 
 For an infinite chain ${\bf c} = (v_i,t_i)_{i=0}^\infty$, we redefine $E_{\bf c}$ to be the event that $y^{t_i,\infty}_{x_i}=0$ for all $i$.
Let $\cE$ be the event that $E_{\bf c}$ occurs for some infinite chain (starting from $(v_0,t_0)=(v,0)$). Using the same argument as in \Cref{claim:E}, we obtain that $\{\tilde X_v = *\} \subseteq \cE$. 

We now need to show that $\P(\cE) = 0$. Observe that if $y^{t,\infty}_{x}=0$ then there is no infinite path $(x=x_0,x_1,\ldots)$ in  $G$ so that $S_{x_i, T_{x_i,t}} = 1$.  One can try to translate this event into an event about site percolation in $H$ as we did in the proof of \Cref{claim:U}. But the events $y^{t_i,\infty}_{x_i}=0$ are non-local and not independent of each other (even if one considers a `shortest' path in some sense). Furthermore, we need to work with times $t$ which are close to each other where the $y^t$s are strongly correlated.
 To avoid these difficulties, we exploit \Cref{thm:stochastic-dom-perc} and the fact that $p$ is close to 1 by essentially projecting to integer times as follows.

Let $\omega^i$ denote the percolation configuration $(\1(U^i_v,U^{i+1}_v \le p))_{v \in V}$ for all $i \ge 1$. Observe that $y^t_{v} \ge \omega^{\lfloor -t \rfloor +1}_v$ for all $v \in V$ and $t \le 0$. In other words, if $\omega_v^i=1$, then there is a `barrier' spanning the entire segment $\{v\} \times [-i,-i+1]$ in \cref{fig:auxiliary}.
Since $(\omega^i_v)_{i=1}^\infty$ is a 1-dependent percolation process with marginals at least $p^2$, when $p$ is sufficiently close to 1, it stochastically dominates independent Bernoulli($q$) random variables $(\tau^i_v)_{i=1}^\infty$, where $q$ is a prescribed number (also close to 1)~\cite{liggett1997domination}.
We may further take $(\tau^i_v)_{v \in V,i \ge 1}$ to be independent and assume that $\omega_v^i \ge \tau^i_v$
for all $i,v$ almost surely. Let $\tau^{i,\infty}$ denote the set of sites which are infinite clusters of $\tau^i$. Observe that $\tau^{i,\infty}_x =1$ implies that $y^{t, \infty}_x =1$ for all $t \in [-i,-i+1]$.

Now we think of $\tau' := (1-\tau^{i,\infty}_v)_{i \ge 1, v \in V}$ as a site percolation $\tau$ on a graph $H$ with vertex set $V \times \N$ and two vertices $(x,i)$ and $(y,j)$ adjacent if $x \sim y$ and $|i - j| \le 1$ . Observe that if there is a chain satisfying $E_{\bf c}$, then there is an infinite path $(x_i,j_i)$ in $H$ with $\tau^{j_i,\infty}_{x_i} =0$ for all $i$. That is, $\tau'$ has an infinite open cluster.
Since $G$ is nonamenable, we can use \Cref{thm:stochastic-dom-perc} to stochastically dominate $\tau'$ by a Bernoulli site percolation $\eta$ with parameter $p'$ on $H$ with $p'$ close to 0 (when $q$ is close to 1). Since $H$ has bounded degree, $p_c(H)>0$ and we conclude that $\eta$ (and hence also $\tau'$) does not percolate almost surely (when $p'$ is close to 0). Thus, $\P(\cE)=0$ as required.
\end{proof}

\begin{remark}
    \cref{thm:ffiid-general-conn} extends to quasi-transitive amenable graphs for which, in Bernoulli site percolation with $p$ close to 1, there is a unique infinite cluster and the connected component $C_v$ of any vertex $v$ in its complement has expected size less than $1/(\Delta+1)$. For example, it is not hard to see that this holds on $\Z^d$ when $d\ge2$.

    The only modification in the proof is in the final argument that $\tau'$ does not percolate. Here it is helpful to orient the vertical edges in $H$ from $(v,i)$ to $(v,i+1)$. It suffices to show that there is no oriented percolation of $\tau'$. Let $S_i$ be the vertices in $V \times \{i\}$ that can be reached by an open oriented path starting from $(v,1)$. Then $S_1 = C_v(\tau^1)$ and $S_{i+1} = \bigcup_{u \in S_i \cup \partial S_i} C_u(\tau^i)$. This readily implies a comparison to a subcritical branching process.
\end{remark}

\begin{lemma}
Let $X$ and $Y$ be independent processes each of which is decoupled by connected ones. Then $(X,Y)$ is also decoupled by connected ones.
\end{lemma}
\begin{proof}
The proof follows verbatim to \cref{lem:decoupled-by-ones-product} upon defining $E_X$ to be the event that $X_{\partial A} \equiv 1$ and $\partial A$ is in a single cluster of $\{ v \in A^c : X_v=1\}$.
\end{proof}

Let us now return to the proof of \cref{thm:perc-ffiid} and show that $X=\omega^\infty$ is a finitary factor of an i.i.d.\ process.
Let $Y \sim \nu_q$ be independent of $X$. Then $(X,Y)$ is decoupled by connected ones by the lemma above. In particular, $XY$ is also decoupled by connected ones.
Recall from~\eqref{eq:perc-dom*} that we have shown that $p_*(XY) \to q$ as $p\to1$. In particular, $p_*(XY) \ge q-\eps$ for $p$ sufficiently close to 1. When this occurs, with $q-\eps$ chosen in advance to be sufficiently close to 1, \cref{thm:ffiid-general-conn} tells us that $XY$ is a finitary factor of an i.i.d.\ process.
To show that $X$ is a finitary factor of an i.i.d.\ process, we shall use that $(X,Y)$ is decoupled by connected ones, and that these ones (which are the same as the ones of $XY$) are very likely.

Since $Z:=XY$ stochastically dominates $\nu_{q-\eps}$, by the earlier observations, $Z$ almost surely has a unique infinite open cluster $C_\infty$, whose complement has only finite connected components.
Using these two properties we can create the following hierarchical structure:
for $v \notin C_\infty$, let $A_v$ be the minimal finite subset $A$ of $V$ (with respect to inclusion) such that $v \in A$, $Z_{\partial A} \equiv 1$, $Z_{\partial(A^c)} \equiv 0$, and $\partial A$ is in a single cluster of $\{ u \in V \setminus A : Z_u=1\}$. Such a set exists since we can take the connected component of $V \setminus C_\infty$ containing $v$. The fact that a minimal such set exists follows from the fact that the intersection of two sets with these properties again has these properties. It is therefore also easy to see that $A_v$ is  necessarily connected.
Let $B_v$ denote the set $A_v$ minus all those $A_u$ which are proper subsets of $A_v$. Note that $B_v$ and $B_u$ are either disjoint or equal, and that each of $A_v$ and $B_v$ depends on $Z$ in a finitary manner.

Since $(X,Y)$ is decoupled by connected ones, and the coupling between $X$ and $Z$ is invariant, the conditional law of $X_{A_v}$ given $(X_{A_v^c},Z)$ depends only on the shape of $A_v$. Thus, the conditional law of $X_{B_v}$ given $(X_{B_v^c},Z)$ depends only on the shapes of $A_v$ and $B_v$. In particular, given $Z$, if $\{v_i\}_i$ are such that $\{B_{v_i}\}$ are distinct, then $\{X_{B_{v_i}}\}_i$ are conditionally independent and the conditional law of each $X_{B_{v_i}}$ depends only on the shapes of $A_{v_i}$ and $B_{v_i}$.

Hence, one may obtain a sample of $X$ by first sampling $Z$, setting $X$ to equal 1 wherever $Z$ is such, and then independently sampling $X$ on each distinct $B_v$ according to its conditional law, which depends only on $(A_v,B_v)$.
This translates to a description of $X$ as a finitary factor of $Z$ and an additional independent \iid\ source (this is done in a manner similar to that in the proof of \cref{thm:ising-ffiid}; we leave the details to the reader). Thus, $X$ is a finitary factor of an \iid\ process.

\subsubsection{The infinitely ended case}

We cannot proceed as in the previous case by utilizing the fact that $\omega^\infty$ is decoupled by connected ones, since there will almost surely be closed clusters $A$ in $\omega$ such that $V \setminus A$ contains more than one infinite connected component (so that $A$ is not contained in any finite set $A'$ for which $V \setminus A'$ is connected). Instead, we take the following approach.

Let $r \ge 0$ be such that $V \setminus B_r(v)$ has at least two infinite connected components for all $v$ (the significance of this choice is illuminated in~\eqref{eq:inf-conn} and~\eqref{eq:inf-conn2} below).
Let $G'$ be the $(6r+1)$-power graph of $G$, i.e., the graph on $V$ in which two vertices are adjacent if they are at distance at most $6r+1$ in $G$. All graph notions below (e.g., distance, balls, clusters, etc.) are taken with respect to the base graph $G$, unless explicitly indicated otherwise.
Let us already point out that since the graphs $G$ and $G'$ have the same vertex set and same automorphism group, the notion of finitary factor on them coincide.

Let $\tilde X$ be as in \cref{thm:stochastic-dom-perc-inf-ends}.
The advantage of $\tilde X$ over $\omega^\infty$ is that the former is decoupled by ones when viewed as a process on $G'$.
In fact, we claim that $(\omega,\tilde X)$ is decoupled by ones in $G'$ (this is a stronger claim since $\omega \ge \tilde X$ almost surely). We postpone the proof of this claim to below.

Set $p_0 := 1-\frac1{3\Delta(G')-1}$.
Fix $q \in (p_0,1)$ and let $Y \sim \nu_q$ be independent of $\omega$.
We have shown in the proof of \cref{thm:stochastic-dom-perc-inf-ends} that $p_*(\tilde XY)\to q$ as $p \to 1$. In particular, $p_*(\tilde XY)>p_0$ for all $p$ sufficiently close to 1.
Since $\tilde X$ is decoupled by ones in $G'$, so is $\tilde XY$ (see \cref{lem:decoupled-by-ones-product}).
Thus, by \cref{thm:ffiid-general}, $\tilde XY$ is a finitary factor of an i.i.d.\ process.

Denote $Z := \tilde XY$. Note that $Z$ almost surely has no infinite closed $G'$-clusters.
Also, since $(\tilde X,Y)$ is decoupled by ones in $G'$, given $Z$, the conditional law of $\tilde X$ is independent on each closed $G'$-cluster and the conditional law on each such cluster depends only on its shape. In particular, this can be sampled finitarily from an additional i.i.d.\ source independent of $Z$. This yields a sample of $\tilde X$ as a finitary factor of an i.i.d.\ process.

To get $\omega^\infty$ as a finitary factor of an i.i.d.\ process, we use that $(\omega,\tilde X)$ is decoupled by ones in $G'$ (and the ones of this process coincide with the ones of $\tilde X$ since $\omega \ge \tilde X$), so that given $\tilde X$, we may sample $\omega$ using an additional independent i.i.d.\ source. Finally, $\omega^\infty$ is a simple finitary function of $(\omega,\tilde X)$, as $\omega^\infty_v=1$ if and only if $v$ is connected in $\omega$ to $\{v \in V : \tilde X_v=1\}$. This shows that $\omega^\infty$ is a finitary factor of an i.i.d.\ process as desired.

\begin{claim}
$(\omega,\tilde X)$ is decoupled by ones when viewed as a process on $G'$.
\end{claim}

The proof of the claim is somewhat technical.
It may be instructive to consider the case when $G$ is a regular tree so that we may take $r=0$ and $G'=G$. In this case, removing a vertex $v$ splits the tree into $\Delta$ branches, and if $\tilde X_v=1$ then $v$ is open and connected to infinity in each of these branches. It is not too hard to convince oneself that if $A$ is a finite connected set and $\tilde X_{\partial A} \equiv 1$, then the conditional law of $\omega_A$ given $(\omega,\tilde X)_{A^c}$ is simply Bernoulli percolation conditioned that each boundary vertex connects to at least one other boundary vertex. Moreover, for $v\in A$, one has $\tilde X_v=1$ if and only if $v$ connects to the boundary through each of its $\Delta$ neighbors. In particular, the conditional law of $(\omega,\tilde X)_A$ given $(\omega,\tilde X)_{A^c}$ is deterministic on the event that $\tilde X_{\partial A}\equiv 1$, showing that $(\omega,\tilde X)$ is decoupled by ones.
The proof below is an extension of these ideas.

\begin{proof}
Let $A^{+k} := \{ v \in V : \dist(v,A) \le k \}$ denote the $k$-neighborhood of $A$ in $G$.
Let $\partial' A := A^{+6r+1} \setminus A$ denote the boundary of a set $A \subset V$ in the graph $G'$.
We need to show that for any finite set $A \subset V$, $(\omega,\tilde X)_A$ and $(\omega,\tilde X)_{A^c}$ are conditionally independent given $(\omega,\tilde X)_{\partial' A} \equiv (1,1)$. It suffices to show this when $A$ is $G'$-connected. Fix such a set $A$. Note that $A^{+3r}$ is connected.

It suffices to show that $\omega_A$ and $(\omega,\tilde X)_{A^c}$ are conditionally independent given $\tilde X_{\partial' A} \equiv 1$, and that $\tilde X_A$ is a function of $\omega_A$ on this event.
It is straightforward that the conditional law of $\omega_A$ given $(\omega,\tilde X)_{A^c}$ is the law of a Bernoulli percolation $\tau$ of parameter $p$ on $A$ restricted to some set $E_A(\omega,\tilde X) \subset \{0,1\}^A$ which depends on $(\omega,\tilde X)$ only through its restriction to $A^c$. Specifically, $E_A(\omega,\tilde X)$ consists of those $\tau$ such that $\tilde X(\tau^\omega)_{A^c}=\tilde X_{A^c}$, where $\tau^\omega \in \{0,1\}^V$ is the configuration which agrees with $\tau$ on $A$ and with $\omega$ on $A^c$. We must show that $E_A(\omega,\tilde X)$ is equal to one particular set $E_A$ whenever $\tilde X_{\partial' A} \equiv 1$. Indeed, the set $E_A$ can be described explicitly as follows.

Let $\bar B_r(v)$ be the union of $B_r(v)$ with the finite connected components of $V \setminus B_r(v)$. Then $V \setminus \bar B_r(v)$ has at least two connected components and they are all infinite, and $\{\tilde X_v=1\}$ is the event that $B_r(v)$ is open and every $u \in \partial \bar B_r(v)$ is connected to infinity by an open path disjoint from $\bar B_r(v)$ (we indicate this by $u \leftrightarrow \infty$ off $B_r(v)$).

Denote $\partial^* A := \partial A^{+6r+1}  = \{ v \in V : \dist(v,A)=6r+1 \}$. For $\tau \in \{0,1\}^A$, we write $\tau^\1 \in \{0,1\}^V$ for the configuration which agrees with $\tau$ on $A$ and is all ones on $A^c$.
Let $E_A$ be the set of those $\tau \in \{0,1\}^A$ for which every vertex in $A \cap (\partial A)^{+r}$ is open, and for every $v \in A^{+4r} \setminus A$ and $u \in \partial \bar B_r(v)$ there is an open path in $\tau^\1$ from $u$ to $\partial^* A$ disjoint from $B_r(v)$ (we indicate this by $u \leftrightarrow \partial^* A$ off $B_r(v)$ in $\tau^\1$).
Note that $E_A$ is deterministic (it does not depend on $\omega$ or $\tilde X$).

Recall that $r$ was chosen so that $V \setminus B_r(v)$ has at least two infinite connected components for all $v$. The significance of this choice of $r$ is that it ensures that for any $v,w \in V$,
\begin{equation}\label{eq:inf-conn}
\dist(v,w)>2r\quad\text{and}\quad \tilde X_w=1 \qquad\implies\qquad w \leftrightarrow \infty\text{ off }B_r(v) .
\end{equation}
More generally, for any $S \subset V$ and $w \in V$,
\begin{equation}\label{eq:inf-conn2}
S \text{ connected},\quad \dist(S,w)>r, \quad \tilde X_w=1 \qquad\implies\qquad w \leftrightarrow \infty\text{ off }S .
\end{equation}
Indeed, since $S$ is a connected set disjoint from $B_r(w)$, it is contained in a single connected component of $V \setminus B_r(w)$. In particular, there is an infinite connected component $C$ of $V \setminus B_r(w)$ which is disjoint from $S$. If $\tilde X_w=1$, then taking any $u \in C \cap \partial B_r(w)$, we have that $u \leftrightarrow \infty$ in $C$, and in particular, $u \leftrightarrow \infty$ off $S$.
Since $w \leftrightarrow u$ in $B_r(w)$, we conclude that $w \leftrightarrow \infty$ off $S$.

Suppose now that $\tilde X_{\partial' A}\equiv 1$ and let us show that $E_A(\omega,\tilde X)=E_A$.
Let us first show that $E_A(\omega,\tilde X) \subset E_A$.
For this, it suffices to show that $\omega_A \in E_A$ (since $\omega_A$ can be any element in $E_A(\omega,\tilde X)$).
Clearly, every vertex in $A \cap (\partial A)^{+r}$ is open, since $\tilde X_{\partial A}\equiv 1$. Let $v \in A^{+4r} \setminus A$ and $u \in \partial\bar B_r(v)$. We need to show that $u \leftrightarrow \partial^* A$ off $B_r(v)$ in $\omega_A^\1$. This is clearly the case, as $\tilde X_v=1$ implies that $u \leftrightarrow \infty$ off $B_r(v)$ in $\omega$.

Let us now show that $E_A \subset E_A(\omega,\tilde X)$.
We need to show that $\tilde X(\tau^\omega)_{A^c}=\tilde X_{A^c}$ for $\tau \in E_A$. That is, we need to show that modifying $\omega$ to equal some $\tau \in E_A$ on $A$ does not change the values of $\tilde X$ outside $A$. Fix $\tau \in E_A$ and denote $\omega' := \tau^\omega$ and $\tilde X' := \tilde X(\omega')$.
Let us first show that $\tilde X'_{A^c} \ge \tilde X_{A^c}$.
Fix $v \in A^c$ such that $\tilde X_v=1$.
We need to show that $\tilde X'_v=1$.
Clearly, $B_r(v)$ is open in $\omega'$.
Let $u \in \partial\bar B_r(v)$. We need to show that $u \leftrightarrow \infty$ off $B_r(v)$ in $\omega'$.

Suppose first that $v \in A^{+4r}$.
Then $u \leftrightarrow \partial^*A$ off $B_r(v)$ in $\tau^\1$ by the definition of $E_A$. Consider a path witnessing this connection, i.e., an open path in $\tau^\1$ from $u$ to $\partial^*A$ disjoint from $B_r(v)$. Let $w$ be the first vertex this path reaches in $\partial^*A$, and note that $u \leftrightarrow w$ off $B_r(v)$ in $\omega'$.
Since $\tilde X_w=1$ and $\dist(w,A)=6r+1>6r$, \eqref{eq:inf-conn2} now implies that $w \leftrightarrow \infty$ off $A^{+5r}$ in $\omega$. Since $\omega$ and $\omega'$ agree outside of $A$, we have that $w \leftrightarrow \infty$ off $A^{+5r}$ also in $\omega'$. In particular, $u \leftrightarrow w \leftrightarrow \infty$ off $B_r(v)$ in $\omega'$. This shows that $\tilde X'_v=1$ when $v \in A^{+4r}$.

Suppose now that $v \notin A^{+4r}$.
Since $\tilde X_v=1$, we have that $u \leftrightarrow \infty$ off $B_r(v)$ in $\omega$.
Consider a path witnessing this connection, i.e., an open path in $\omega$ from $u$ to infinity disjoint from $B_r(v)$. If this path does not pass through $A$, then it also witnesses that $u \leftrightarrow \infty$ off $B_r(v)$ in $\omega'$. If it does pass through $A$, then let $w$ be the first vertex it reaches in $\partial A$, and note that $u \leftrightarrow w$ off $B_r(v)$ in $\omega'$. Since $\dist(v,w) \ge \max\{4r,1\} > 2r$ and since we have already seen that $\tilde X'_w=1$, \eqref{eq:inf-conn} now implies that $w \leftrightarrow \infty$ off $B_r(v)$ in $\omega'$. Hence, $u \leftrightarrow w \leftrightarrow \infty$ off $B_r(v)$ in $\omega'$. This shows that $\tilde X'_v=1$ when $v \notin A^{+4r}$. 

This completes the proof that $\tilde X'_{A^c} \ge \tilde X_{A^c}$. The converse $\tilde X'_{A^c} \le \tilde X_{A^c}$ follows similarly, noting that $\tilde X_{\partial' A} \equiv 1$ so that the case when $v \in A^{+4r}$ does not require any verification. The argument for $v \notin A^{+4r}$ follows the exact same argument as in the preceding paragraph.
Thus, $\tilde X'_{A^c} = \tilde X_{A^c}$, showing that $E_A \subset E_A(\omega,\tilde X)$.
This completes the proof that $E_A=E_A(\omega,\tilde X)$.

It remains to show that $\tilde X_A$ is a function of $\omega_A$ on the event that $\tilde X_{\partial'A}\equiv1$. Suppose that $\tilde X_{\partial'A}\equiv1$, and note that this implies that $\partial'A$ is open. Let $v \in A$.
The question of whether $B_r(v)$ is open is clearly a function of $\omega_A$.
Suppose that $B_r(v)$ is open. We claim that $\tilde X_v=1$ if and only if $u \leftrightarrow \partial^*A$ off $B_r(v)$ for every $u \in \partial\bar B_r(v)$, which is clearly determined by $\omega_A$. The `only if' direction is immediate from the definition of $\tilde X$. To see the `if' direction, let $u \in \partial\bar B_r(v)$ and suppose that $u \leftrightarrow w \in \partial^*A$ off $B_r(v)$.
Since $\tilde X_w=1$ and $\dist(w,A) = 6r+1>4r$, \eqref{eq:inf-conn2} now implies that $w \leftrightarrow \infty$ off $A^{+3r}$. In particular, $u \leftrightarrow w \leftrightarrow \infty$ off $B_r(v)$.
This shows that $\tilde X_v$ is a function of $\omega_A$ on the event that $\tilde X_{\partial'A}\equiv1$.
\end{proof}

\section{Open problems} \label{sec:open}

We discuss some open problems below. For simplicity, we consider only transitive graphs, but many of the questions make sense for quasi-transitive graphs or even bounded-degree graphs.

\subsection{Infinite clusters of Bernoulli percolation}

For Bernoulli percolation (with parameter $p$) on a nonamenable bounded-degree graph, we have shown that the infinite clusters invariantly dominate a high-density \iid\ process when $p$ is close to 1, that is, $p_\inv(\omega^\infty)\to1$ as $p\to1$. One may ask what happens just above $p_c$:

\begin{question}
On a nonamenable transitive graph,
\begin{itemize}
 \item Is $p(\omega^\infty)$ positive for all $p>p_c$? 
 \item Is $p_\inv(\omega^\infty)$ positive for all $p>p_c$?
\end{itemize}
\end{question}

We are able to answer the first part of the question in the special case of regular trees.

\begin{thm}
On a regular tree $\T_d$ of degree $d \ge 3$, we have %Let $\omega^\infty$ be as in \cref{thm:stochastic-dom-perc}.
$p(\omega^\infty)>0$ for all $p>p_c(\T_d)=\frac1{d-1}$.
\end{thm}
\begin{proof}
Given any invariant process $X$ on $\T_d$, it is not hard to see that
\[ p(X) \ge \essinf \E\left[X_v \mid (X_u)_{u \in P_v}\right] ,\]
where $v$ is any vertex and $P_v$ is any connected component of $\T_d \setminus v$. The reason for this inequality is that a monotone coupling can be constructed sequentially (note however that this does not produce an invariant coupling).

Denote $X:=\omega^\infty$. We use the above in order to show that $p(X)>0$.
Thus, we condition on $(X_u)_{u \in P_v}$.
Let $w$ be the unique vertex in $P_v$ which is adjacent to $v$.
Suppose first that $X_w=1$. Let us further condition on $(\omega_u)_{u \in P_v}$.
If $w$ is not in an infinite cluster of $\omega$ in $P_v$, then it must be the case that $X_v=1$.
If $w$ is in an infinite cluster of $\omega$ in $P_v$, then the conditional law of $\omega_{\T_d \setminus P_v}$ is the same as its unconditional law (Bernoulli percolation with parameter $p$), and in particular, the conditional probability that $X_v=1$ is simply $p$.
Suppose now that $X_w=0$. Then the conditional probability that $\omega_w=0$ is at least $1-p$ (this is true even when conditioning on $X$ and $\omega_{\T_d \setminus \{w\}}$). Given that $\omega_w=0$, the conditional law of $\omega_{\T_d \setminus P_v}$ is again the same as its unconditional law.
Since the $d$-regular tree and the $d$-ary tree have the same critical value $p_c=\frac1{d-1}$, the conditional probability that $X_v=1$ is lower bounded. This shows that $p(X)>0$ as required.
\end{proof}

Regarding finitary factors, we have shown that $X$ is a finitary factor of an i.i.d.\ process on most nonamenable quasi-transitive graphs when $p$ is close to 1. Is this the case for any $p>p_c$? We do not know the answer even on a regular tree.

\begin{question}
On a regular tree, is $\omega^\infty$ a finitary factor of an \iid\ process for all $p>p_c$?
\end{question}

\subsection{Ising}

For an infinite connected bounded-degree graph, we have shown that $p(\mu^+_\beta)$ tends to 0 or 1 as $\beta\to1$ according whether the graph is amenable or not. In particular, on a nonamenable transitive graph, $p(\mu^+_\beta)>p(\mu^-_\beta)$ for large $\beta$. One may ask whether this holds all the way down to criticality, and to what extent this fails for amenable graphs:

\begin{question}\,
\begin{itemize}
 \item On a nonamenable transitive graph, is $p(\mu^+_\beta)>p(\mu^-_\beta)$ for all $\beta>\beta_c$?
 \item On an amenable transitive graph, is $p(\mu^+_\beta)=p(\mu^-_\beta)$ for large $\beta$ (or perhaps even all $\beta$)?
\end{itemize}
\end{question}

For $\Z^d$ and regular trees, these questions were answered by Liggett and Steif~\cite{liggett2006stochastic}.
For regular trees this is a consequence of an explicit expression for $p(\mu^\pm_\beta)$. For $\Z^d$ this is the consequence of a formula for $p(\mu^\pm_\beta)$ in terms of the probabilities of boxes being all minus (this formula holds for any downward FKG measure).
In \cite[Question~7]{liggett2006stochastic}, it is asked whether amenability for transitive graphs is characterized by the property that $p(\mu^+_\beta)=p(\mu^-_\beta)$ or alternatively by the property that the plus states at different temperatures are not stochastically comparable. While we have fully established the second characterization (in the larger class of bounded-degree graphs), the first characterization is only partial at the moment.

On $\Z^d$, it is known that $p(\mu^+_\beta)$ is strictly decreasing for all $\beta$, whereas on a regular tree, it is increasing for all $\beta>\beta_c$~\cite{liggett2006stochastic}. As before, these results rely on special properties of the graphs, and one may ask whether this holds more generally.

\begin{question}\,
\begin{itemize}
 \item On a nonamenable transitive graph, is $p(\mu^+_\beta)$ non-decreasing for all $\beta>\beta_c$?
 \item On an amenable transitive graph, is $p(\mu^+_\beta)$ non-increasing for all $\beta$?
\end{itemize}
\end{question}

For sufficiently low temperatures, we have shown that the plus states are not only stochastically ordered, but that there is invariant domination as well. We do not know whether this holds throughout the entire low temperature regime, even in the special case of regular trees.

\begin{question}
On a regular tree, does $\mu_{\beta_1}^+$ invariantly dominate $\mu_{\beta_2}^+$ for all $\beta_1 >\beta_2 > \beta_c$? What about on other nonamenable graphs?
\end{question}

Regarding finitary factors, we have shown that $\mu^+_\beta$ is a finitary factor of an i.i.d.\ process on any quasi-transitive nonamenable graph when the temperature is sufficiently low. A natural question is whether this extends all the way down to the critical temperature, and it is reasonable to first attempt to answer this for regular trees.

\begin{question}\label{thm:main-tree}
On a regular tree, is $\mu^+_\beta$ a finitary factor of an \iid\ process for all $\beta>\beta_c$?
\end{question}

\subsection{More on invariant domination}

Suppose $G$ is transitive and $X$ is an invariant $\{0,1\}$-valued process.
When $G$ is amenable, we always have $p_\inv(X)=p(X)$.
When $G$ is nonamenable, Mester~\cite{mester2013invariant} showed that it is possible (for a certain $G$) that $X$ stochastically dominates an invariant process $Y$, but does not invariantly dominate it. In Mester's counterexample, both processes $X$ and $Y$ are somewhat artificial (though they have nice properties such as finite dependence and uniform finite energy). In particular, neither $X$ nor $Y$ is an \iid\ process, and it is still undetermined whether $p_\inv(X)$ can be strictly less than $p(X)$. This raises the question of whether $\omega^\infty$ and $\mu^+_\beta$ can serve as counterexamples or not:

\begin{question}
On a nonamenable transitive graph (perhaps even a regular tree),
\begin{itemize}
 \item Is $p_\inv(\omega^\infty)=p(\omega^\infty)$ for all $p$?
 \item Is $p_\inv(\mu^+_\beta)=p(\mu^+_\beta)$ for all $\beta$?
\end{itemize}
\end{question}

Recall that $p(X) \ge p_*(X)$ for any invariant process $X$.
In \cref{lem:invariant-monotone-coupling} we showed that $p_\inv(X) \ge p_*(X)$ when $X$ satisfies a certain decoupling condition. One may ask whether the latter assumption can be dropped. More generally, $X$ stochastically dominates $Y$ whenever $X \succeq_* Y$, and one may ask whether $X$ invariantly dominates $Y$ as well, with no further assumptions.

\begin{question}
On a nonamenable transitive graph,
\begin{itemize}
 \item Is $p_\inv(X) \ge p_*(X)$ whenever $X$ is an invariant process?
 \item Does $X$ invariantly dominate $Y$ whenever $X$ and $Y$ are invariant processes such that $X \succeq_* Y$?
\end{itemize}
\end{question}

\bibliographystyle{abbrv}
\bibliography{library}

\begin{thebibliography}{10}

\bibitem{AL07}
D.~Aldous and R.~Lyons.
\newblock Processes on unimodular random networks.
\newblock {\em Electron. J. Probab.}, 12:no. 54, 1454--1508, 2007.

\bibitem{dichotomy15}
O.~Angel, T.~Hutchcroft, A.~Nachmias, and G.~Ray.
\newblock Hyperbolic and parabolic unimodular random maps.
\newblock {\em Geom. Funct. Anal.}, 28(4):879--942, 2018.

\bibitem{benjamini1999group}
I.~Benjamini, R.~Lyons, Y.~Peres, and O.~Schramm.
\newblock Group-invariant percolation on graphs.
\newblock {\em Geometric \& Functional Analysis GAFA}, 9(1):29--66, 1999.

\bibitem{benjamini2001special}
I.~Benjamini, R.~Lyons, Y.~Peres, and O.~Schramm.
\newblock Uniform spanning forests.
\newblock {\em Annals of probability}, pages 1--65, 2001.

\bibitem{van1999existence}
J.~V.~D. Berg and J.~E. Steif.
\newblock On the existence and nonexistence of finitary codings for a class of
  random fields.
\newblock {\em Annals of probability}, pages 1501--1522, 1999.

\bibitem{bhandari2020improved}
S.~Bhandari and S.~Chakraborty.
\newblock Improved bounds for perfect sampling of k-colorings in graphs.
\newblock In {\em Proceedings of the 52nd Annual ACM SIGACT Symposium on Theory
  of Computing}, pages 631--642, 2020.

\bibitem{bissacot2011improvement}
R.~Bissacot, R.~Fern{\'a}ndez, A.~Procacci, and B.~Scoppola.
\newblock An improvement of the {L}ov{\'a}sz local lemma via cluster expansion.
\newblock {\em Combinatorics, Probability and Computing}, 20(5):709--719, 2011.

\bibitem{bowen2004couplings}
L.~Bowen.
\newblock Couplings of uniform spanning forests.
\newblock {\em Proceedings of the American Mathematical Society},
  132(7):2151--2158, 2004.

\bibitem{broman2006refinements}
E.~I. Broman, O.~H{\"a}ggstr{\"o}m, and J.~E. Steif.
\newblock Refinements of stochastic domination.
\newblock {\em Probability theory and related fields}, 136:587--603, 2006.

\bibitem{de2012developments}
E.~De~Santis and A.~Lissandrelli.
\newblock Developments in perfect simulation of {G}ibbs measures through a new
  result for the extinction of {G}alton-{W}atson-like processes.
\newblock {\em Journal of Statistical Physics}, 147:231--251, 2012.

\bibitem{dobrushin1996estimates}
R.~L. Dobrushin.
\newblock Estimates of semi-invariants for the {I}sing model at low
  temperatures.
\newblock {\em Translations of the American Mathematical Society-Series 2},
  177:59--82, 1996.

\bibitem{dobrushin1996perturbation}
R.~L. Dobrushin.
\newblock Perturbation methods of the theory of {G}ibbsian fields.
\newblock {\em Lectures on probability theory and statistics}, pages 1--66,
  1996.

\bibitem{ferrari2002perfect}
P.~A. Ferrari, R.~Fern{\'a}ndez, and N.~L. Garcia.
\newblock Perfect simulation for interacting point processes, loss networks and
  {I}sing models.
\newblock {\em Stochastic Processes and their Applications}, 102(1):63--88,
  2002.

\bibitem{galves2008perfect}
A.~Galves, N.~Garcia, and E.~L{\"o}cherbach.
\newblock Perfect simulation and finitary coding for multicolor systems with
  interactions of infinite range.
\newblock {\em arXiv preprint arXiv:0809.3494}, 2008.

\bibitem{galves2010perfect}
A.~Galves, E.~L{\"o}cherbach, and E.~Orlandi.
\newblock Perfect simulation of infinite range {G}ibbs measures and coupling
  with their finite range approximations.
\newblock {\em Journal of Statistical Physics}, 138:476--495, 2010.

\bibitem{georgii2011gibbs}
H.-O. Georgii.
\newblock Gibbs measures and phase transitions.
\newblock In {\em Gibbs Measures and Phase Transitions}. de Gruyter, 2011.

\bibitem{georgii2001random}
H.-O. Georgii, O.~H{\"a}ggstr{\"o}m, and C.~Maes.
\newblock The random geometry of equilibrium phases.
\newblock In {\em Phase transitions and critical phenomena}, volume~18, pages
  1--142. Elsevier, 2001.

\bibitem{grimmett2006random}
G.~R. Grimmett.
\newblock {\em The random-cluster model}, volume 333.
\newblock Springer Science \& Business Media, 2006.

\bibitem{haggstrom1999exact}
O.~Haggstrom and K.~Nelander.
\newblock On exact simulation of {M}arkov random fields using coupling from the
  past.
\newblock {\em Scandinavian Journal of Statistics}, 26(3):395--411, 1999.

\bibitem{haggstrom2000ising}
O.~H{{\"a}}ggstr{{\"o}}m, R.~H. Schonmann, and J.~E. Steif.
\newblock The {I}sing model on diluted graphs and strong amenability.
\newblock {\em Annals of probability}, pages 1111--1137, 2000.

\bibitem{haggstrom2000propp}
O.~H{\"a}ggstr{\"o}m and J.~E. Steif.
\newblock Propp--{W}ilson algorithms and finitary codings for high noise
  {M}arkov random fields.
\newblock {\em Combinatorics, Probability and Computing}, 9(5):425--439, 2000.

\bibitem{harel2018finitary}
M.~Harel and Y.~Spinka.
\newblock Finitary codings for the random-cluster model and other
  infinite-range monotone models.
\newblock {\em Electronic Journal of Probability}, 27:1--32, 2022.

\bibitem{he2021perfect}
K.~He, X.~Sun, and K.~Wu.
\newblock Perfect sampling for (atomic) {L}ov\'asz local lemma.
\newblock {\em arXiv preprint arXiv:2107.03932}, 2021.

\bibitem{HeSc}
Z.-X. He and O.~Schramm.
\newblock Hyperbolic and parabolic packings.
\newblock {\em Discrete Comput. Geom.}, 14(2):123--149, 1995.

\bibitem{huber1998exact}
M.~Huber.
\newblock Exact sampling and approximate counting techniques.
\newblock In {\em S{TOC} '98 ({D}allas, {TX})}, pages 31--40. ACM, New York,
  1999.

\bibitem{huber2004perfect}
M.~Huber.
\newblock Perfect sampling using bounding chains.
\newblock {\em Ann. Appl. Probab.}, 14(2):734--753, 2004.

\bibitem{jonasson1999random}
J.~Jonasson.
\newblock The random cluster model on a general graph and a phase transition
  characterization of nonamenability.
\newblock {\em Stochastic Processes and their Applications}, 79(2):335--354,
  1999.

\bibitem{jonasson1999amenability}
J.~Jonasson and J.~E. Steif.
\newblock Amenability and phase transition in the {I}sing model.
\newblock {\em Journal of Theoretical Probability}, 12(2):549--559, 1999.

\bibitem{kesten1959full}
H.~Kesten.
\newblock Full {B}anach mean values on countable groups.
\newblock {\em Mathematica Scandinavica}, pages 146--156, 1959.

\bibitem{kesten1959symmetric}
H.~Kesten.
\newblock Symmetric random walks on groups.
\newblock {\em Transactions of the American Mathematical Society},
  92(2):336--354, 1959.

\bibitem{liggett1985interacting}
T.~M. Liggett.
\newblock {\em Interacting particle systems}, volume~2.
\newblock Springer, 1985.

\bibitem{liggett1997domination}
T.~M. Liggett, R.~H. Schonmann, and A.~M. Stacey.
\newblock Domination by product measures.
\newblock {\em The Annals of Probability}, 25(1):71--95, 1997.

\bibitem{liggett2006stochastic}
T.~M. Liggett and J.~E. Steif.
\newblock Stochastic domination: the contact process, {I}sing models and {FKG}
  measures.
\newblock In {\em Annales de l'IHP Probabilit{\'e}s et statistiques},
  volume~42, pages 223--243, 2006.

\bibitem{lyons2000phase}
R.~Lyons.
\newblock Phase transitions on nonamenable graphs.
\newblock {\em Journal of Mathematical Physics}, 41(3):1099--1126, 2000.

\bibitem{lyons2017probability}
R.~Lyons and Y.~Peres.
\newblock {\em Probability on trees and networks}, volume~42.
\newblock Cambridge University Press, 2017.

\bibitem{lyons2016invariant}
R.~Lyons and A.~Thom.
\newblock Invariant coupling of determinantal measures on sofic groups.
\newblock {\em Ergodic Theory and Dynamical Systems}, 36(2):574--607, 2016.

\bibitem{mester2013invariant}
P.~Mester.
\newblock Invariant monotone coupling need not exist.
\newblock {\em The Annals of Probability}, 41(3A):1180--1190, 2013.

\bibitem{Moh91}
B.~Mohar.
\newblock Some relations between analytic and geometric properties of infinite
  graphs.
\newblock {\em Discrete mathematics}, 95(1-3):193--219, 1991.

\bibitem{schonmann1999stability}
R.~H. Schonmann.
\newblock Stability of infinite clusters in supercritical percolation.
\newblock {\em Probability Theory and Related Fields}, 113(2):287--300, 1999.

\bibitem{scott2005repulsive}
A.~D. Scott and A.~D. Sokal.
\newblock The repulsive lattice gas, the independent-set polynomial, and the
  {L}ov{\'a}sz local lemma.
\newblock {\em Journal of Statistical Physics}, 118(5):1151--1261, 2005.

\bibitem{scott2006dependency}
A.~D. Scott and A.~D. Sokal.
\newblock On dependency graphs and the lattice gas.
\newblock {\em Combinatorics, Probability and Computing}, 15(1-2):253--279,
  2006.

\bibitem{shearer1985problem}
J.~B. Shearer.
\newblock On a problem of {S}pencer.
\newblock {\em Combinatorica}, 5(3):241--245, 1985.

\bibitem{spinka2018finitaryising}
Y.~Spinka.
\newblock Finitary coding for the sub-critical {I}sing model with finite
  expected coding volume.
\newblock {\em Electronic Journal of Probability}, 25:1--27, 2020.

\bibitem{spinka2018finitarymrf}
Y.~Spinka.
\newblock Finitary codings for spatial mixing {M}arkov random fields.
\newblock {\em The Annals of Probability}, 48(3):1557--1591, 2020.

\bibitem{temmel2014shearer}
C.~Temmel.
\newblock Shearer’s measure and stochastic domination of product measures.
\newblock {\em Journal of Theoretical Probability}, 27(1):22--40, 2014.

\end{thebibliography}
\end{document}